\documentclass[twoside,12pt]{article}
\usepackage{amsmath,amssymb,amsthm,amsfonts}
\usepackage{paralist}
\usepackage{subfigure}
\usepackage[title]{appendix}

\UseRawInputEncoding 

\usepackage{graphics} 
\usepackage{epsfig} 
\usepackage{graphicx}
\usepackage{caption}
\usepackage{epstopdf}
\usepackage[colorlinks=true]{hyperref}
\usepackage{float}
\usepackage{pdfpages}
\usepackage{amsfonts}
\usepackage{amsmath,amsthm}
\usepackage{multirow}
\usepackage[T1]{fontenc}
\usepackage{lmodern}

\usepackage{amsmath}
\usepackage{txfonts}
\usepackage{enumerate}
\usepackage[numbers,sort&compress]{natbib}
\usepackage[font=scriptsize,labelfont=bf,width=0.95\textwidth]{caption}

\numberwithin{equation}{section}
\newcommand{\R}{{\mathbb R}}

\newcommand{\be}{\begin{equation}}
\newcommand{\ee}{\end{equation}}

\newcommand{\ben}{\begin{eqnarray*}}
\newcommand{\enn}{\end{eqnarray*}}

\newtheorem{theorem}{\textbf Theorem}[section]
\newtheorem{lemma}{\textbf Lemma}[section]

 \numberwithin{equation}{section}
\newtheorem{remark}{Remark}[section]
\renewcommand{\theequation}{\arabic{section}.\arabic{equation}}

\usepackage{microtype}

\begin{document}

\title{\textbf{
 Blow-up Behaviors of Ground States in Ergodic Mean-field Games Systems with Hartree-type Coupling}}
\author{
Fanze Kong \thanks{Department of Applied Mathematics, University of Washington, Seattle, WA 98195, USA; fzkong@uw.edu}, 
Yonghui Tong\thanks{Center for Mathematical Sciences, Wuhan University of Technology, Wuhan 430070, China; myyhtong@whut.edu.cn},
Xiaoyu Zeng\thanks{Center for Mathematical Sciences, Wuhan University of Technology, Wuhan 430070, China; xyzeng@whut.edu.cn}
 and Huan-Song Zhou\thanks{Center for Mathematical Sciences, Wuhan University of Technology, Wuhan 430070, China; hszhou@whut.edu.cn}
}
\date{\today}
\maketitle
\abstract{
In this paper, we investigate the concentration behaviors of ground states to stationary Mean-field Games systems (MFGs) with the nonlocal coupling in $\mathbb R^n$, $n\geq 2.$  With the mass critical exponent imposed on Riesz potentials, we first discuss the existence of ground states to potential-free MFGs, which corresponds to the establishment of Gagliardo-Nirenberg type's inequality.  Next, with the aid of the optimal inequality, we classify the existence of ground states to stationary MFGs with Hartree-type coupling in terms of the $L^1$-norm of population density defined by $M$.  In addition, under certain types of coercive potentials, the asymptotics of ground states to ergodic MFGs with the nonlocal coupling are captured.  Moreover, if the local polynomial expansions are imposed on potentials, we study the refined asymptotic behaviors of ground states and show that they concentrate on the flattest minima of potentials.



\maketitle



\section{Introduction}\label{intro1}
In this paper, we are concerned with the following ergodic stationary Mean-field Games systems 
  \begin{align}\label{goalmodel}
\left\{\begin{array}{ll}
-\Delta u+H(\nabla u)+\lambda=V(x)- K_{\alpha}\ast m,&x\in\mathbb R^n,\\
\Delta m+\nabla\cdot (m\nabla H(\nabla u))=0,&x\in\mathbb R^n,\\
 \int_{\mathbb R^n}m\,dx=M>0,
\end{array}
\right.
\end{align}
where $(m,u,\lambda)$ denotes a solution, $\lambda$ is a so-called Lagrange multiplier, $V$ is the potential function and $K_{\alpha}$ is defined as the Riesz potential satisfying 
\begin{align}\label{MFG-K}
K=\frac{1}{|x|^{n-\alpha}}\text{ with }0<\alpha<n. 
\end{align}
Here $m$ represents the population density and $u$ is the value function of a typical player.  In particular, Hamiltonian $H:\mathbb R^n\rightarrow \mathbb R$ is in general assumed to be convex uniformly and the typical form is
\begin{align}\label{MFG-H}
H(\boldsymbol{p}):={C_H}\vert \boldsymbol{p} \vert^{\gamma},~~\exists \gamma>1,~C_H>0.
\end{align}
Correspondingly, the Lagrangian is defined by $L(\boldsymbol{q}):=\sup_{\boldsymbol{p}\in\mathbb R^n }[\boldsymbol{q}\cdot \boldsymbol{p}-H(\boldsymbol{p})]$ and if $H$ is given by \eqref{MFG-H}, $L$ can be written as  
\begin{align}\label{MFG-L}
L(\boldsymbol{q})={C_L}\vert \boldsymbol{q}\vert^{\gamma'},~~\gamma'=\frac{\gamma}{\gamma-1}>1~\text{and}~C_L=\frac{1}{\gamma'}(\gamma C_H)^{\frac{1}{1-\gamma}}>0,
\end{align}
where $\gamma'$ is the conjugate number of $\gamma.$

Assume $H$ in system (\ref{goalmodel}) is given by (\ref{MFG-H}) and $V(x)$ has polynomial lower and upper bounds when $|x|$ is large enough, then Cesaroni and Bernardini \cite{bernardini2023ergodic,bernardini2022mass} studied the existence and concentration of ground states to (\ref{goalmodel}) under the subcritical mass exponent case by using the variational method.  Motivated by their results and our analysis focused on Mean-field Games systems with the local coupling \cite{cirant2024critical}, we shall utilize the variational approach to discuss the existence and asymptotic behaviors of ground states to (\ref{goalmodel}) under the \textit{critical mass exponent} case, i.e. $\alpha=\alpha^*:= n-\gamma'$ in \eqref{MFG-K}.
\subsection{Mean-field Games Theory and Systems}

Motivated by the theories of statistical physics, Huang et al. \cite{Huang} and Lasry et al. \cite{Lasry} in 2007 developed Mean-field Games theories and proposed a class of coupled PDE systems to describe the differential games among a huge number of players, which have rich applications in the fields of economics, finance and management.  

The general form of time-dependent Mean-field Games systems reads as
 \begin{equation}\label{MFG-time}
\left\{
\begin{array}{ll}
u_t= -\Delta u+H(\nabla u)-V(x)-f(m), &x \in \mathbb R^n,t>0,\\
m_t=\Delta m +\nabla\cdot (\nabla H(\nabla u)m),&x \in \mathbb R^n, t>0,\\
u\vert_{t=T}=u_T, m|_{t=0}=m_0,&x\in \mathbb R^n,
\end{array}
\right.
\end{equation}
where $m$ and $u$ denote the density and the value function, respectively.  Here $m_0$ represents the initial data of density and $u_T$ is the terminal data of the value function.  Now, we give a brief summary of the derivation of (\ref{MFG-time}).   Suppose the dynamics of the $i$-th player satisfies
\begin{align}\label{game-process-dXti}
dX_t^i=-\nu^i_t dt+\sqrt{2}dB_t^i, \ \ X_0^i=x^i\in\mathbb R^n,~i=1,\cdots,N,
\end{align}
where $x^i$ is the initial condition, $\nu^i_t$ is the velocity and $B_t^i$ represents the Brownian motion.  Assume $B_t^i$ for $i=1,\cdots,N$ are independent and all players are homogeneous, then we have $X_t^i$ for $i=1,\cdots,N$ follow the same process and drop ``$i$" in \eqref{game-process-dXti}.  On the other hand, each player aims to minimize the following expected cost:
\begin{align}\label{longsenseexpectation}
J(\gamma_t):=\mathbb E\int_0^T[L(\gamma_t)+V(X_t)+f(m(X_t))] dt + u_T(X_T),
\end{align}
where $L$ is the Lagrangian, $V$ measures the spatial preference and $f$ is the coupling. Invoking the dynamic programming principle \cite{BF1984,BF1995}, one can formulate the time-dependent system (\ref{MFG-time}) by analyzing the minimization of (\ref{longsenseexpectation}). We point out that many results are concentrated on the study of global well-posedness to (\ref{MFG-time}), see \cite{Car12,Car13,CGMT13,cirantgoffi2021,GPM12,GPM13,GPV13}. 


As stated in \cite{cirant2024critical}, the corresponding stationary problem of (\ref{MFG-time}) is
 \begin{equation}\label{MFG-SS}
\left\{
\begin{array}{ll}
-\Delta u+H(\nabla u)+\lambda=f(m)+V(x) , &x \in \mathbb R^n,     \\
 \Delta m+\nabla\cdot (m\nabla H(\nabla u))=0,&x \in \mathbb R^n,\\
\int_{\mathbb R^N} mdx=M>0,
\end{array}
\right.
\end{equation}
where the triple $(m,u,\lambda)$ denotes the solution, $V$ is the potential function and $f$ is the cost function.
There are also some results concerning the existence and qualitative properties of non-trivial solutions to the stationary problem (\ref{MFG-SS}), see \cite{cesaroni2018concentration,GM15,gomes2016regularity,meszaros2015variational,cirant2016stationary,cirant2024critical,bernardini2023ergodic,bernardini2022mass}.  We mention that when the cost $f$ is monotone increasing, as shown in \cite{Lasry}, the uniqueness of the solution to (\ref{MFG-SS}) can be in general guaranteed.  Whereas, when the cost $f$ is monotone decreasing and unbounded, the case is delicate and (\ref{MFG-SS}) may admit many distinct solutions.  In particular, the pioneering work in the study of ground states to stationary Mean-field Games systems with decreasing cost was finished by Cesaroni and Cirant \cite{cesaroni2018concentration}.

We also would like to point out the stationary Mean-field Games systems can be trivialized to nonlinear $\gamma'$-Laplacian Schr\"{o}dinger equations when $H$ is chosen as (\ref{MFG-H}).  Indeed, Fokker-Planck equation in (\ref{MFG-SS}) can be reduced into the following form:
\begin{align}\label{FPeqpartially}
\nabla m+mC_H|\nabla u|^{\gamma-2}\nabla u=0~~\text{a.e.,}~~x\in\mathbb R^n.
\end{align}
Similarly as shown \cite{cirant2015generalization}, we define $v:=m^{\frac{1}{\gamma'}}$ and obtain from (\ref{FPeqpartially}) and the $u$-equation in (\ref{MFG-SS}) that
\begin{align}\label{nonlinear-Schrodinger}
\left\{\begin{array}{ll}
-\mu\Delta_{\gamma'} v+[f(v^{\gamma'})+V(x)-\lambda]v^{\gamma'-1}=0,~x\in\mathbb R^n,\\
\int_{\mathbb R^n} v^{\gamma'}\,dx=M,~v>0,~\mu=\big(\frac{\gamma'}{C_H}\big)^{\gamma'-1},
\end{array}
\right.
\end{align}
where $\Delta_{\gamma'}$ is the $\gamma'$-Laplacian and given by $\Delta_{\gamma'}v=\nabla\cdot(|\nabla v|^{\gamma'-2}\nabla v)$. It is well-known that nonlinear $\gamma'$-Laplacian Schr\"{o}dinger equation (\ref{nonlinear-Schrodinger}) admits the following variational structures:
\begin{align}\label{variation-schrodinger}
\mathcal F(v):=\int_{\mathbb R^n}\bigg[\frac{\mu}{\gamma'}|\nabla v|^{\gamma'}+F(v)+\frac{1}{\gamma'}V(x)v^{\gamma'}\bigg]\, dx,
\end{align}
    where $F(v)$ denotes the anti-derivative of $f(v^{\gamma'})v^{\gamma'-1}.$  In particular, when $\gamma'=2$ and  $f(v^2)=-K_{\alpha}\ast v^{2}$ in (\ref{nonlinear-Schrodinger}), the equation is the standard nonlinear Schr\"{o}dinger equation with the Hartree-type aggregation term. 




Inspired by the relation between Schr\"{o}dinger equations and Mean-field Games systems discussed above, furthermore, the results of Cirant et al. \cite{cirant2024critical} and Bernardini et al. \cite{bernardini2022mass,bernardini2023ergodic}, we focus on the existence and asymptotic behaviors of ground states to \eqref{goalmodel} when $\alpha=n-\gamma'.$  In particular, Bernardini and Cesaroni studied the subcritical mass exponent case with $\alpha\in(n-\gamma',n)$ extensively via the variational method.  It is well-known that system (\ref{goalmodel}) admits the following variational structure:
\begin{align}\label{energy-dual}
\mathcal E(m,w):=
\int_{\mathbb R^n} \left[mL\bigg(-\frac{w}{m}\bigg)+V(x)m+F(m)\right]\, dx,
\end{align}
where $F(m):=-\frac{1}{2}(K_{\alpha}\ast m)m$ for $m\geq 0$ and $F(m)=0$ for $m\leq 0.$  Here Lagrangian $L$ is defined by
\begin{align}\label{general-Lagrangian}
L\bigg(-\frac{w}{m}\bigg):=\left\{\begin{array}{ll}
\sup\limits_{p\in\mathbb R^n}\big(-\frac{p\cdot w}{m}-H(p)\big),&m>0,\\
0,&(m,w)=(0,0),\\
+\infty,&\text{otherwise}.
\end{array}
\right.
\end{align}
To explain the range of exponent $\alpha$, we are concerned with the following constrained minimization problem:
\begin{align}\label{ealphaM-117}
e_{\alpha,M}:=\inf_{(m,w)\in \mathcal K_{M}}\mathcal E(m,w),
\end{align}
where the admissible set $\mathcal K_{M}$ is given by
\begin{align}\label{constraint-set-K}
\mathcal K_{ M}:=\Big\{&(m,w)\in (L^1(\mathbb R^n)\cap W^{1,\hat q}(\mathbb R^n))\times L^{1}(\mathbb R^n)\nonumber\\
~&\text{s. t. }\int_{\mathbb R^n}\nabla m\cdot\nabla\varphi\,dx=\int_{\mathbb R^n}w\cdot\nabla\varphi\, dx,\forall \varphi\in C_c^{\infty}(\mathbb R^n),\nonumber\\
~&\int_{\mathbb R^n} V(x)m\,dx<+\infty,~
\int_{\mathbb R^n}m\,dx=M>0,~m\geq 0\text{~a.e.~}\Big\},
\end{align}
with
\begin{equation}\label{hatqconstraint}
\hat q:=
\frac{n}{n-\gamma'+1}, \text{ for each } \gamma'<n.
\end{equation}
It is straightforward to show that $e_{\alpha,M}<+\infty.$  Indeed, by choosing $(m_s,w_s)=\Big(ce^{-|x|},-\frac{xe^{-|x|}}{|x|}\Big)$ with $c$ determined by $\int_{\mathbb R^n}m\,dx=M$ and $w_s=\nabla m_s$, then one has $(m_s,w_s)\in\mathcal K_{\alpha,M}$ and $\mathcal E(m_s,w_s)<+\infty,$ which implies $e_{\alpha,M}<+\infty.$  Now, we mention that the lower bound $\alpha>n-\gamma'$ is a necessary condition to guarantee that $e_{\alpha,M}>-\infty$ for all $M>0$.  To clarify this, we find if $\alpha<n-\gamma',$ for any $(\bar m,\bar w)\in \mathcal K_{\alpha, M},$
\begin{align*}
\mathcal E(\bar m_{\delta},\bar w_{\delta})\rightarrow -\infty\text{ as }\delta\rightarrow 0^+,
\end{align*}
 where $(\bar m_{\delta},\bar w_{\delta})$ is defined as $(\bar m_{\delta},\bar w_{\delta}):=(\delta^{-n}\bar m(\delta^{-1}x),\delta^{-(n+1)}\bar w(\delta^{-1}x))\in \mathcal K_{\alpha,M}$ and $\delta$ is chosen such that $\delta^{-n}\int_{\mathbb R^n}\bar m(\delta^{-1}x)\, dx\equiv M.$  Based on the discussion stated above, Bernardini and Cesaroni employed the direct method and the concentration-compactness approach to investigate the ground states to (\ref{goalmodel}) with $H$ given by (\ref{MFG-H}) when $\alpha$ satisfies $n-\gamma'<\alpha<n$ in \eqref{MFG-K}.  In this paper,  similarly as the work finished in \cite{cirant2024critical}, we shall study the existence and blow-up behaviors of ground states to (\ref{goalmodel}) under the critical mass exponent case.   We also would like to mention that there exists the other critical exponent $\alpha=n-2\gamma'$ from the restriction of Sobolev embedding Theorem.
Next, we state our main results in Subsection \ref{mainresults11}.
\begin{remark}
    We would like to mention that, throughout the paper, we shall only consider the case $\gamma'<n$. Since when $\gamma'<n$, the minimization problem \eqref{ealphaM-117} will be well-posed for any $\alpha\in [n-\gamma',n)$, where the relevant discussions are shown in \cite{bernardini2022mass,bernardini2023ergodic}. 
\end{remark}

\subsection{Main results}\label{mainresults11}
We consider Hamiltonian $H$ satisfies \eqref{MFG-H} and $f(m):=- K_{\alpha}* m$ in \eqref{MFG-SS}, where $K_{\alpha}$ is the Riesz potential of order $\alpha\in[n-\gamma',n)$ defined by $K_{\alpha}(x)=\frac{1}{|x|^{n-\alpha}}$.  In particular, we assume potential $V$ is locally H\"{o}lder continuous and satisfies
\begin{center}
\begin{itemize}
\item[(V1). ] $\inf\limits_{x\in\mathbb R^n}V(x)=0\leq V(x)\in L^\infty_{\text{loc}}(\mathbb R^n)$.
\item[(V2). ] there exist positive constants $C,\bar C, K$ and $b, \delta$ such that 
\begin{subequations}\label{V2mainasumotiononv}
\begin{align}
    &C(1+|x|^b)\leq V(x)\leq \bar C e^{\delta|x|}, \ \ \forall x\in\mathbb R^n; \label{V2mainassumption_1}\\
   &0< C\leq \frac{V(x+y)}{V(x)}\leq \bar C\text{~for~}|x|\geq K\text{~with~}|y|<2; \label{V2mainassumption_2}\\
   &\sup_{\nu\in[0,1]}V(\nu x)\leq \bar CV(x)\text{~for~}|x|\geq K.\label{V2mainassumption_3}
\end{align}
\end{subequations}

\item[(V3).]  $|\mathcal Z|=0$ with $\mathcal Z:=\{x\in\mathbb R^n~|~V(x)=0\}$.
\end{itemize}
\end{center}
With the assumptions shown above, we shall classify the existence of ground states to \eqref{goalmodel} with $\alpha=n-\gamma'$ in terms of the total mass of density via the variational method.  Compared to the arguments for the existence of ground states to the Mean-field Games system with a local coupling, one has to control $m$ in some $L^p$ space with the aid of the nonlinearity in (\ref{energy-dual}).  Motivated by this, we exploit Hardy-Littlewood-Sobolev inequality stated in Appendix \ref{appendixA} and establish desired estimates.  
One of our main goals is to study the attainability of the constrained minimization problem \eqref{ealphaM-117}} with the critical mass exponent $\alpha=\alpha^*$, namely,
\begin{align}\label{ealphacritical-117}
e_{\alpha^*,M}:=\inf_{(m,w)\in \mathcal K_{M}}\mathcal E(m,w),
\end{align}
where $\mathcal K_{M}$ is given in \eqref{constraint-set-K} and the energy $\mathcal E(m,w)$ (\ref{energy-dual}) is precisely written as
\begin{align}\label{41engliang}
\mathcal E(m,w)=C_L\int_{\mathbb R^n}\Big|\frac{w}{m}\Big|^{\gamma'}m\,dx+\int_{\mathbb R^n}V(x) m\,dx-\frac{1}{2}\int_{\mathbb R^n}m(x)(K_{\alpha^*}*m)(x)\,dx
\end{align}
and 
$$\int_{\mathbb R^n}m(x)(K_{\alpha^*}*m)(x)\,dx=\int_{\mathbb R^n}\int_{\mathbb R^n}\frac{m(x)m(y)}{|x-y|^{n-\alpha^*}}\,dx\,dy.$$
To this end, similarly as discussed in \cite{cirant2024critical}, we have to first investigate the Gagliardo-Nirenberg type's inequality up to the critical mass exponent, which is
\begin{align}\label{sect2-equivalence-scaling}
\Gamma_{\alpha}:=\inf_{(m,w)\in\mathcal A}\frac{\Big(C_L\int_{\mathbb R^n}m\big|\frac{w}{m}\big|^{\gamma'}\,dx\Big)^{\frac{n-\alpha}{\gamma'}}\Big(\int_{\mathbb R^n}m\,dx\Big)^{\frac{2\gamma'+\alpha-n}{\gamma'}}}{\int_{\mathbb R^n}m(x)(K_{\alpha}*m)(x)\,dx}, \ \  \alpha\in[n-\gamma',n),
\end{align}
where
\begin{align}\label{mathcalA-equivalence}
\mathcal A:=\Big\{&(m,w)\in (L^1(\mathbb R^n)\cap W^{1,\hat q}(\mathbb R^n))\times L^{1}(\mathbb R^n)\nonumber\\
~&\text{s. t. }\int_{\mathbb R^n}\nabla m\cdot\nabla\varphi\,dx=\int_{\mathbb R^n}w\cdot\nabla\varphi\, dx,\forall \varphi\in C_c^{\infty}(\mathbb R^n),~\int_{\mathbb R^n} |x|^bm\,dx<+\infty,~m\geq,\not\equiv 0\text{~a.e.~}\Big\},
\end{align}
with $\hat q$ defined as (\ref{hatqconstraint}) and
$b>0$.  It is worthy mentioning that problem (\ref{sect2-equivalence-scaling}) is  scaling invariant  under the scaling $(t^{\beta}m(tx),t^{\beta+1}w(tx))$ for any $t>0$ and $\beta>0.$ 


With the help of the conclusions shown in \cite{bernardini2022mass}, we can prove the existence of minimizers to \eqref{sect2-equivalence-scaling} for any $\alpha\in(n-\gamma',n)$.  Then, we perform the approximation argument to study the case of $\alpha=n-\gamma'$.  In fact, we have
\begin{theorem}\label{thm11-optimal}
Suppose $\alpha=\alpha^*:=n-\gamma'$ in (\ref{sect2-equivalence-scaling}) and $\gamma'\in (1,n)$, then we have $\Gamma_{\alpha^*}$ is finite and attained by some minimizer $(m_{\alpha^*},w_{\alpha^*})\in\mathcal A$.  Moreover, we have there exists a classical solution {$\big(m_{\alpha^*},u_{\alpha^*}\big)\in W^{1,p}(\mathbb R^n)\times C^2(\mathbb R^n)$}, $\forall p>1,$ to the following Mean-field Games systems:
\begin{align}\label{limitingproblemminimizercritical0}
\left\{\begin{array}{ll}
-\Delta u+C_H|\nabla u|^{\gamma}-\frac{1}{M^*}=-K_{\alpha^*}*m,&x\in\mathbb R^n,\\
\Delta m+C_H\gamma\nabla\cdot (m|\nabla u|^{\gamma-2}\nabla u)=0,&x\in\mathbb R^n,\\
w=-C_H\gamma m\big|\nabla u\big|^{\gamma-2}\nabla u,\ \int_{\mathbb R^n}m\,dx=M^*,
\end{array}
\right.
\end{align}
where
\begin{align}\label{Mstar-critical-mass}
M^*:=2 \Gamma_{\alpha^*}.
\end{align}
In particular, there exists constants $c_1>0$ and $c_2>0$ such that $0<  m_{\alpha^*}(x)\leq c_1e^{-c_2|x|}$.
\end{theorem}

Theorem \ref{thm11-optimal} implies the best constant in (\ref{sect2-equivalence-scaling}) exists even if $\alpha=\alpha^*:=n-\gamma'.$  
Next, with the aid of Theorem \ref{thm11-optimal}, we are able to study the attainability of $e_{\alpha^*,M}$ and classify the existence of minimizers to \eqref{ealphacritical-117}, which is
\begin{theorem}\label{thm11}
    Suppose $V$ satisfies assumption (V1)-(V2) and  $M^*=2\Gamma_{\alpha^*}$, where $\Gamma_{\alpha^*}$ is shown in \eqref{Mstar-critical-mass},
the we have the following alternatives:
\begin{itemize}
    \item[(i).] 
    If $0<M<M^*$, problem \eqref{ealphacritical-117} admits at least one minimizer $(m_M,w_M)\in W^{1,p}(\mathbb R^n)\times L^{p}(\mathbb R^n)$,$\forall p>1$, which satisfies for some $\lambda_M\in \mathbb R$,
    \begin{align}\label{125potentialfreesystem}
\left\{\begin{array}{ll}
-\Delta u_{M}+C_H|\nabla u_{M}|^{\gamma}+\lambda_M=V(x)- K_{(n-\gamma')}\ast m_{M},\\
\Delta m_{M}+C_H\gamma\nabla\cdot (m_{M}|\nabla u_{M}|^{\gamma-2}\nabla u_{M})=0,\\
w_{M}=-C_H\gamma m_{M}|\nabla u_{M}|^{\gamma-2}\nabla u_{M}, \ \int_{\mathbb R^n}m_{M}\,dx=M<M^*.
\end{array}
\right.
\end{align}

    \item[(ii).] If $M> M^*$, problem \eqref{ealphacritical-117} does not admit any minimizer.
    \item [(iii).] If $M=M^*$ and potential $V$ satisfies (V3) additionally, then there does not exist any minimizer to problem \eqref{ealphacritical-117}.
\end{itemize}
\end{theorem}
\begin{remark}
We remark that in case (i), the $L^\infty$ estimates of $m$ is crucial due to the maximal regularity properties of Hamilton-Jacobi equations.  Following the arguments in \cite{cesaroni2018concentration,cirant2024critical}, we perform the blow-up analysis to obtain the desired estimates.
\end{remark}


Theorem \ref{thm11} indicates that the minimizers to \eqref{ealphacritical-117} do not exist when $M$ is large enough.  A natural question is the behaviors of ground states as $M\nearrow M^*$, where $M^*$ is the existence threshold defined by (\ref{Mstar-critical-mass}).  To explore this, we perform the scaling argument and investigate the convergence to get
\begin{theorem}\label{thm13basicbehavior} 
Suppose that $V(x)$ satisfies $(V1)-(V3)$ and let $(m_M,w_M)$ be the minimizer of $e_{\alpha^*,M}$ given in Theorem \ref{thm11} with $0<M<M^*$.  Then, we have
\begin{itemize}
    \item[(i).]
    \begin{align}\label{thm51property1}
    {\varepsilon}_M={\varepsilon}:=\Big(C_L\int_{\mathbb R^n}\bigg|\frac{w_M}{m_M}\bigg|^{\gamma'}m_M\,dx\Big)^{-\frac{1}{\gamma'}}\rightarrow 0\text{~as~}M \nearrow M^*.
    \end{align}
    \item[(ii).] 
    Let $\{x_{\varepsilon}\}$ be one of the global minimum points of $u_M$, then $\text{dist}(x_{\varepsilon},\mathcal Z)\rightarrow 0$ as $M \nearrow M^*$, where $\mathcal Z=\{x\in\mathbb R^n~|~V(x)=0\}$. Moreover,
    \begin{align}\label{thm51property2}
    u_{\varepsilon}:=\varepsilon^{\frac{2-\gamma}{\gamma-1}} u_{M}(\varepsilon x+x_{\varepsilon}),~
m_{\varepsilon}:=\varepsilon^n m_{M}(\varepsilon x+x_{\varepsilon}),~ w_{\varepsilon}:=\varepsilon^{n+1}w_M(\varepsilon x+x_{\varepsilon}),
    \end{align}
    satisfies  up to a subsequence,
     \begin{equation}\label{eq1.40}
      u_{\varepsilon}\rightarrow u_0\text{ in }C^2_{\rm loc}(\mathbb R^n), ~
     m_{\varepsilon}\rightarrow m_0 \text{ in }L^p(\mathbb R^n) ~\forall ~p\in[1,{\hat q}^*),~ \text{ and } w_{\varepsilon}\rightharpoonup  w_0 \text{ in }L^{\hat q}(\mathbb R^n),\end{equation}
    where $(m_0,w_0)$ is a minimizer of (\ref{sect2-equivalence-scaling}), and $(u_0,m_0,w_0)$ satisfies \eqref{limitingproblemminimizercritical0}.
In particular, when $V$ satisfies 
\begin{align}\label{cirant-V}
 C_V^{-1}(\max\{|x|-C_V,0\})^b\leq V(x)\leq C_V(1+|x|)^b,~~\text{ for some } b,C_V>0.
 \end{align}
 and $\bar x_{\varepsilon}$ denotes any one of global maximum points of $m_M$, then
\begin{align}\label{eq141realtionuminmmax}
\limsup_{\varepsilon\rightarrow 0^+}\frac{|\bar x_{\varepsilon}-x_{\varepsilon}|}{\varepsilon}<+\infty.
\end{align}
\end{itemize}
\end{theorem}
 Theorem \ref{thm13basicbehavior} implies as $M\nearrow M^*$, the ground states to (\ref{goalmodel}) concentrate and their basic blow-up behaviors are captured by the least energy solution to potential-free Mean-field Games systems with some mild assumptions imposed on $V$.  Moreover, by imposing some typical local expansions on potential $V(x)$, one can obtain the refined asymptotics of ground states, which are summarized as
\begin{theorem}\label{thm14preciseblowup}
Assume that all conditions in Theorem \ref{thm13basicbehavior} hold and suppose that $V$ has $l\in\mathbb{N}_+$ distinct zeros denoted by $\{P_1,\cdots,P_l\}$ and  there exist  $a_i>0$, $q_i>0$ and $d>0$ such that
\begin{align*}
V(x)=a_i|x-P_i|^{q_i}+O\big(|x-P_i|^{q_i+1}\big), \ \ 0<|x-P_i|\leq d,\ i=1,\cdots,l.
\end{align*}
Define
$$Z:=\{P_i~|~q_i=q,\ i=1,\cdots,l\} ~\text{ and }~Z_0:=\{P_i~|~q_i\in Z \text{~and~}\mu_i=\mu,i=1,\cdots,l\},$$
where $q:=\max\{q_1,\cdots,q_l\}$ and $
\mu:=\min\{\mu_i~|~P_i\in Z, i=1,\cdots,l\}$  with
\begin{equation*}
\mu_i:=\min\limits_{y\in\mathbb R^n}H_i(y),\ H_i(y):=\int_{\mathbb R^n} a_i|x+y|^{q_i}m_0(x)\,dx, ~i=1,\cdots,l.
\end{equation*}
Let $(m_{\varepsilon},w_{\varepsilon},u_{\varepsilon})$ be  the sequence given by (\ref{thm51property2}) and $(m_0,w_0,u_0)$ be the limiting solution. Then we have $x_{\varepsilon}\rightarrow P_i\in Z_0$.  Moreover, as $M\nearrow M^*,$
\begin{align*}
\frac{e_{\alpha^*,M}}{\frac{q+\gamma'}{q}\Big(\frac{q\mu}{\gamma'}\Big)^{\frac{\gamma'}{\gamma'+q}}\Big[1-\frac{M}{M^*}\Big]^{\frac{q}{\gamma'-q}}}\rightarrow 1,
\end{align*}
and
\begin{align}\label{131thm14}
\frac{\varepsilon}{\left(\frac{\gamma'}{q\mu}\left(1-\frac{M}{M^*}\right)\right]^{\frac{1}{\gamma'+q}}}\rightarrow 1,
\end{align}
where $e_{\alpha^*,M}$ and $\varepsilon=\varepsilon_M$ are given by \eqref{ealphaM-117} and \eqref{thm51property1}, respectively.  In particular, up to a subsequence,
\begin{align}\label{132thm14}
\frac{x_{\varepsilon}-P_i}{\varepsilon_M}\rightarrow y_0 ~\text{ with }~
P_i\in Z_0 ~\text{ and }~  H_i(y_0)=\inf_{y\in\mathbb R^n} H_i(y)=\mu.
\end{align}
\end{theorem}
Theorem \ref{thm14preciseblowup} demonstrates that under certain types of potentials with the local polynomial expansions, ground states to (\ref{goalmodel}) are localized as $M\nearrow M^*$, in which the locations converge to the flattest minima of $V$.

\medskip

The rest of this paper is organized as follows. In Section \ref{preliminaries}, we give some preliminaries for the investigation of ground states to \eqref{goalmodel} with $\alpha=n-\gamma'$.  Section \ref{sect3-optimal} is dedicated to the formulation of the optimal Gagliardo-Nirenberg type's inequality and the proof of Theorem \ref{thm11-optimal}. In Section \ref{sect4-criticalmass}, we prove Theorem \ref{thm11} by using the blow-up analysis and the Gagliardo-Nirenberg inequality shown in Theorem \ref{thm11-optimal}.  Finally, in Section \ref{sect5preciseblowup}, we focus on Theorem \ref{thm13basicbehavior} and \ref{thm14preciseblowup}, i.e. discuss the existence and concentration behaviors of ground states in some singular limit of $M$ given in \eqref{goalmodel}.  Without confusing readers, $C>0$ is chosen as a generic constant, which may vary line to line.

\medskip

\section{Preliminaries}\label{preliminaries}
This section is devoted to some preliminary results including existence and regularities of the solutions to Hamilton-Jacobi and Fokker-Planck equations.
\subsection{Hamilton-Jacobi Equations}\label{subsection1}
Consider the following Hamilton-Jacobi equation:
\begin{align}\label{introductioneqmaximalnew}
-\Delta u+C_H|\nabla u|^{\gamma}=f, ~x\in\Omega,
\end{align}
where $\Omega$ is a bounded domain with the smooth boundary, $C_H>0$ and $\gamma>1$.  
For the local $W^{2,p}$ estimates of the solutions $u$ to \eqref{introductioneqmaximalnew}, we have
\begin{lemma}[C.f. Theorem 1.1 in \cite{cirant2024critical}]\label{thmmaximalregularity}
Let $C_H>0$, $p>\frac{n}{\gamma'}$, $ \gamma\geq \frac{n}{n-1}$ and $f\in L^p(\Omega)$.   Suppose $u\in W^{2,p}(\Omega)$ solves \eqref{introductioneqmaximalnew} in the strong sense.  Then for each $M>0$ and $\Omega'\subset\subset \Omega$, we have
{$$ \Vert \nabla u\Vert_{L^p(\Omega')} + \Vert D^2 u\Vert_{L^p(\Omega')}\leq C,$$}
where $\Vert f\Vert_{L^p(\Omega)}\leq M $ and the constant $C=C(M,\mathrm{dist}(\Omega',\partial\Omega),n,p, C_H,\gamma')>0$.
\end{lemma}

Since our arguments in Section \ref{sect3-optimal}, \ref{sect4-criticalmass} and \ref{sect5preciseblowup} involve some limits of solution sequences, we also focus on the following sequence of Hamilton-Jacobi equations:
\begin{align}\label{HJB-regularity}
-\Delta u_k+C_H|\nabla u_k|^{\gamma}+\lambda_k=V_k(x)+f_k(x),\ \ x\in\mathbb R^n,
\end{align}
where $C_H>0$ and $\gamma>1$ are fixed.  Here $(u_k,\lambda_k)$ denote the solution pair to \eqref{HJB-regularity}. Concerning the regularities of $u_k$, we obtain
\begin{lemma} [C.f. Lemma 3.1 in \cite{cirant2024critical}]\label{sect2-lemma21-gradientu}
Assume that $f_k\in L^{\infty}(\mathbb R^n)$ satisfies $\Vert f_k\Vert_{L^\infty}\leq C_f$ and $|\lambda_k|\leq \lambda$.  Suppose the potential functions $V_k(x)$ are uniformly local H\"{o}lder continuous satisfying  $0\leq V_k(x)\rightarrow +\infty$ as $|x|\rightarrow +\infty,$ and there exists $R>0$ sufficiently large such that
\begin{align}\label{themoregeneralvkcondition}
0< C_1\leq \frac{V_k(x+y)}{V_k(x)}\leq C_2,\text{~for~all~}k\text{~and~all~}|x|\geq R \text{~with~}|y|<2,
\end{align}
where the positive constants $C_f$, $\lambda$, $R$, $C_1$ and $C_2$ are independent of $k$.  Define $(u_k,\lambda_k)\in C^2(\mathbb R^n)\times \mathbb R$ as the solutions to (\ref{HJB-regularity}).  Then, we have for all $k$,
\begin{align}\label{usolutiongradientestimatepre1}
|\nabla u_k(x)|\leq C(1+V_k(x))^{\frac{1}{\gamma}}, \text{ for all } x\in\mathbb{R}^n,
\end{align}
where constant $C$ depends on $C_H$, $C_1$, $C_2$, $\lambda$, $\gamma'$, $n$ and $C_f.$

In particular, if each $V_k$ satisfies 
\begin{align}\label{cirant-Vk}
 C_F^{-1}(\max\{|x|-C_F,0\})^b\leq V_k(x)\leq C_F(1+|x|)^b,~~\text{for all }k\text{ and }x\in\mathbb R^n,
 \end{align}
  where $b\geq 0$ and $C_{F}>0$ independent of $k,$ then we have
\begin{align}\label{usolutiongradientestimatepre}
|\nabla u_k|\leq C(1+|x|)^{\frac{b}{\gamma}}, ~\text{for all }k\text{ and }x\in\mathbb R^n,
\end{align}
where constant $C$ depends on $C_H$, $C_{F}$, $b$, $\lambda$, $\gamma'$, $n$ and $C_f.$
\end{lemma}


For the lower bounds of $u_k$, we have the following results:
\begin{lemma}[C.f. Lemma 3.2 in \cite{cirant2024critical}]\label{lowerboundVkgenerallemma22}
Suppose all conditions in Lemma \ref{sect2-lemma21-gradientu} hold. 
Let $u_k$ be a family of $C^2$ solutions and assume that $u_k(x)$ are bounded from below uniformly.  Then there exist positive constants $C_3$ and $C_4$ independent of $k$ such that
\begin{align}\label{29uklemma22}
u_k(x)\geq C_3V^{\frac{1}{\gamma}}_k(x)-C_4,\text{~}\forall x\in\mathbb R^n,~\text{for all }k.
\end{align}
In particular, if the following conditions hold on $V_k$
\begin{align}\label{cirant-Vk-1}
 C_F^{-1}(\max\{|x|-C_F,0\})^b\leq V_k(x)\leq C_F(1+|x|)^b,~~\text{for all }k\text{ and }x\in\mathbb R^n,
 \end{align}
 where constants $b> 0$ and $C_{F}$ are independent of $k,$ then we have
\begin{align}\label{usolutionlowerestimatepre-11}
 u_k(x)\geq C_3|x|^{1+\frac{b}{\gamma}}-C_4,\text{~for all}k, x\in\mathbb R^n.
\end{align}
If $b=0$ in (\ref{cirant-Vk-1}) and there exist $R>0$ and $\hat \delta>0$ independent of $k$ such that
\begin{align}\label{lemma22holdsbeforeconclusion}
f_k+V_k-\lambda_k>\hat \delta>0\text{~for~all }|x|>R,
\end{align}
then \eqref{usolutionlowerestimatepre-11} also holds.
\end{lemma}

The following results are concerned with
 the existence of the classical solution to (\ref{HJB-regularity}), which are 
\begin{lemma} [C.f. Lemma 3.3 in \cite{cirant2024critical}] \label{lemma22preliminary}
Suppose $V_k+f_k$ are locally H\"{o}lder continuous and bounded from below uniformly in $k$.  Define
\begin{align}
\bar \lambda_k:=\sup\{\lambda\in\mathbb R~|~(\ref{HJB-regularity})\text{ has a solution }u_k\in C^2(\mathbb R^n)\}.
\end{align}
Then we have
\begin{itemize}
    \item[(i).] $\bar \lambda_k$ are finite for every $k$ and (\ref{HJB-regularity}) admits a solution $(u_k,\lambda_k)\in C^2(\mathbb R^n)\times \mathbb R$  with $\lambda_k=\bar \lambda_k$ and $u_k(x)$ being bounded from below (may not uniform in $k$).  Moreover,
    $$\bar \lambda_k=\sup\{\lambda\in\mathbb R~|~(\ref{HJB-regularity})\text{ has a subsolution }u_k\in C^2(\mathbb R^n)\}.$$
    \item[(ii).] If $V_k$ satisfies (\ref{cirant-Vk}) with $b>0$, then $u_k$ is unique up to constants for fixed $k$ and there exists a positive constant $C$ independent of $k$ such that
    \begin{align}\label{lowerboundusect2}
    u_k(x)\geq C|x|^{\frac{b}{\gamma}+1}-C, \forall x\in\mathbb R^n.
    \end{align}
    In particular, if $V_k\equiv 0$ and $b=0$ in \eqref{cirant-V} and there exists $\sigma>0$ independent of $k$ such that
    \begin{align}\label{verifylemma22}
    f_k-\lambda_k\geq \sigma>0, \ \ \text{for~} |x|>K_2,
    \end{align}
    where $K_2>0$ is a large constant independent of $k$, then (\ref{lowerboundusect2}) also holds.
 \end{itemize}
 (iii). If $V_k$ satisfies \eqref{V2mainassumption_2} with $V$ replaced by $V_k$ and positive constants $C_1$, $C_2$ and $\delta$ independent of $k,$ then there exist uniformly bounded from below classical solutions $u_k$ to problem (\ref{HJB-regularity}) satisfying estimate (\ref{29uklemma22}).
\end{lemma}

\subsection{Fokker-Planck Equations}\label{subsection2}
Now, we focus on the following Fokker-Planck equations:
\begin{align}\label{sect2-FP-eq}
-\Delta m+\nabla\cdot w=0,\ \ x\in\mathbb R^n,
\end{align}
where $w$ is given and $m$ denotes the solution. Firstly, we state the regularity results of solutions to equation \eqref{sect2-FP-eq}, which are
\begin{lemma}\label{lemma21-crucial}
Assume that $(m,w)\in \left(L^1(\mathbb R^n)\cap W^{1, \hat q}(\mathbb R^n)\right)\times L^1(\mathbb R^n)$ is a  solution to (\ref{sect2-FP-eq}) and
\begin{equation*}
\Lambda_{\gamma'}:=\int_{\mathbb R^n}|m|\Big|\frac{w}{m}\Big|^{\gamma'}\, dx<\infty.
\end{equation*}
Then, we have $w\in L^{1}(\mathbb R^n)\cap L^{\hat q}(\mathbb R^n)$ and there exists constant $\mathcal{C}=\mathcal{C}(\Lambda_{\gamma'},\|m\|_{L^1(\mathbb R^n)})>0$ such that
$$\|m\|_{W^{1,\hat q}(\mathbb R^n)}, \|w\|_{L^1(\mathbb R^n)},\|w\|_{L^{\hat q}(\mathbb R^n)}\leq \mathcal{C}.$$ 

\end{lemma}
\begin{proof}
See the proof of Lemma 3.5 in \cite{cirant2024critical}.  
\end{proof}

\begin{lemma}[C.f. Corollary 1.1 in \cite{cirant2024critical}]\label{lemma21-crucial-cor}
Assume that $(m,w)\in (L^1(\mathbb R^n)\cap L^{1+\beta}(\mathbb R^n)\cap W^{1, q}(\mathbb R^n))\times L^1(\mathbb R^n)$ is the solution to (\ref{sect2-FP-eq}) with
$$\frac{1}{ q}=\frac{1}{\gamma'}+\frac{1}{\gamma(1+\beta)}.$$
Then for $\beta \in(0,\frac{\gamma'}{n}\big]$, there exists a positive constant $C$ depending only on $n$ and $\beta$ such that
\begin{align}\label{lemma24eq28w1q}
\Vert \nabla m\Vert_{L^{q}(\mathbb R^n)}\leq C\Big(\int_{\mathbb R^n}m\Big|\frac{w}{m}\Big|^{\gamma'}\, dx\Big)^\frac{1}{\gamma'}\Vert m\Vert_{L^{1+\beta}}^{\frac{1}{\gamma}}.
\end{align}
Moreover, there exists a positive constant $C$ only depending on $\gamma',$ $n$ and $\alpha$ such that
\begin{align}
\Vert m\Vert^{1+\beta}_{L^{1+\beta}(\mathbb R^n)}\leq  C\bigg(\int_{\mathbb R^n}m\Big|\frac{w}{m}\Big|^{\gamma'}\,dx\bigg)^{\frac{n\beta}{\gamma'}}\bigg(\int_{\mathbb R^n}m\,dx\bigg)^{\frac{(\beta+1)\gamma'-n\beta}{\gamma'}}.
\end{align}
\end{lemma}

Next, we discuss the exponential decay property of the solutions to system \eqref{MFG-SS} and obtain

\begin{lemma}\label{mdecaylemma}
Assume that $(u,m,\lambda)\in   C^2(\mathbb R^n)\times \big(W^{1,p}(\mathbb R^n)\cap L^1(\mathbb R^n)\big)\times\mathbb R$ with $ p>n$ and $\lambda<0$ is the solution of the following potential-free problem
\begin{align}\label{26preliminaryfinal}
\left\{\begin{array}{ll}
-\Delta u+C_H|\nabla u|^{\gamma}+\lambda=-K_{\alpha}\ast m, &x\in\mathbb R^n,\\
\Delta m+C_H\gamma\nabla\cdot(m|\nabla u|^{\gamma-2}\nabla u )=0, &x\in\mathbb R^n.
\end{array}
\right.
\end{align}
 Suppose $u$ is bounded from below.  Then, we have there exist $\kappa_1,\kappa_2>0$ such that
\begin{align}\label{exponentialdecaym}
m(x)\leq \kappa_1 e^{-\kappa_2|x|}  ~\text{ for all } x\in \mathbb R^n.
\end{align}
\end{lemma}
\begin{proof}
Noting that $m\in W^{1,p}(\mathbb R^n)$ with $p>n$, we use Sobolev embedding to get $m\in  C^{0,\theta}(\mathbb R^n)$ for some $\theta\in(0,1)$, and thus $m\in L^{\infty}(\mathbb R^n)$. Moreover, by using the fact that $m\in L^{1} (\mathbb R^n)$ and the interpolation inequality, one finds $m\in L^{q}(\mathbb R^n)$ for every $q\in (1,\infty)$. Therefore, invoking Lemma \ref{H-L-S} and Lemma \ref{HolderforRiesz}, one can obtain that $K_{\alpha}\ast m\in L^{\beta}(\mathbb R^n)\cap C^{0,\theta_1}(\mathbb R^n)$ for some $\beta>1$ and $\theta_1\in(0,1)$, which implies  
\begin{align*}
    K_{\alpha}\ast m\rightarrow 0\ \ \text{as} \ \  |x|\rightarrow +\infty.
\end{align*}
The rest of proof follows from \cite[Proposition 4.2]{bernardini2022mass} and  \cite[Lemma 3.6]{cirant2024critical}.
\end{proof}

Thanks to Lemma \ref{mdecaylemma}, we establish Pohozaev identities satisfied by the solution to system (\ref{26preliminaryfinal}), which are
\begin{lemma}[C.f. Lemma 3.1 in \cite{bernardini2023ergodic}]\label{poholemma}
Assume all conditions satisfied by $(u,m,\lambda)$ hold in Lemma  \ref{mdecaylemma} and denote $w=-C_H\gamma m|\nabla u|^{\gamma-2}\nabla u$.
 Then the following identities hold:
\begin{align}\label{eq2.49}
\left\{\begin{array}{ll}
\lambda\int_{\mathbb R^n}m\, dx=-\frac{\alpha+2\gamma'-n}{2\gamma'}\int_{\mathbb R^n}m(x)(K_{\alpha}*m)(x)\,dx,\\
C_L\int_{\mathbb R^n}m\big|\frac{w}{m}\big|^{\gamma'}\, dx=\frac{n-\alpha}{2\gamma'}\int_{\mathbb R^n}m(x)(K_{\alpha}*m)(x)\,dx=(\gamma-1)C_H\int_{\mathbb R^n} m|\nabla u|^{\gamma}\, dx.
\end{array}
\right.
\end{align}

\end{lemma}
\begin{proof}

Proceeding the similar argument shown in Lemma 3.7 of \cite{cirant2024critical}, we can prove this lemma.  For the sake of completeness, we exhibit the proof briefly.  First all, we multiply the $u$-equation and $m$-equation in (\ref{26preliminaryfinal}) by $m$ and $u$, respectively, then integrate them by parts and subtract the two identities to get
\begin{align}\label{sect3-combine11}
(1-\gamma)C_H\int_{\mathbb R^n}m|\nabla u|^{\gamma}\,dx+\lambda\int_{\mathbb R^n}m\,dx =-\int_{\mathbb R^n}m(x)(K_{\alpha}*m)(x)\, dx,
\end{align}
where we have used the exponential decay property of $m$ shown in Lemma \ref{mdecaylemma} and the uniformly boundedness of $\nabla u$ stated in Lemma \ref{HJB-regularity}.

Next, we focus on the proof of the following identity:
\begin{align}\label{pohoequation2}
- n\lambda\int_{\mathbb R^n}m\,dx-\frac{n+\alpha }{2}\int_{\mathbb R^n}m(x)(K_{\alpha}*m)(x)\,dx+C_H\frac{n-\gamma'}{\gamma'-1}\int_{\mathbb R^n} m|\nabla u|^{\gamma}\,dx=0.
\end{align}
In fact, by testing the first equation and the second equation in (\ref{26preliminaryfinal}) against $\nabla m\cdot x$ and $\nabla u\cdot x$, we apply the integration by parts to obtain
\begin{align}\label{pohocombine1}
\int_{\mathbb R^n}(-(K*m)(x)-\lambda)\nabla m\cdot x\,dx
=& {\int_{\mathbb R^n}\nabla u\cdot \nabla(\nabla m\cdot x)\,dx}-{C_H}\int_{\mathbb R^n}\nabla\cdot(|\nabla u|^{\gamma}x)m\,dx,
\end{align}
and
\begin{align}\label{pohocombine2}
-{C_H}\int_{\mathbb R^n}\nabla(|\nabla u|^{\gamma})\cdot xm\,dx=\int_{\mathbb R^n}\nabla m\cdot \nabla(\nabla u\cdot x)\,dx+C_H\gamma\int_{\mathbb R^n}|\nabla u|^{\gamma}m\,dx,
\end{align}
where the boundary integrals vanish due to the decay property of $m$ and the upper bound of $u.$
Also, we find 
\begin{align}\label{pohocombine3}
\int_{\mathbb R^n}\nabla u\cdot \nabla(\nabla m\cdot x)\,dx=&\sum_{i,j=1}^n\int_{\mathbb R^n} u_{x_i}m_{x_ix_j}x_j\,dx+\int_{\mathbb R^n}\nabla u\cdot \nabla m\,dx\nonumber\\
=&-\sum_{i,j=1}^n\int_{\mathbb R^n}m_{x_i}u_{x_ix_j}x_j\,dx+(1-n)\int_{\mathbb R^n}\nabla u\cdot\nabla m\,dx\nonumber\\
=&-\int_{\mathbb R^n} \nabla m\cdot\nabla(\nabla u\cdot x)\,dx+(2-n)\int_{\mathbb  R^n}\nabla u\cdot \nabla m\,dx.
\end{align}
Collecting (\ref{pohocombine1}), (\ref{pohocombine2}) and (\ref{pohocombine3}), we have the following equality holds:
\begin{align}\label{pohofinalbefore}
{\int_{\mathbb R^n}(-(K*m)(x)-\lambda)\nabla m\cdot x\,dx}=C_H\big(\gamma-{n}\big)\int_{\mathbb R^n}|\nabla u|^{\gamma}m\,dx+(2-n)\int_{\mathbb R^n}\nabla u\cdot \nabla m\,dx.
\end{align}
With the help of the integration by parts, one further gets
\begin{align*}
- n\lambda\int_{\mathbb R^n}m\,dx-\frac{n+\alpha}{2}\int_{\mathbb R^n}m(x)(K*m)(x)\,dx+ C_H\Big(\gamma-{n}\Big)\int_{\mathbb R^n} |\nabla u|^{\gamma}m\,dx+(2-n)\int_{\mathbb R^n}\nabla u\cdot \nabla m\,dx=0,
\end{align*}
which indicates
\eqref{pohoequation2} by using the $u$-equation in \eqref{26preliminaryfinal}.   In addition, since $w=-C_H\gamma m|\nabla u|^{\gamma-2}\nabla u$ and $C_L=\frac{1}{\gamma'}(\gamma C_H)^{\frac{1}{1-\gamma}}$, we obtain
\begin{align}\label{sect3-combine200}
C_L\int_{\mathbb R^n}m\bigg|\frac{w}{m}\bigg|^{\gamma'}\, dx=C_L (C_H\gamma)^{\gamma'}\int_{\mathbb R^n}m|\nabla u|^{\gamma}\,dx=(\gamma-1)C_H\int_{\mathbb R^n}m|\nabla u|^{\gamma}\,dx.
\end{align}
Finally, by using \eqref{sect3-combine11}, \eqref{pohoequation2} and \eqref{sect3-combine200}, we conclude that \eqref{eq2.49} holds.

We mention that the argument shown above hold only when $\gamma\geq 2$ and in this case, the Fokker-Planck equation can be solved in the strong sense.  When $1<\gamma<2$, one can only solve the Fokker-Planck equation in the weak sense.  Whereas, we can replace $H$ with $H_{\epsilon}(\boldsymbol{p}):=C_H(\epsilon+|\boldsymbol{p}|^2)^{\frac{\gamma}{2}}$ in \eqref{MFG-H} and proceed the same argument shown above with $m_{\epsilon}$, then take the limit $\epsilon\rightarrow 0$ to get the desired conclusion.

\end{proof}
Now, we are well prepared to prove Theorem \ref{thm11-optimal} and the arguments are shown in Section \ref{sect3-optimal}.

\section{Optimal Gagliardo-Nirenberg Type's Inequality}\label{sect3-optimal}
In this section, we are going to discuss the existence of minimizers to problem \eqref{sect2-equivalence-scaling} and prove Theorem \ref{thm11-optimal}. As mentioned above, problem \eqref{sect2-equivalence-scaling} is scaling invariant under the scaling $(t^{\beta}m(tx),t^{\beta+1}w(tx))$ for any $t>0$ and $\beta>0.$
Therefore, one can verify that \eqref{sect2-equivalence-scaling} is equivalent to 
\begin{align}\label{optimal-inequality-sub}
\Gamma_{\alpha}:=\inf_{(m,w)\in\mathcal \mathcal A_M}\frac{\Big(C_L\int_{\mathbb R^n}m\big|\frac{w}{m}\big|^{\gamma'}\,dx\Big)^{\frac{n-\alpha}{\gamma'}}\Big(\int_{\mathbb R^n}m\,dx\Big)^{\frac{2\gamma'+\alpha-n}{\gamma'}}}{\int_{\mathbb R^n}m(x)(K_{\alpha}*m)(x)\,dx}, \ \  \alpha\in[n-\gamma',n),
\end{align}
where
\begin{align}\label{mathcalAMbegining}
\mathcal A_M:=\Big\{(m,w)\in  \mathcal{A}, \int_{\mathbb R^n}m\,dx=M>0\Big\}.
\end{align}

Now, we start by studying the subcritical mass exponent case of problem \eqref{optimal-inequality-sub}, namely, $\alpha\in(n-\gamma',n)$.  For this case, Bernardini \cite{bernardini2022mass} proved that there exists $(\bar u_{\alpha,M},\bar m_{\alpha,M},\lambda_{\alpha,M})\in C^{2}(\mathbb R^n)\times W^{1,p}(\mathbb R^n)\times \mathbb R$ for every $p\in(0,+\infty)$ solving  the following system 
\begin{align}\label{eq-attained-sub}
\left\{\begin{array}{ll}
-\Delta u+C_H|\nabla u|^{\gamma}+\lambda=-K_{\alpha}*m,&x\in\mathbb R^n,\\
\Delta m+C_H\gamma\nabla\cdot(m|\nabla u|^{\gamma-2}\nabla u)=0,&x\in\mathbb R^n,\\
\int_{\mathbb R^n}m\, dx=M>0,
\end{array}
\right.
\end{align}
which is the classical solution to system \eqref{eq-attained-sub}.  Furthermore, the author showed there exist $c_{1,M}, c_{2,M}>0$ such that 
\begin{equation}\label{eq3.2}
    0<\bar m_{\alpha,M}<c_{1,M}e^{-c_{2,M}|x|}.
\end{equation}
In particular, Bernardini obtained the following minimization problem 
\begin{align}\label{energy-epsilon0}
e_{0,\alpha,M}:=\inf\limits_{(m,w)\in {\mathcal A}_M}\mathcal E_0(m,w) 
\end{align}
with
\begin{align}\label{energy-epsilon00}
\mathcal E_0(m,w)=\int_{\mathbb R^n}\Bigg(C_Lm\bigg|\frac{w}{m}\bigg|^{\gamma'}\Bigg)\,dx-\frac{1}{2}\int_{\mathbb R^n}m(x)(K_{\alpha}*m)(x)\,dx
\end{align}
is attained by the pair $(\bar m_{\alpha,M},\bar w_{\alpha,M})$ with $\bar w_{\alpha,M}=-C_H\gamma\bar m_{\alpha,M}|\nabla\bar u_{\alpha,M}|^{\gamma-2}\nabla\bar u_{\alpha,M}$.  In addition, invoking Lemma \ref{poholemma}, one finds
\begin{align}\label{sect3-poho-final}
	\left\{\begin{array}{ll}
       \lambda\int_{\mathbb R^n}\bar m_{\alpha,M}\, dx=-\frac{2\gamma'+\alpha-n}{2\gamma'}\int_{\mathbb R^n}\bar m_{\alpha,M}(x)(K_{\alpha}*\bar m_{\alpha,M})(x)\,dx,\\
		C_L\int_{\mathbb R^n}\bar m_{\alpha,M}\big|\frac{\bar w_{\alpha,M}}{\bar m_{\alpha,M}}\big|^{\gamma'}\, dx=\frac{n-\alpha}{2\gamma'}\int_{\mathbb R^n}\bar m_{\alpha,M}(x)(K_{\alpha}*\bar m_{\alpha,M})(x)\,dx=(\gamma-1)C_H\int_{\mathbb R^n} \bar m_{\alpha,M}|\nabla u|^{\gamma}\, dx.
	\end{array}
	\right.
\end{align}
Collecting the results shown above, we are able to investigate a relationship between $(\bar m_{\alpha,M},\bar w_{\alpha,M})$, the minimizer of \eqref{energy-epsilon0} and the minimizer of problem (\ref{optimal-inequality-sub}), which is
\begin{lemma}\label{sect3-lemma32}
For any fixed $\alpha\in\big(n-\gamma',n)$ and $M>0$, problem \eqref{sect2-equivalence-scaling} is attained by $(\bar m_{\alpha,M},\bar w_{\alpha,M})$ with $e_{0,\alpha,M}=\mathcal E_0(\bar m_{\alpha,M},\bar w_{\alpha,M})$.  More precisely, we have
\begin{align}\label{sect3-relation-Calpha-emalpha}
\Gamma_{\alpha}=\frac{(n-\alpha) (-e_{0,\alpha,M})^{\frac{n-\alpha-\gamma'}{\gamma'}}M^{\frac{2\gamma'+\alpha-n}{\gamma'}}}{2\gamma'}\Bigg(\frac{\gamma'-n+\alpha}{n-\alpha}\Bigg)^{\frac{\gamma'+\alpha-n}{\gamma'}}.
\end{align}
\end{lemma}
\begin{proof}
We follow the procedures shown in Lemma 4.1 of \cite{cirant2024critical} to prove this lemma.  First of all, we define
\begin{align}\label{sect3-Galpha-39}
G_{\alpha}(m,w):&=\frac{\Big(C_L\int_{\mathbb R^n}m\big|\frac{w}{m}\big|^{\gamma'}\,dx\Big)^{\frac{n-\alpha}{\gamma'}}\Big(\int_{\mathbb R^n}m\,dx\Big)^{\frac{2\gamma'+\alpha-n}{\gamma'}}}{\int_{\mathbb R^n}m(x)(K_{\alpha}*m)(x)\,dx}.
\end{align}
With the definition \eqref{sect3-Galpha-39}, the minimization problem \eqref{optimal-inequality-sub} can be rewritten as
\begin{align}\label{sect3-Calpha-equivalent}
\Gamma_{\alpha}=\inf_{(m,w)\in{{\mathcal A}}_M}G_{\alpha}(m,w).
\end{align}

Now, we aim to verify that  \eqref{sect3-Calpha-equivalent} is attained by $(\bar m_{\alpha,M},\bar w_{\alpha,M})$, which is the minimizer of \eqref{energy-epsilon0}. First of all, we estimate the energy $\mathcal E_0$ defined by \eqref{energy-epsilon0} from below. We remark that $G_\alpha(m,w)=+\infty$ provided with $\int_{\mathbb R^n}m\big|\frac{w}{m}\big|^{\gamma'}\, dx=+\infty$.  Thus, we only need to consider the case that $(m,w)\in{{\mathcal A}}_M$ satisfying $\int_{\mathbb R^n}m\big|\frac{w}{m}\big|^{\gamma'}\, dx<\infty$.
Define $(m_{\mu}(x),w_{\mu}(x))=(\mu^nm(\mu x),\mu^{n+1}w(\mu x))$ for $\mu\in\mathbb R^+\backslash\{0\}$, then we have
\begin{align}\label{sect3-311-mathcalE0}
\mathcal E_0(m_{\mu},w_{\mu})=&\mu^{\gamma'}\int_{\mathbb R^n}C_Lm\Big|\frac{w}{m}\Big|^{\gamma'}\, dx-\frac{1}{2}\mu^{n-\alpha}\int_{\mathbb R^n}m(x)(K_{\alpha}*m)(x)\,dx\nonumber\\
\geq&-\bigg(\frac{n-\alpha}{2\gamma'}\bigg)^{\frac{\gamma'}{\gamma'-n+\alpha}}\Bigg(\frac{\gamma'-n+\alpha}{n-\alpha}\Bigg)\bigg(\int_{\mathbb R^n}m(x)(K_{\alpha}*m)(x)\,dx\bigg)^{\frac{\gamma'}{\gamma'-n+\alpha}}\bigg(C_L\int_{\mathbb R^n}m\Big|\frac{w}{m}\Big|^{\gamma'}\, dx\bigg)^{-\frac{n-\alpha}{\gamma'-n+\alpha}},
\end{align}
where the equality holds if and only if
\begin{align*}
\mu=\Bigg[\frac{(n-\alpha)\int_{\mathbb R^n}m(x)(K_{\alpha}*m)(x)\,dx} {2\gamma' C_L\int_{\mathbb R^n}m\big|\frac{w}{m}\big|^{\gamma'}\, dx}\Bigg]^{\frac{1}{\gamma'-n+\alpha}}.
\end{align*}
 It then follows from the definition of $e_{0,\alpha,M}:=\inf\limits_{(m,w)\in { {\mathcal A}}_M}\mathcal E_0(m,w)$ and \eqref{sect3-311-mathcalE0} that
\begin{align*}
-\bigg(\frac{n-\alpha}{\gamma'}\bigg)^{\frac{\gamma'}{\gamma'-n+\alpha}}\bigg(\frac{\gamma'-n+\alpha}{n-\alpha}\bigg)\bigg( \frac{1}{2} \int_{\mathbb R^n}m(x)(K_{\alpha}*m)(x)\,dx\bigg)^{\frac{\gamma'}{\gamma'-n+\alpha}}\bigg(C_L\int_{\mathbb R^n}m\Big|\frac{w}{m}\Big|^{\gamma'}\, dx\bigg)^{-\frac{n-\alpha}{\gamma'-n+\alpha}}\geq e_{0,\alpha,M},
\end{align*}
which yields
\begin{align}\label{313-estimate-sect3}
\frac{\bigg(C_L\int_{\mathbb R^n}m\Big|\frac{w}{m}\Big|^{\gamma'}\, dx\bigg)^{\frac{n-\alpha}{\gamma'-n+\alpha}}}{\bigg( \frac{1}{2} \int_{\mathbb R^n}m(x)(K_{\alpha}*m)(x)\,dx\bigg)^{\frac{\gamma'}{\gamma'-n+\alpha}}}\geq (-e_{0,\alpha,M})^{-1}\Big(\frac{n-\alpha}{\gamma'}\Big)^{\frac{\gamma'}{\gamma'-n+\alpha}}\bigg(\frac{\gamma'-n+\alpha}{n-\alpha}\bigg).
\end{align}
Denote 
$$\mathcal H_{\alpha,M}:=\frac{n-\alpha}{\gamma'}(-e_{0,\alpha,M})^{\frac{n-\gamma'-\alpha}{\gamma'}}\bigg(\frac{\gamma'-n+\alpha}{n-\alpha}\bigg)^{\frac{\gamma'-n+\alpha}{\gamma'}},$$
then we invoke \eqref{313-estimate-sect3} to obtain
\begin{align}\label{combine-sect3-314}
G_{\alpha}(m,w)=\frac{\Big(C_L\int_{\mathbb R^n}m\big|\frac{w}{m}\big|^{\gamma'}\, dx\Big)^{\frac{n-\alpha}{\gamma'}}\Big(\int_{\mathbb R^n}m\,dx\Big)^{\frac{2\gamma'+\alpha-n}{\gamma'}}}{\int_{\mathbb R^n}m(x)(K_{\alpha}*m)(x)\,dx}\geq\frac{1}{2}\mathcal H_{\alpha,M}M^{\frac{2\gamma'+\alpha-n}{\gamma'}},
\end{align}
where we have used the definition \eqref{sect3-Galpha-39} and the fact 
$\int_{\mathbb R^n}m\,dx=M.$

Next, by using the fact that $(\bar m_{\alpha,M},\bar w_{\alpha,M})$ is a minimizer of problem  \eqref{energy-epsilon0}, we apply \eqref{sect3-poho-final} to get
\begin{align}\label{sect3-combine-315}
G_\alpha(\bar m_{\alpha,M},\bar w_{\alpha,M})=\frac{1}{2}\mathcal H_{\alpha,M}M^{\frac{2\gamma'+\alpha-n}{\gamma'}}.
\end{align}
Combining \eqref{combine-sect3-314} with \eqref{sect3-combine-315}, one can conclude that \eqref{sect3-Calpha-equivalent} is attained by $(\bar m_{\alpha,M},\bar w_{\alpha,M}).$  Moreover, we have
\begin{align*}
\Gamma_{\alpha}=G_\alpha(\bar m_{\alpha,M},\bar w_{\alpha,M})=\frac{1}{2}\mathcal H_{\alpha,M}M^{\frac{2\gamma'+\alpha-n}{\gamma'}}=\frac{(n-\alpha) (-e_{0,\alpha,M})^{\frac{n-\alpha-\gamma'}{\gamma'}}M^{\frac{2\gamma'+\alpha-n}{\gamma'}}}{2\gamma'}\bigg(\frac{\gamma'-n+\alpha}{n-\alpha}\bigg)^{\frac{\gamma'+\alpha-n}{\gamma'}},
\end{align*}
 which shows that \eqref{sect3-relation-Calpha-emalpha} holds and the proof of this lemma is completed.
\end{proof}
We can see from Lemma \ref{sect3-lemma32} that for all $M>0$, Gagliardo-Nirenberg type inequalities given by \eqref{optimal-inequality-sub} can be attained under the subcritical mass exponent case $\alpha\in(n-\gamma',n)$.  In addition, invoking \eqref{sect3-poho-final} and  \eqref{sect3-relation-Calpha-emalpha}, we obtain that
\begin{align}\label{sect3-lemma31-imply}
	e_{0,\alpha,M}=\bigg(\frac{\gamma'+\alpha-n}{2\gamma'+\alpha-n}\bigg)\lambda M,
\end{align}
and
\begin{align}\label{sect3-318lambdam}
	\lambda M=-S_{\alpha,M}\bigg(\frac{2\gamma'+\alpha-n}{n-\alpha}\bigg)\bigg(\frac{n-\alpha}{\gamma'}\bigg)^{\frac{\gamma'}{\gamma'+\alpha-n}}, \ \ S_{\alpha,M}:=\Bigg[\frac{M^{\frac{2\gamma'+\alpha-n}{\gamma'}}}{2\Gamma_{\alpha}}\Bigg]^{\frac{\gamma'}{\gamma'+\alpha-n}}.
\end{align}


The next lemma will indicate that $\Gamma_{\alpha}$ defined in \eqref{optimal-inequality-sub} is uniformly bounded as  $\alpha\searrow(n-\gamma')$, which is essential for us to investigate the mass critical exponent case and prove Theorem \ref{thm11-optimal}.  
\begin{lemma}\label{uniformlyboundC1CalphaC2}
There are constants $C_1>0$ and $C_2>0$ independent of $\alpha$ such that for all $\alpha\in[n-\gamma',n-\gamma'+\epsilon)$ with $\epsilon>0$ small,
\begin{align}\label{317inlemma33alphaC2}
0<C_1\leq \Gamma_{\alpha}\leq C_2.
\end{align}
\end{lemma}
\begin{proof}
We first estimate $\Gamma_{\alpha}$ from above uniformly in $\alpha$.  By setting $\tilde m=e^{-|x|}$ with $\tilde w=\nabla \tilde m$, we have $(\tilde m,\tilde w)\in\mathcal A$ for any $\alpha\in(n-\gamma',n)$ and
\begin{align}\label{upperboundCalpha}
\Gamma_{\alpha}\leq G_{\alpha}(\tilde m,\tilde w)=\frac{\Big(C_L\int_{\mathbb R^n}\tilde m\Big|\frac{\tilde w}{\tilde m}\Big|^{\gamma'}\, dx\Big)^{\frac{n-\alpha}{\gamma'}}\Big(\int_{\mathbb R^n}\tilde m\,dx\Big)^{\frac{2\gamma'-n-\alpha}{\gamma'}}}{\int_{\mathbb R^n}\tilde m(x)(K_{\alpha}*\tilde m)(x)\,dx}
\leq C_2(C_L,\gamma',n)<+\infty,
\end{align}
where we have used the following inequality
\begin{align*}
\int_{\mathbb R^n}\tilde m(x)(K_{\alpha}*\tilde m)(x)\,dx&=\int_{\mathbb R^n}\int_{\mathbb R^n}\frac{e^{-|x|}e^{-|y|}}{|x-y|^{n-\alpha}}\,dx\,dy\geq \iint_{\mathbb R^{2n}\cap \{|x-y|\leq 1\}}\frac{e^{-|x|}e^{-|y|}}{|x-y|^{n-\alpha}}\,dx\,dy\\
&\geq \iint_{\mathbb R^{2n}\cap \{|x-y|\leq 1\}\cap \{|x|\leq \frac{1}{4},|y|\leq \frac{1}{4}\}}e^{-|x|}e^{-|y|}\,dx\,dy=\tilde C(n).
\end{align*}


We next focus on the positive lower bound satisfied by $\Gamma_{\alpha}$ uniformly in $\alpha$. To show this, we argue by contradiction and assume
\begin{align}\label{contradiction-assumption}
\liminf_{\alpha\searrow (n-\gamma')} \Gamma_{\alpha}=0.
\end{align}
Lemma \ref{sect3-lemma32} implies there is a minimizer $(m_{\alpha},w_{\alpha})\in \mathcal A$ of problem \eqref{sect2-equivalence-scaling}.  Since \eqref{sect2-equivalence-scaling} is invariant under the scaling $s(t^{n}m(tx),t^{n+1}w(tx))$ for any $s>0$ and $t>0$, we normalize $m_{\alpha}$ to get
\begin{align}\label{normalize-alpha1-alpha}
\int_{\mathbb R^n} m_{\alpha}\, dx=\int_{\mathbb R^n} m_{\alpha}^{\frac{2n}{n+\alpha}}\, dx\equiv 1.
\end{align}
By using this equality and the Hardy-Littlewood-Sobolev inequality given in \eqref{eqHLS_2}, we obtain
\begin{align}\label{sect2-Hardy-Littlewood-Sobolev}
	0<\int_{\mathbb R^n} m_{\alpha}(x)(K_{\alpha}*m_{\alpha})(x)\,dx\leq C(n,\alpha)\bigg(\int_{\mathbb R^n} m_{\alpha}^{\frac{2n}{n+\alpha}}\,dx\bigg)^{\frac{n+\alpha}{n}}\leq C(n,\alpha),
\end{align}
where $C(n,\alpha)>0$ is the best constant. 
On the other hand, since $\alpha\in(n-\gamma',n)$, we follow the argument shown in \cite[Theorem 4.3]{LiebLoss} to get 
\begin{align}\label{eq3.170}
	\lim_{\alpha\searrow (n-\gamma')} C(n,\alpha)=C(n,n-\gamma')<+\infty.
\end{align}
Hence, 
\begin{align}\label{eq3.171}
	\liminf_{\alpha\searrow (n-\gamma')}\int_{\mathbb R^n} m_{\alpha}(x)(K_{\alpha}*m_{\alpha})(x)\,dx=:\tilde{C}(n,\gamma')\leq C(n,n-\gamma')<+\infty.
\end{align}
Then it follows from \eqref{contradiction-assumption}, \eqref{eq3.171} and  \eqref{sect2-equivalence-scaling} that, as $\alpha \searrow (n-\gamma'),$
\begin{align}\label{eq3.18}
\int_{\mathbb R^n}m_{\alpha}\bigg|\frac{w_{\alpha}}{m_{\alpha}}\bigg|^{\gamma'}\,dx\rightarrow 0.
\end{align}
Proceeding the same argument shown in Lemma 3.5 of \cite{cirant2024critical}, one finds $\Vert m_{\alpha}\Vert_{W^{1,\hat q}(\mathbb R^n)}\rightarrow 0.$  By using the Sobolev embedding theorem, we obtain 
\begin{align}\label{combine20241123}
  \Vert m_{\alpha}\Vert_{L^{\frac{n}{n-\gamma'}}(\mathbb R^n)}\rightarrow 0,\text{ as }\alpha\searrow (n-\gamma'). 
\end{align}
 On the other hand, the following interpolation inequality holds:
$$\Vert m_{\alpha}\Vert_{L^{\frac{2n}{n+\alpha}}(\mathbb R^n)}\leq \Vert m_{\alpha}\Vert^{1-\theta_{\alpha}}_{L^1(\mathbb R^n)}\Vert m_{\alpha}\Vert^{\theta_{\alpha}}_{L^{\frac{n}{n-\gamma'}}(\mathbb R^n)},$$
where $\theta_{\alpha}:=\frac{n-\alpha}{\gamma'}\in (0,1).$  With the help of (\ref{combine20241123}), we further get as $\alpha\searrow (n-\gamma'),$ $\theta_{\alpha}\nearrow 1$ and  
$$\Vert m_{\alpha}\Vert_{L^{\frac{2n}{n+\alpha}}(\mathbb R^n)}\rightarrow 0,$$
which reaches a contradiction to (\ref{normalize-alpha1-alpha}).
Thus, we have  $\exists C_1>0$ independent of $\alpha$ such that
\begin{align}\label{lemma33lowebounddesired20231027}
0<C_1\leq \Gamma_{\alpha}.
\end{align}
Finally, combining \eqref{lemma33lowebounddesired20231027} with \eqref{upperboundCalpha}, one finds \eqref{317inlemma33alphaC2} holds.  This completes the proof of the lemma.
\end{proof}

With the aid of the uniform boundedness of $\Gamma_{\alpha}$, we next establish the uniform $L^\infty$ bound of $m_{\alpha}$ as $\alpha\searrow(n-\gamma')$, which is
\begin{lemma}\label{lemma34uniformlinfboundmalpha}
Let $(u_{\alpha},m_{\alpha},\lambda_{\alpha})\in C^2(\mathbb R^n)\times W^{1,p}(\mathbb R^n)\times \mathbb R$, $\forall p>1$ be the solution of
\begin{align}\label{lemma34mualphaeq}
\left\{\begin{array}{ll}
-\Delta u+C_H|\nabla u|^{\gamma}+\lambda=-K_{\alpha }*m,&x\in\mathbb R^n,\\
-\Delta m-C_H \gamma\nabla\cdot(m|\nabla u|^{\gamma-2}\nabla u)=0,&x\in\mathbb R^n,\\
\int_{\mathbb R^n}m\, dx=M_{\alpha}.
\end{array}
\right.
\end{align}
Define $w_{\alpha}=-C_H\gamma m_{\alpha}|\nabla u_{\alpha}|^{\gamma-2}\nabla u_{\alpha}$.
 Assume that each $u_{\alpha}$ is bounded from below  and there exists a constant $C>0$ independent of $\alpha$ such that
\begin{align}\label{energypohoboundedness}
\limsup_{\alpha\searrow(n-\gamma')}\int_{\mathbb R^n}m_{\alpha}|\nabla u_{\alpha}|^{\gamma}\,dx\leq C,\ \ \lim_{\alpha\searrow(n-\gamma')}\int_{\mathbb R^n}m_{\alpha}\, dx=\lim_{\alpha\searrow(n-\gamma')} M_{\alpha}\leq C, \ \ \limsup_{\alpha\searrow(n-\gamma')}|\lambda_{\alpha}|\leq C,
\end{align}
then there is $C_1>0$ independent of $\alpha$ such that
\begin{align}\label{linfinityboundofmuniform}
\limsup_{\alpha\searrow(n-\gamma')}\Vert m_{\alpha}\Vert_{L^\infty(\mathbb R^n)}\leq C_1.
\end{align}

\end{lemma}
\begin{proof}
The proof is similar as shown in \cite[Lemma 4.3]{cirant2024critical}.
We proceed by contradiction and suppose that up to a subsequence,
\begin{align}\label{blowupassumption}
\mu_{\alpha}:=\Vert m_{\alpha}\Vert_{L^\infty(\mathbb R^n)}^{-\frac{1}{n}}\rightarrow 0 \ \ \text{~as~}\alpha\searrow(n-\gamma').
\end{align}
Now, we fix $0=u_{\alpha}(0)=\inf\limits_{x\in \mathbb R^n} u_{\alpha}(x)$ without loss of generality, as this is due to the fact that $u_{\alpha}$ is bounded from below.  Define
\begin{align}\label{blowupanalysisrescaling}
\bar u_{\alpha}:=\mu_{\alpha}^{\frac{2-\gamma}{\gamma-1}}u_{\alpha}(\mu_{\alpha}x)+1,\
\bar m_{\alpha}:=\mu_{\alpha}^n m_{\alpha}(\mu_{\alpha} x) ~\text{ and }~
\bar w_{\alpha}:=\mu_{\alpha}^{n+1}w_{\alpha}(\mu_{\alpha}x),
\end{align}
then, by (\ref{energypohoboundedness}) and (\ref{blowupassumption}), we obtain that up to a subsequence
\begin{align}\label{eq3.36}
\int_{\mathbb R^n}\bar m_{\alpha}\,dx=\int_{\mathbb R^n} m_{\alpha}\, dx=M_{\alpha},\
\int_{\mathbb R^n}\bar m_{\alpha}^{\frac{2n}{n+\alpha}}\, dx=\mu_{\alpha}^{\frac{n(n-\alpha)}{n+\alpha}}\int_{\mathbb R^n}m_{\alpha}^{\frac{2n}{n+\alpha}}\, dx\rightarrow 0 \ \ \text{as~}\alpha\searrow (n-\gamma'),
\end{align}
and
\begin{align}\label{338thevanishofmgradu}
\int_{\mathbb R^n}\bar m_{\alpha}|\nabla \bar u_{\alpha}|^{\gamma}\,dx=\mu_{\alpha}^{\gamma'}\int_{\mathbb R^n}m_{\alpha}|\nabla u_{\alpha}|^{\gamma}\, dx\rightarrow 0 \ \ \text{as~}\alpha\searrow (n-\gamma').
\end{align}
Recall the definition of $w_{\alpha}$, then we deduce from \eqref{MFG-L} and \eqref{blowupanalysisrescaling} that
\begin{align}\label{operatorbodyconverge0}
C_L\int_{\mathbb R^n}\bigg|\frac{\bar w_{\alpha}}{\bar m_{\alpha}}\bigg|^{\gamma'} \bar m_{\alpha}\,dx=(\gamma-1)C_H\int_{\mathbb R^n}\bar m_{\alpha}|\nabla\bar u_{\alpha}|^{\gamma}\,dx\rightarrow 0, \ \ \text{as~}\alpha\searrow (n-\gamma').
\end{align}
In light of \eqref{blowupassumption} and \eqref{blowupanalysisrescaling}, we infer that
\begin{align}\label{342malphainfequiv1always}
\Vert \bar m_{\alpha}\Vert_{L^\infty}\equiv 1.
\end{align}
This together with \eqref{eq3.36} implies that for any $q>\frac{2n}{n+\alpha}$,
\begin{align}\label{343mL1plusgammaovernnorm}
\int_{\mathbb R^n}\bar m_{\alpha}^{q}\, dx\leq \bigg(\int_{\mathbb R^n}\bar m_{\alpha}^{\frac{2n}{n+\alpha}}\, dx\bigg)\|\bar m_\alpha\|_{L^\infty(\mathbb \R^n)}^{q-\frac{2n}{n+\alpha}}\rightarrow 0 ~~\text{as}~~\alpha\searrow (n-\gamma').
\end{align}
On the other hand, invoking \eqref{blowupanalysisrescaling} and  \eqref{lemma34mualphaeq}, one can obtain that
\begin{align}\label{limitingproblembarubarm}
\left\{\begin{array}{ll}
-\Delta_x\bar u_{\alpha}+C_H|\nabla_x \bar u_{\alpha}|^{\gamma}+\lambda_{\alpha}\mu_{\alpha}^{\gamma'}=-\mu_{\alpha}^{(\gamma'-n+\alpha)} K_{\alpha}\ast\bar m_{\alpha},&x\in\mathbb R^n,\\
-\Delta_x \bar m_{\alpha}-C_H\gamma\nabla_x\cdot(\bar m_{\alpha}|\nabla_x \bar u_{\alpha}|^{\gamma-2}\nabla_x \bar u_{\alpha})=0,&x\in\mathbb R^n,\\
\int_{\mathbb R^n}\bar m_{\alpha}\, dx=M_{\alpha}.
\end{array}
\right.
\end{align}
It follows from \eqref{energypohoboundedness} and \eqref{blowupassumption} that
$\lambda_{\alpha}\mu_{\alpha}^{\gamma'}\rightarrow 0\ \ \text{as~} \alpha\searrow (n-\gamma').$
In addition, we claim that 
\begin{align}\label{convolution-terms-inftynormestimate}
0\leq \mu_{\alpha}^{\gamma'-n-\alpha}\Vert K_{\alpha}*\bar m_{\alpha}\Vert_{L^\infty(\mathbb \R^n)}\leq C,
\end{align}
where $C>0$ independent of $\alpha$. Indeed, by the definition of $K_{\alpha}$, we get
\begin{align}\label{convolution-terms-estimate-1}
|K_{\alpha}*\bar m_{\alpha}|\leq \int_{\mathbb R^n} \frac{|\bar m_{\alpha}(x-y)|}{|y|^{n-\alpha}}\, dy =\overbrace{\int_{B_{1}} \frac{|\bar m_{\alpha}(x-y)|}{|y|^{n-\alpha}}\, dy}^{I}+\overbrace{\int_{B^{c}_{1}} \frac{|\bar m_{\alpha}(x-y)|}{|y|^{n-\alpha}}\, dy}^{II}.
\end{align}
For $I$, we apply \eqref{342malphainfequiv1always} to get
\begin{align}\label{convolution-terms-estimate-2}
	I=\int_{B_{1}} \frac{|\bar m_{\alpha}(x-y)|}{|y|^{n-\alpha}}\,dy\leq \Vert \bar m_{\alpha}\Vert_{L^\infty}(\mathbb \R^n)\int_{B_{1}} \frac{1}{|y|^{n-\alpha}}\,dy=: C(n).
\end{align}
For $II$, taking into account the condition  \eqref{energypohoboundedness}, we get from \eqref{eq3.36} that 
\begin{align}\label{convolution-terms-estimate-3}
	II=\int_{B^{c}_{1}} \frac{|\bar m_{\alpha}(x-y)|}{|y|^{n-\alpha}}\,dy\leq \int_{B^{c}_{1}} |\bar m_{\alpha}(x-y)|\,dy \leq \Vert \bar m_{\alpha}\Vert_{L^1(\mathbb \R^n)}=M_{\alpha}\leq C.
\end{align}
Then, we conclude that \eqref{convolution-terms-inftynormestimate} holds by collecting \eqref{blowupassumption}, \eqref{convolution-terms-estimate-1}-\eqref{convolution-terms-estimate-3} and the fact $\gamma'\geq n-\alpha$. 
Hence, applying Lemma \ref{sect2-lemma21-gradientu} to the first equation in (\ref{limitingproblembarubarm}), one finds
\begin{align}\label{346gradualphaestimate}
\limsup_{\alpha\searrow (n-\gamma')}\Vert\nabla \bar u_{\alpha}\Vert_{L^\infty(\mathbb \R^n)}\leq C<\infty.
\end{align}
Noting that $\bar w_{\alpha}=-C_H\gamma\bar m_{\alpha}|\nabla \bar u_{\alpha}|^{\gamma-2}\nabla \bar u_{\alpha}$,
we deduce from \eqref{346gradualphaestimate} that
\begin{align}\label{350barwalphaLinflessC}
\limsup_{\alpha\searrow (n-\gamma')}\Vert \bar w_{\alpha}\Vert_{L^\infty(\mathbb \R^n)}\leq C<\infty.
\end{align}

Now, we turn our attention to H\"{o}lder estimates of $\bar m_{\alpha}$ and the proof of H\"{o}lder continuity of $\bar m_{\alpha}$ is the same as shown in Lemma 4.3 of \cite{cirant2024critical}.   
In fact, we obtain for some $\theta'\in(0,1),$
\begin{align}\label{malphaholdercontinuousuniformlyinalpha}
\Vert \bar m_{\alpha}\Vert_{C^{0,\theta'}(\mathbb R^n)}\rightarrow 0 \ \ \text{as~}\alpha\searrow (n-\gamma').
\end{align}
Assume that $x_{\alpha}$ is a maximum point of $\bar m_{\alpha}$, i.e.,  $\bar m_{\alpha}(x_{\alpha})=\Vert \bar m_{\alpha}\Vert_{L^\infty(\mathbb R^n)}=1.$  Then we deduce from (\ref{malphaholdercontinuousuniformlyinalpha}) that there exists $R_1>0$ independent of $\alpha$ such that
$|\bar m_{\alpha}(x)|\geq \frac{1}{2},\forall x\in B_{R_1}(x_{\alpha})$.
Thus,
\begin{align*}
\bigg(\frac{1}{2}\bigg)^{\frac{2n}{2n-\gamma'}}|B_{R_1}|\leq \int_{B_{R_1}(x_{\alpha})}\bar m_{\alpha}^{\frac{2n}{2n-\gamma'}}\, dx\leq \int_{\mathbb R^n}\bar m_{\alpha}^{\frac{2n}{2n-\gamma'}}\, dx,
\end{align*}
which contradicts \eqref{343mL1plusgammaovernnorm}, and the proof of the lemma is finished.
\end{proof}

With the aid of Lemma \ref{sect3-lemma32}, Lemma \ref{uniformlyboundC1CalphaC2} and Lemma \ref{lemma34uniformlinfboundmalpha}, we are able to show conclusions stated in Theorem \ref{thm11-optimal}, which are

\medskip

\textbf{Proof of Theorem \ref{thm11-optimal}:}
\begin{proof}
We first recall that, for 
any $M>0$ and $p\in(1,+\infty)$, $(\bar u_{\alpha,M},\bar m_{\alpha,M},\lambda_{\alpha,M})\in C^{2}(\mathbb R^n)\times W^{1,p}(\mathbb R^n)\times \mathbb R$  denotes the solution to system \eqref{eq-attained-sub}, and the pair $(\bar m_{\alpha,M},\bar w_{\alpha,M})$ with $\bar w_{\alpha,M}=-C_H\gamma\bar m_{\alpha,M}|\nabla\bar u_{\alpha,M}|^{\gamma-2}\nabla\bar u_{\alpha,M}$ is a minimizer of the minimization problem \eqref{sect3-Calpha-equivalent}, in which $\bar m_{\alpha,M}$ satisfies the estimate \eqref{eq3.2}. Now, we take $$M=M_{\alpha}:=e^{\frac{\gamma'+\alpha-n}{2\gamma'+\alpha-n}}\big[2\Gamma_{\alpha}\big]^{\frac{\gamma'}{2\gamma'+\alpha-n}}$$ in \eqref{eq-attained-sub}, then one can deduce from \eqref{sect3-318lambdam} that
$$S_{\alpha,M_{\alpha}}:=\Bigg[\frac{M_{\alpha}^{\frac{2\gamma'+\alpha-n}{\gamma'}}}{2\Gamma_{\alpha}}\Bigg]^{\frac{\gamma'}{\gamma'+\alpha-n}}\equiv e.$$
Moreover, we obtain that, up to a subsequence,  
\begin{align}\label{356Malphagammaalpharelation}
\frac{M_{\alpha}^{\frac{2\gamma'+\alpha-n}{\gamma'}}}{2\Gamma_{\alpha}}\rightarrow 1, \ \ \frac{2\gamma'+\alpha-n}{\gamma'}\rightarrow 1,\ \ \text{as} \ \ \alpha\searrow (n-\gamma').
\end{align}
Since $M=M_\alpha$ depends on $\alpha$, to emphasize the dependence of a solution on $\alpha$,
we will rewrite $(\bar m_{\alpha,M},\bar{w}_{\alpha,M}, \lambda_{\alpha,M})$ as $(\bar m_{\alpha,M_{\alpha}},\bar{w}_{\alpha,M_{\alpha}}, \lambda_{\alpha,M_{\alpha}})$.
Hence, we know from \eqref{eq-attained-sub} that $(\bar m_{\alpha,M_{\alpha}},\bar u_{\alpha,M_\alpha}, \bar{w}_{\alpha,M_{\alpha}}, \lambda_{\alpha,M_{\alpha}})$ satisfies
\begin{align}\label{eq3.50}
\left\{\begin{array}{ll}
-\Delta u+C_H|\nabla u|^{\gamma}+\lambda=-K_{\alpha}*m,&x\in\mathbb R^n,\\
\Delta m+C_H\gamma\nabla\cdot(m|\nabla u|^{\gamma-2}\nabla u)=0,&x\in\mathbb R^n,\\
 w=-C_H\gamma m|\nabla u|^{\gamma-2}\nabla u,\ \int_{\mathbb R^n}m\, dx=M_\alpha.
\end{array}
\right.
\end{align}

We can infer from Lemma \ref{uniformlyboundC1CalphaC2} that, up to a subsequence,
\begin{align*}
\Gamma_{\alpha}\rightarrow \bar \Gamma_{\alpha^*}:=\liminf\limits_{\alpha\searrow (n-r)} \Gamma_{\alpha}>0\ \ \text{ as}\ \ \alpha\searrow \alpha^*.
\end{align*}
In addition, invoking \eqref{356Malphagammaalpharelation} we have that $M_{\alpha}\rightarrow M_{\alpha^*}:=M^*$ \ as \ $\alpha\searrow (n-\gamma')$, where
\begin{align}\label{mstarlimitfinal}
M^*=2\bar \Gamma_{\alpha^*}, \ \alpha^*=n-\gamma'.
\end{align}
Moreover, due to the relation \eqref{sect3-318lambdam}, we obtain that, up to a subsequence,
\begin{align}\label{lambdalimitbehavior}
\lambda_{\alpha, M_{\alpha}}\rightarrow \lambda_{\alpha^*}:=-\frac{1}{M^*} \ as \ \alpha\searrow (n-\gamma')
\end{align}
and it follows from \eqref{sect3-poho-final} that
\begin{align}\label{eq3.58}
\int_{\mathbb R^n}{\bar m}_{\alpha,M_{\alpha}}\,dx=M_{\alpha}\rightarrow M^*>0,\ \ \int_{\mathbb R^n}{\bar m}_{\alpha,M_{\alpha}}(x)(K_{\alpha}\ast{\bar m}_{\alpha,M_{\alpha}})(x)\, dx\rightarrow  2, \ \ C_L\int_{\mathbb R^n}\bar m_{\alpha,M_{\alpha}}\bigg|\frac{\bar w_{\alpha,M_{\alpha}}}{\bar m_{\alpha,M_{\alpha}}}\bigg|^{\gamma'}\,dx\rightarrow 1.
\end{align}
Applying Lemma \ref{lemma34uniformlinfboundmalpha}, we derive from  \eqref{lambdalimitbehavior} and \eqref{eq3.58} that
\begin{equation}\label{eq3.54}
\limsup_{\alpha\searrow (n-\gamma')}\Vert \bar m_{\alpha, M_\alpha}\Vert_{L^\infty(\mathbb R^n)}<\infty.
\end{equation}
Then, by using the estimate \eqref {usolutiongradientestimatepre} with $b=0$ from Lemma \ref{lowerboundVkgenerallemma22}, we have
\begin{align}\label{364supnablauupper}
\limsup_{\alpha\searrow (n-\gamma')}\Vert \nabla \bar u_{\alpha,M_{\alpha}}\Vert_{L^\infty}<\infty,
\end{align}
which, together with the definition of $\bar w_{\alpha,M_{\alpha}}$, yields
\begin{align}\label{361sandwichbefore}
\limsup_{\alpha\searrow (n-\gamma')}\Vert \bar w_{\alpha,M_{\alpha}}\Vert_{L^\infty}<\infty.
\end{align}
Proceeding the arguments similar
as those used in the proof of \eqref{malphaholdercontinuousuniformlyinalpha}, we collect \eqref{eq3.58}-\eqref{361sandwichbefore} to obtain  that
\begin{align}\label{eq3.60}
\limsup_{\alpha\searrow (n-\gamma')}\Vert \bar m_{\alpha,M_{\alpha}}\Vert_{W^{1,q}(\mathbb R^n)}<+\infty, \ \forall \ q>n, 
\end{align}
and thus 
\begin{align}\label{eq3.60_0}
\limsup_{\alpha\searrow (n-\gamma')}\Vert\bar m_{\alpha,M_{\alpha}}\Vert_{C^{0,\tilde\theta}(\mathbb R^n)}<\infty ~\text{ for some}~\tilde\theta\in(0,1).
\end{align}
We may assume that $\bar u_{\alpha,M_\alpha}(0)=0=\inf_{x\in \mathbb R^n}\bar u_{\alpha,M_\alpha}(x)$ due to the fact that $\bar u_{\alpha,M_\alpha}\in C^2(\mathbb R^n)$ is bounded from below. Hence, by the first equation of \eqref{eq3.50}, one has $$(K_{\alpha}*\bar m_{\alpha,M_\alpha})(0)\geq -\lambda_{\alpha, M_{\alpha}}>0,$$ which together with \eqref{eq3.54} and a similar argument as used in \cite[Lemma 4.1]{bernardini2023ergodic} to obtain that there are $\delta_1$ and a large $R>0$ independent of $\alpha$ such that,
\begin{equation}\label{eq3.62}
\int_{|x|\leq R}{\bar m_{\alpha,M_{\alpha}}(x)}\,dx>\frac{\delta_1}{2}>0.
\end{equation}

Now, we rewrite the first equation of \eqref{eq3.50} as
\begin{equation}\label{eq3.56}
     -\Delta \bar u_{\alpha,M_\alpha}=-C_H|\nabla \bar u_{\alpha,M_\alpha}|^{r'}+h_\alpha(x),
\end{equation}
 where $h_\alpha(x):=-\lambda_{\alpha, M_\alpha}-K_{\alpha}*\bar m_{\alpha,M_\alpha},\ x\in\mathbb R^n$.
By performing the same procedure shown in \eqref{convolution-terms-estimate-1}, one can see that 
\begin{equation*}
\big|(K_{\alpha}*m_{\alpha,M_\alpha})(x)\big|\leq C
\end{equation*}
 for some $C>0$ independent of $\alpha$.
Then, we apply the standard elliptic regularity to \eqref{eq3.56} and obtain
 \begin{equation}\label{eq4.54}
     \|\bar u_{\alpha,M_\alpha}\|_{C^{2,\theta}(B_{R}(0))} \leq C_{\theta,\bar R}<\infty\ \text{ for some $\theta\in(0,1)$,}
 \end{equation}
 where $0<R<\bar R.$
Performing the standard diagonal procedure, we take the limit and apply Arzel\`{a}-Ascoli theorem, \eqref{eq3.60} and \eqref{eq4.54} to obtain that there exists  $(m_{\alpha^*},u_{\alpha^*})\in W^{1,p}(\mathbb R^n)\times C^2(\mathbb R^n)$  such that
\begin{align}\label{269locallyuniform}
 \bar m_{\alpha,M_{\alpha}}\rightharpoonup m_{\alpha^*} \text{ in }W^{1,p}(\mathbb R^n), \text{ and }\bar  u_{\alpha,M_{\alpha}}\rightarrow  u_{\alpha^*} \ \ \text{in  } C^2_{\rm loc}(\mathbb R^n),  \text{ as }\alpha\searrow (n-\gamma').
\end{align}
Combining  \eqref{eq3.50}, \eqref{lambdalimitbehavior} and \eqref {269locallyuniform}, we conclude that
$(m_{\alpha^*},u_{\alpha^*})\in W^{1,p}(\mathbb R^n)\times C^2(\mathbb R^n)$ satisfies
\begin{align}\label{limitingproblemminimizercritical}
\left\{\begin{array}{ll}
-\Delta u+C_H|\nabla u|^{\gamma}-\frac{1}{M^*}=-K_{\alpha^*}*m,&x\in\mathbb R^n,\\
-\Delta m-\gamma C_H\nabla\cdot(m|\nabla u|^{\gamma-2}\nabla u)=0,&x\in\mathbb R^n,\\
w=-C_H\gamma m|\nabla u|^{\gamma-2}\nabla u.
\end{array}
\right.
\end{align}
In light of \eqref{eq3.62} and Fatou's lemma, we have
\begin{equation}\label{eq3.67}
\int_{\mathbb R^n}m_{\alpha^*}\,dx=\tilde{M} \in(0,M^*].
\end{equation}
 Moreover, by Lemma \ref{mdecaylemma}, we obtain that there exists some $\kappa,C>0$ such that $m_{\alpha^*}(x)<Ce^{-\kappa|x|}$.  In addition, by using \eqref{364supnablauupper}, we get $\Vert \nabla u_{\alpha^*}\Vert_{L^\infty}<\infty$.  It then follows from Lemma \ref{poholemma} that
\begin{align}\label{370criticalpohoidentity}
C_L\int_{\mathbb R^n}\Bigg|\frac{w_{\alpha^*}}{m_{\alpha^*}}\Bigg|^{\gamma'}m_{\alpha^*}\,dx=\frac{1}{2}\int_{\mathbb R^n} m_{\alpha^*}(x)(K_{\alpha^*}*m_{\alpha^*})(x)\, dx.
\end{align}

Next, we discuss the relationship between $\bar\Gamma_{\alpha^*}:=\liminf\limits_{\alpha\searrow(n-\gamma')}\Gamma_{\alpha}$ and $\Gamma_{\alpha^*}$ with $\alpha^*=(n-\gamma').$  We claim that
\begin{align}\label{ourclaimtwolimit}
\bar\Gamma_{\alpha^*}=\Gamma_{(n-\gamma')}.
\end{align}
Indeed, we first utilize Lemma \ref{sect3-lemma32} and obtain
\begin{align}\label{337beforecomparecbarc}
&\Gamma_{\alpha}=G_{\alpha}(\bar m_{\alpha,M_{\alpha}},\bar w_{\alpha,M_{\alpha}})\\
&=G_{(n-\gamma')}(\bar m_{\alpha,M_{\alpha}},\bar w_{\alpha,M_{\alpha}})\frac{\Big(C_L\int_{\mathbb R^n}\bar m_{\alpha,M_{\alpha}}\Big|\frac{\bar w_{\alpha,M_{\alpha}}}{\bar m_{\alpha,M_{\alpha}}}\Big|^{\gamma'}\, dx\Big)^{\frac{n-\alpha}{\gamma'}}\Big(\int_{\mathbb R^n}\bar m_{\alpha,M_{\alpha}}\,dx\Big)^{\frac{2\gamma'+\alpha-n}{\gamma'}}}{\Big(C_L\int_{\mathbb R^n}\bar m_{\alpha,M_{\alpha}}\Big|\frac{\bar w_{\alpha,M_{\alpha}}}{\bar m_{\alpha,M_{\alpha}}}\Big|^{\gamma'}\, dx\Big)\Big(\int_{\mathbb R^n}\bar m_{\alpha,M_{\alpha}}\,dx\Big)}\cdot\frac{\int_{\mathbb R^n}\bar m_{\alpha,M_{\alpha}}(x)(K_{(n-\gamma')}*\bar m_{\alpha,M_{\alpha}})\,dx}{{\int_{\mathbb R^n}\bar m_{\alpha,M_{\alpha}}(x)(K_{\alpha}*\bar m_{\alpha,M_{\alpha}}})\,dx}\nonumber\\
&\geq \Gamma_{(n-\gamma')}\frac{\Big(C_L\int_{\mathbb R^n}\bar m_{\alpha,M_{\alpha}}\Big|\frac{\bar w_{\alpha,M_{\alpha}}}{\bar m_{\alpha,M_{\alpha}}}\Big|^{\gamma'}\, dx\Big)^{\frac{n-\alpha}{\gamma'}}\Big(\int_{\mathbb R^n}\bar m_{\alpha,M_{\alpha}}\,dx\Big)^{\frac{2\gamma'+\alpha-n}{\gamma'}}}{\Big(C_L\int_{\mathbb R^n}\bar m_{\alpha,M_{\alpha}}\Big|\frac{\bar w_{\alpha,M_{\alpha}}}{\bar m_{\alpha,M_{\alpha}}}\Big|^{\gamma'}\, dx\Big)\Big(\int_{\mathbb R^n}\bar m_{\alpha,M_{\alpha}}\,dx\Big)}\cdot\frac{\int_{\mathbb R^n}\bar m_{\alpha,M_{\alpha}}(x)(K_{(n-\gamma')}*\bar m_{\alpha,M_{\alpha}})\,dx}{{\int_{\mathbb R^n}\bar m_{\alpha,M_{\alpha}}(x)(K_{\alpha}*\bar m_{\alpha,M_{\alpha}}})\,dx}.
\end{align}
Then, we derive
from \eqref{eq3.58} that, as $\alpha\searrow (n-\gamma')$,
\begin{align*}
\frac{\Big(C_L\int_{\mathbb R^n}\bar m_{\alpha,M_{\alpha}}\Big|\frac{\bar w_{\alpha,M_{\alpha}}}{\bar m_{\alpha,M_{\alpha}}}\Big|^{\gamma'}\, dx\Big)^{\frac{n-\alpha}{\gamma'}}\Big(\int_{\mathbb R^n}\bar m_{\alpha,M_{\alpha}}\,dx\Big)^{\frac{2\gamma'+\alpha-n}{\gamma'}}}{\Big(C_L\int_{\mathbb R^n}\bar m_{\alpha,M_{\alpha}}\Big|\frac{\bar w_{\alpha,M_{\alpha}}}{\bar m_{\alpha,M_{\alpha}}}\Big|^{\gamma'}\, dx\Big)\Big(\int_{\mathbb R^n}\bar m_{\alpha,M_{\alpha}}\,dx\Big)}\cdot\frac{\int_{\mathbb R^n}\bar m_{\alpha,M_{\alpha}}(x)(K_{(n-\gamma')}*\bar m_{\alpha,M_{\alpha}})\,dx}{{\int_{\mathbb R^n}\bar m_{\alpha,M_{\alpha}}(x)(K_{\alpha}*\bar m_{\alpha,M_{\alpha}}})\,dx}\rightarrow 1.
\end{align*}
Moreover, one takes the limit in \eqref{337beforecomparecbarc} to get
\begin{align}\label{gammaalphastar}
\bar \Gamma_{\alpha^*}:=\liminf_{\alpha\searrow (n-\gamma')}\Gamma_{\alpha}\geq \Gamma_{(n-\gamma')}.
\end{align}
To complete the proof of our claim, it suffices to prove that the "=" holds in \eqref{gammaalphastar}. Suppose the contrary that $\Gamma_{(n-\gamma')}<\bar \Gamma_{\alpha^*}$, then by the definition of $\Gamma_{(n-\gamma')},$  we get that there exists $(\hat m,\hat w)\in \mathcal A$ given in \eqref{sect2-equivalence-scaling} such that
\begin{align}\label{340limit1}
G_{(n-\gamma')}(\hat m,\hat w)\leq \Gamma_{(n-\gamma')}+\delta<\Gamma_{(n-\gamma')}+2\delta<\bar \Gamma_{\alpha^*},
\end{align}
where $\delta>0$ is sufficiently small.
On the other hand, by the definition of $\Gamma_{\alpha},$  one finds
\begin{align}\label{342limit2}
G_{(n-\gamma')}(\hat m,\hat w)=&G_{\alpha}(\hat m,\hat w)\frac{\Big(C_L\int_{\mathbb R^n}\hat m\Big|\frac{\hat w}{\hat m}\Big|^{\gamma'}\, dx\Big)\Big(\int_{\mathbb R^n}\hat m\,dx\Big)}{{\Big(C_L\int_{\mathbb R^n}\hat  m\Big|\frac{\hat w}{\hat m}\Big|^{\gamma'}\, dx\Big)^{\frac{n-\alpha}{\gamma'}}\Big(\int_{\mathbb R^n}\hat m\,dx\Big)^{\frac{2\gamma'+\alpha-n}{\gamma'}}}}\cdot\frac{\int_{\mathbb R^n}{\hat m(x)(K_{\alpha}*\hat m)(x)}\,dx}{\int_{\mathbb R^n}\hat m(x)(K_{(n-\gamma')}*\hat m)(x)\,dx}\nonumber\\
\geq& \Gamma_{\alpha}\frac{\Big(C_L\int_{\mathbb R^n}\hat m\Big|\frac{\hat w}{\hat m}\Big|^{\gamma'}\, dx\Big)\Big(\int_{\mathbb R^n}\hat m\,dx\Big)}{{\Big(C_L\int_{\mathbb R^n}\hat  m\Big|\frac{\hat w}{\hat m}\Big|^{\gamma'}\, dx\Big)^{\frac{n-\alpha}{\gamma'}}\Big(\int_{\mathbb R^n}\hat m\,dx\Big)^{\frac{2\gamma'+\alpha-n}{\gamma'}}}}\cdot\frac{\int_{\mathbb R^n}{\hat m(x)(K_{\alpha}*\hat m)(x)}\,dx}{\int_{\mathbb R^n}\hat m(x)(K_{(n-\gamma')}*\hat m)(x)\,dx}.
\end{align}
Since
\begin{align*}
\frac{\Big(C_L\int_{\mathbb R^n}\hat m\Big|\frac{\hat w}{\hat m}\Big|^{\gamma'}\, dx\Big)\Big(\int_{\mathbb R^n}\hat m\,dx\Big)}{{\Big(C_L\int_{\mathbb R^n}\hat  m\Big|\frac{\hat w}{\hat m}\Big|^{\gamma'}\, dx\Big)^{\frac{n-\alpha}{\gamma'}}\Big(\int_{\mathbb R^n}\hat m\,dx\Big)^{\frac{2\gamma'+\alpha-n}{r}}}}\cdot\frac{\int_{\mathbb R^n}{\hat m(x)(K_{\alpha}\ast\hat m)(x)}\,dx}{\int_{\mathbb R^n}\hat m(x)(K_{(n-\gamma')}\ast\hat m)(x)\,dx}\rightarrow 1 \ \ \text{~as~}\alpha\searrow (n-\gamma'),
\end{align*}
then we can pass a limit in \eqref{340limit1} and \eqref{342limit2} to get
\begin{align*}
\bar\Gamma_{\alpha^*}=\liminf_{\alpha\searrow (n-\gamma')}\Gamma_{\alpha}\leq \Gamma_{\alpha\searrow (n-\gamma')}+\delta<\Gamma_{\alpha\searrow (n-\gamma')}+2\delta\leq \liminf_{\alpha\searrow (n-\gamma') }\Gamma_{\alpha},
\end{align*}
which reaches a contradiction.  Hence, the claim holds, i.e. $\bar \Gamma_{\alpha^*}=\Gamma_{n-\gamma'}$.

Next, we prove $(m_{\alpha^*},w_{\alpha^*})\in \mathcal A.$
  Since $(m_{\alpha^*},w_{\alpha^*})$ solves (\ref{limitingproblemminimizercritical}) and $m_{\alpha^*}\in C^{0,\theta}(\mathbb R^n)$ with $\theta\in(0,1)$, we conclude from (\ref{269locallyuniform}) and Lemma \ref{sect2-lemma21-gradientu} that $u_{\alpha^*}\in C^1(\mathbb R^n).$  Then by standard elliptic estimates, the boundedness of $\Vert\nabla u_{\alpha^*}\Vert_{L^\infty}$ and the exponentially decaying property of $m_{\alpha^*}$, one can prove that $(m_{\alpha^*},w_{\alpha^*}) \in \mathcal A$.

Finally, it follows from \eqref{eq3.58} and \eqref{gammaalphastar} that
\begin{align}\label{combining1strongconverge}
\liminf_{\alpha\searrow (n-\gamma')}\Gamma_{\alpha}=\frac{1}{2}M^*=\Gamma_{n-\gamma'},
\end{align}
where $M^*$ is given in \eqref{mstarlimitfinal}.  Then, by the fact $(m_{\alpha^*},w_{\alpha^*}) \in \mathcal A$, we deduce from \eqref{eq3.67}, (\ref{370criticalpohoidentity}) and (\ref{combining1strongconverge}) that  
\begin{align}\label{combining1strongconverge2}
\Gamma_{(n-\gamma')}=\frac{1}{2}M^*\leq \frac{\bigg(C_L\int_{\mathbb R^n}\Big|\frac{w_{\alpha^*}}{m_{\alpha^*}}\Big|^{\gamma'}m_{\alpha^*}\,dx\bigg)\Big(\int_{\mathbb R^n}m_{\alpha^*}\, dx\Big)}{\int_{\mathbb R^n}{m_{\alpha^*}(x)(K_{\alpha}*m_{\alpha^*})(x)}\,dx}=\frac{1}{2}\tilde{M} \leq \frac{1}{2}M^*,
\end{align}
which shows $(m_{\alpha^*},w_{\alpha^*}) \in \mathcal A$ is a minimizer of $\Gamma_{(n-\gamma')}$ and 
\begin{align*}
\int_{\mathbb R^n}m_{\alpha^*}\,dx=M^* ~\text{ and }~\bar  m_{\alpha,M_{\alpha}}\rightarrow m_{\alpha^*} \ \text{in~} L^1(\mathbb R^n) \text{ as }\alpha\searrow (n-\gamma').
\end{align*}
These facts together with \eqref{limitingproblemminimizercritical} indicate \eqref{limitingproblemminimizercritical0} holds.  Now, we finish The proof of Theorem \ref{thm11-optimal}.

\end{proof}

As shown in Theorem \ref{thm11-optimal}, we have obtained the existence of ground states to potential-free MFG systems under the mass critical exponent case, which is the Gagliardo-Nirenberg type's inequality. In next section, we focus on the proof of Theorem \ref{thm11}.
\section{Existence of Ground States: Coercive Potential MFGs}\label{sect4-criticalmass}
In this section, we shall discuss the existence of minimizers to problem \eqref{ealphacritical-117}. To this end, we have to perform the regularization procedure on \eqref{41engliang} since when $\gamma'<n$, the $m$-component enjoys the worse regularity. In detail, we first consider the following auxiliary minimization problem
\begin{align}\label{Eepsilonminimizerproblem1}
e_{\epsilon,M}:=\inf_{(m,w)\in\mathcal A_M}\mathcal E_{\epsilon}(m,w),
\end{align}
where ${\mathcal A}_M$ is given by \eqref{mathcalAMbegining}
and
\begin{align}\label{approxenergy}
\mathcal E_{\epsilon}(m,w):=C_L\int_{\mathbb R^n}\Big|\frac{w}{m}\Big|^{\gamma'}m\,dx+\int_{\mathbb R^n}V(x)m\,dx-\frac{1}{2}\int_{\mathbb R^n} \bigg\{ m(x)(K_{(n-\gamma')}\ast m)(x)\bigg\}\ast\eta_{\epsilon}\,dx,
\end{align}
 and $\eta_{\epsilon}\geq 0$ is the standard mollifier with
$$\int_{\mathbb R^n}\eta_{\epsilon}\,dx=1, \ \ \text{supp}(\eta_{\epsilon})\subset B_{\epsilon}(0),$$
for $\epsilon>0$ is sufficiently small. With the regularized energy \eqref{approxenergy}, we are able to study the existence of minimizers to \eqref{ealphacritical-117} by taking the limit. The crucial step in this procedure, as discussed in \cite{cesaroni2018concentration}, is the uniformly boundedness of $m_{\epsilon}$ in $L^{\infty}$, in which $(m_{\epsilon},w_{\epsilon})$ is assumed to be a minimizer of (\ref{approxenergy}).
 
Before proving Theorem \ref{thm11}, we collect some vital result shown in Section \ref{sect3-optimal}, which is 
\begin{align}\label{optimalinequality415}
\int_{\mathbb R^n} m(x)(K_{(n-\gamma')}*m)(x)\,dx \leq \frac{2C_L}{M^*}\bigg(\int_{\mathbb R^n}\big|\frac{w}{m}\big|^{\gamma'}m\,dx\bigg)\bigg(\int_{\mathbb R^n}m\,dx\bigg),~~\forall (m,w)\in\mathcal A,
\end{align}
where $\mathcal A$ is given by \eqref{mathcalA-equivalence} and $M^*$ is defined by \eqref{Mstar-critical-mass}.

Then, we shall first prove energy $\mathcal E(m,w)$ given by \eqref{41engliang}
has a minimizer $(m,w)\in \mathcal K_{M}$ if and only if $M<M^*$, where  $\mathcal K_{M}$ is defined by \eqref{constraint-set-K}.  Next, we show that there exists $(u,\lambda)\in C^2(\mathbb R^n)\times \mathbb R$ such that $(m,u,\lambda)\in W^{1,p}(\mathbb R^n)\times C^2(\mathbb R^n)\times \mathbb R$ is a solution to \eqref{125potentialfreesystem} when $V$ is assumed to satisfy \eqref{V2mainasumotiononv} when $\gamma'>1.$ 
Following the procedures discussed above, we are able to prove conclusions stated in Theorem \ref{thm11}.  We would like to remark that with (\ref{V2mainassumption_2}) in assumption (V2) imposed on potential $V$, the condition $\int_{\mathbb R^n}|x|^bm\,dx<+\infty$ in (\ref{mathcalAMbegining}) must be satisfied for any minimizer.  With this assumption, Gagliardo-Nirenberg type's inequality (\ref{optimalinequality415}) is valid. Next, we state some crucial propositions and lemmas, which will be used in the proof of Theorem \ref{thm11}, as follows:

\begin{lemma}\label{lem4.1}
Let $$\mathcal{W}_{p,V}:=\bigg\{m \big|\ m\in W^{1,p}(\mathbb R^n)\cap L^1(\mathbb R^n) \text{ and }\int_{\mathbb R^n}V(x)|m|\,dx<\infty\bigg\}.$$
 Assume that $0\leq V(x)\in L_{\rm loc}^\infty(\mathbb R^n)$ with $\liminf\limits_{|x|\to\infty}V(x)=\infty$.
Then, the embedding $\mathcal{W}_{p,V}\hookrightarrow L^q(\mathbb R^n)$ is compact for any $1\leq q<p^*$, where $p^*=\frac{np}{n-p}$ if $1\leq p<n$ and $p^*=\infty$ if $p\geq n$. 
\end{lemma}
\begin{proof}
See \cite[Theorem 2.1]{bartsch1995} or \cite[Theorem XIII.67]{reed2012methods}.
\end{proof}

In light of $\gamma'<n,$ we establish the following lemma for the uniformly boundedness of $\Vert m_{\epsilon}\Vert_{L^{\infty}}:$

{
\begin{lemma}\label{blowupanalysismlinfboundcritical}
Suppose that $V(x)$ is locally H\"{o}lder continuous and satisfies \eqref{V2mainasumotiononv}.
Let  $(u_{k},\lambda_{k},m_{k})\in C^2(\mathbb R^n)\times \mathbb R\times (L^1(\mathbb R^n)\cap L^{\frac{2n}{n+\alpha^*}}(\mathbb R^n)) $ be solutions to the following systems
\begin{align}\label{eq-potential-newest}
\left\{\begin{array}{ll}
-\Delta u_k+C_H|\nabla u_k|^{\gamma}+\lambda_k=V-g_k[m_k], &x\in\mathbb R^n,\\
\Delta m_k+C_H\gamma\nabla\cdot(m_k|\nabla u_k|^{\gamma-2}\nabla u_k)=0, &x\in\mathbb R^n,\\
\int_{\mathbb R^n}m_k\,dx=M,
\end{array}
\right.
\end{align}
where $\alpha^*=n-\gamma'$ with $1<\gamma'< n$, $g_k: L^1(\mathbb \R^n) \mapsto L^1(\mathbb \R^n)$ with $\theta\in(0,1)$ satisfies for all $m\in L^p(\mathbb R^n)$, $p\in[1,\infty]$ and $k\in \mathbb N$,
\begin{equation}\label{gkm-newest-estimate-crucial2024207}
\|g_k[m]\|_{L^p(\mathbb R^n)}\leq \mathrm{K}\bigg(\|m^{\frac{n-\alpha^*}{n+\alpha^*}}\|_{L^p(\mathbb R^n)}+1\bigg) \ \text{ for some} \
  \mathrm{K}>0,
\end{equation}
and
\begin{equation}\label{gkm-newest-estimate-crucial202420711}
\|g_k[m]\|_{L^p(B_R(x_0))}\leq \mathrm{K}\bigg(\|m^{\frac{n-\alpha^*}{n+\alpha^*}}\|_{L^p(B_{2R}(x_0))}+1\bigg) \ \text{ for any }R>0\text{ and }x_0\in \mathbb R^n.
\end{equation}
Assume that
\begin{align}\label{assumptionsincrucialemmanewest}
\sup_k\|m_k\|_{L^1(\mathbb R^n)}<\infty,~~\sup_k\|m_k\|_{L^{\frac{2n}{n+\alpha^*}}(\mathbb R^n)}<\infty,~~\sup_k\int_{\mathbb R^n}Vm_k\,dx<\infty,~~\sup_k|\lambda_k|<\infty,
\end{align}
and for all $k$, $u_k$ is bounded from below uniformly.  Then we have
\begin{equation}\label{uniformlyboundmkinlinfty}
\limsup_{k\to\infty}\|m_k\|_{L^\infty(\mathbb \R^n)}<\infty.
\end{equation}
\end{lemma}
}
\begin{proof}
By slightly modifying the argument shown in \cite[Lemma 5.2]{cirant2024critical}, we finish the proof of this lemma.
\end{proof}

With the preliminary results shown above, we now begin the proof of Theorem \ref{thm11}. 

\medskip

\textbf{Proof of Theorem \ref{thm11}:}
\begin{proof}
We first prove the Conclusion (i) in Theorem \ref{thm11}.  To this end, we focus on the auxiliary problem \eqref{Eepsilonminimizerproblem1}. 
Invoking the Young's inequality for convolution and the property of mollifier, one finds 
\begin{align}\label{inlightofsect4}
\int_{\mathbb R^n}m(x)(K_{(n-\gamma')}\ast m)(x)\,dx\geq \int_{\mathbb R^n}\bigg(\Big[m(K_{(n-\gamma')}\ast m)\Big]\ast \eta_{\epsilon}\bigg)(x)\,dx \overset{\epsilon\to0^+}\longrightarrow \int_{\mathbb R^n}m(x)(K_{(n-\gamma')}\ast m)(x)\,dx,
\end{align}
for any  $m\in L^{\frac{2n}{2n-\gamma'}}(\mathbb R^n)$. Here, we have used the following  Hardy-Littlewood-Sobolev inequality  
\begin{align}\label{H-L-S2}
\bigg |\int_{\mathbb R^n}m(x)(K_{(n-\gamma')}\ast m)(x)\,dx \bigg| \leq C(n,\gamma')\Vert m\Vert^{2}_{ L^{\frac{2n}{2n-\gamma'}}(\mathbb R^n)},\ \ \forall m\in L^{\frac{2n}{2n-\gamma'}}(\mathbb R^n).
\end{align}
As a consequence, in light of \eqref{optimalinequality415}  and \eqref{inlightofsect4}, we get
\begin{align}\label{418cruciallowerbound}
\mathcal E_{\epsilon}(m,w)\geq\mathcal E(m,w)\geq \frac{1}{2}\bigg(\frac{M^*}{M}-1\bigg)\int_{\mathbb R^n}m(x)(K_{(n-\gamma')}\ast m)(x)\,dx +\int_{\mathbb R^n}V(x)m\,dx.
\end{align}

Next, we show that the minimization problem \eqref{Eepsilonminimizerproblem1} is attainable. We first show that there exists $C>0$ independent of  $\epsilon$ such that
\begin{equation}\label{eq4.24_0}
e_{\epsilon, M}<C<+\infty,
\end{equation}
where $e_{\epsilon, M}$ is given by \eqref{Eepsilonminimizerproblem1}. 
Indeed, choosing $$(\hat m,\hat w):=\left(\frac{\|e^{-|x|}\|_{L^1(\mathbb R^n)}}{M} e^{-|x|}, \frac{\|e^{-|x|}\|_{L^1(\mathbb R^n)}}{M} \frac{x}{|x|}e^{-|x|}\right)\in\mathcal K_{M}, $$
 one can find
$$e_{\epsilon, M}\leq C_L\int_{\mathbb R^n}\Big|\frac{\hat w}{\hat m}\Big|^{\gamma'}\hat m\,dx+\int_{\mathbb R^n}V(x) \hat m\,dx<+\infty,$$
which indicates that \eqref{eq4.24_0} holds.  Let $(m_{\epsilon,k},w_{\epsilon,k})\in\mathcal K_M$ be a minimizing sequence of problem \eqref{Eepsilonminimizerproblem1}, then we have from \eqref{eq4.24_0} that there exists $C>0$ independent of $\epsilon$ such that
\begin{equation}\label{eq4.24}
\lim_{k\to\infty}\mathcal{E}_{\epsilon}(m_{\epsilon,k},w_{\epsilon,k})=e_{\epsilon, M}<C<+\infty.
\end{equation}
Moreover, it follows from \eqref{optimalinequality415}, \eqref{418cruciallowerbound}, \eqref{eq4.24} and the fact $M<M^*$ that
\begin{align}\label{suanzitouguji}
\sup_{k\in\mathbb N^+}\int_{\mathbb R^n}m_{\epsilon,k}(x)\big (K_{(n-\gamma')}\ast m_{\epsilon,k}\big)(x)\,dx\leq C<+\infty,
\end{align}
and
\begin{align}\label{suanzitouguji1}
~~\sup_{k\in\mathbb N^+}\int_{\mathbb R^n}\bigg(\bigg|\frac{w_{\epsilon,k}}{m_{\epsilon,k}}\bigg|^{\gamma'}m_{\epsilon,k}+V(x)m_{\epsilon,k}\bigg)\,dx\leq C<+\infty,
\end{align}
where $C>0$ is  independent of  $\epsilon$.  The subsequent argument for proving Conclusion (i) is similar as shown in the proof of Theorem 1.3 in \cite{cirant2024critical}.  In fact, with the aid of the key Lemma \ref{lemma21-crucial}, we obtain from (\ref{suanzitouguji}) that
\begin{align}\label{411mepsilonk}
\sup_{k\in\mathbb N^+}\Vert m_{\epsilon,k}\Vert_{W^{1,\hat q}(\mathbb R^n)}\leq C<+\infty \ \text{ and } \ \sup_{k\in\mathbb N^+}\Vert w_{\epsilon,k}\Vert_{L^{p}(\mathbb R^n)}\leq C<+\infty,\ \ \text{for any }p\in[1,\hat q],
\end{align}
where $\hat q$ is defined by \eqref{hatqconstraint} and $C>0$ is some constant independent of $\epsilon$. As a consequence, there exists $(m_{\epsilon},w_{\epsilon})\in W^{1,\hat q}(\mathbb R^n)\times  L^{\hat q}(\mathbb R^n)$ such that
\begin{align}\label{convergencesect4firstone}
(m_{\epsilon,k},w_{\epsilon,k})\overset{k}\rightharpoonup (m_{\epsilon},w_{\epsilon}) \text{~in~}  W^{1,\hat q}(\mathbb R^n)\times L^{\hat q}(\mathbb R^n).
\end{align}
In light of the assumption (V1), $\lim\limits_{|x|\rightarrow \infty}V(x)=+\infty$, given in Subsection \ref{mainresults11}, one can deduce from Lemma \ref{lem4.1} that
\begin{align}\label{convergencesect4secondone}
m_{\epsilon,k}\overset{k}\rightarrow m_{\epsilon} \text{~in~} L^1(\mathbb R^n)\cap L^{\frac{2n}{2n-\gamma'}}(\mathbb R^n) .
\end{align}
Therefore, up to a subsequence,
\begin{equation}\label{eq4.26}
\int_{\mathbb R^n}\bigg(\Big[m_{\epsilon,k}\big (K_{(n-\gamma')}\ast m_{\epsilon,k}\big)\Big]\ast \eta_{\epsilon}\bigg)(x)\,dx\overset{k}\rightarrow \int_{\mathbb R^n}\bigg(\Big[m_{\epsilon}\big(K_{(n-\gamma')}\ast m_{\epsilon}\big)\Big]\ast \eta_{\epsilon}\bigg)(x)\,dx.\end{equation}
In addition, thanks to the convexity of $\int_{\mathbb R^n}\big|\frac{w}{m}\big|^{\gamma'}m\,dx $, by letting $k\to\infty$ in \eqref{suanzitouguji}, we have there exists $C>0$  independent of  $\epsilon>0$ such that
\begin{align}\label{eq4.32}
\int_{\mathbb R^n}\Big|\frac{w_{\epsilon}}{m_{\epsilon}}\Big|^{\gamma'}m_{\epsilon}\,dx+\int_{\mathbb R^n}V(x)m_{\epsilon}\,dx\leq\liminf_{k\rightarrow +\infty }\int_{\mathbb R^n}\big|\frac{w_{\epsilon,k}}{m_{\epsilon,k}}\big|^{\gamma'}m_{\epsilon,k}\,dx+\int_{\mathbb R^n}V(x)m_{\epsilon,k}\,dx\leq C<+\infty.
\end{align}
Moreover,
 \begin{equation}\label{eq4.33}
 \int_{\mathbb R^n}|w_\epsilon|V^{\frac{1}{\gamma}}\,dx\leq \left(\int_{\mathbb R^n}\Big|\frac{w_\epsilon}{m_\epsilon}\Big|^{\gamma'}m_\epsilon\,dx\right)^{\gamma'}\left(\int_{\mathbb R^n}V m_\epsilon \,dx\right)^{\gamma}\leq C<\infty.
 \end{equation}
 and
  \begin{equation}\label{eq4.330}
 \int_{\mathbb R^n}|w_\epsilon|\,dx\leq \left(\int_{\mathbb R^n}\Big|\frac{w_\epsilon}{m_\epsilon}\Big|^{\gamma'}m_\epsilon\,dx\right)^{\gamma'}\left(\int_{\mathbb R^n} m_\epsilon \,dx\right)^{\gamma}\leq C<\infty.
 \end{equation}
 Combining \eqref{convergencesect4firstone} and \eqref{convergencesect4secondone} with \eqref{eq4.330}, we deduce that $(m_{\epsilon},w_{\epsilon})\in \mathcal K_{M}$. Then, one invokes \eqref{eq4.26} and \eqref{eq4.32} to get
\begin{align*}
e_{\epsilon, M}=\lim_{k\rightarrow\infty}\mathcal E_{\epsilon}(m_{\epsilon,k},w_{\epsilon,k})\geq \mathcal E_{\epsilon}(m_{\epsilon},w_{\epsilon})\geq e_{\epsilon,M},
\end{align*}
which indicates $(m_{\epsilon},w_{\epsilon})\in \mathcal K_{M}$ is a minimizer of problem (\ref{Eepsilonminimizerproblem1}).  {Finally, similarly as the proof of Proposition 3.4 in \cite{cesaroni2018concentration} and the arguments shown in Proposition 5.1 and Proposition 5.2 in \cite{cirant2024critical}, we apply Lemma \ref{lemma22preliminary} to obtain that} 
there exists $u_{\epsilon}\in C^2(\mathbb R^n)$ bounded from below (depending on $\epsilon$) and   $\lambda_{\epsilon}\in \mathbb R$ such that
\begin{align}\label{eq4.31}
\left\{\begin{array}{ll}
-\Delta u_{\epsilon}+C_H|\nabla u_{\epsilon}|^{\gamma}+\lambda_{\epsilon}=V(x)-\big(K_{(n-\gamma')}\ast m_{\epsilon}\big) \ast\eta_{\epsilon},\\
\Delta m_{\epsilon}+C_H\gamma \nabla\cdot (m_{\epsilon}|\nabla u_{\epsilon}|^{\gamma-2}\nabla u_{\epsilon})=0,\\
w_{\epsilon}=-C_H\gamma m_{\epsilon}|\nabla u_{\epsilon}|^{\gamma-2}\nabla u_{\epsilon}, \ \int_{\mathbb R^n}m_{\epsilon}\,dx=M<M^*.
\end{array}
\right.
\end{align}

For each fixed $\epsilon>0$, we utilize Lemma \ref{sect2-lemma21-gradientu} to obtain that there exists $C_\epsilon>0$ depends on $\epsilon$ such that
$|\nabla u_{\epsilon}(x)|\leq C_\epsilon(1+V(x))^{\frac{1}{\gamma}}$. Noting that $u_\epsilon\in C^2(\mathbb R^n)$ and $\big(K_{(n-\gamma')}\ast m_{\epsilon}\big)\ast\eta_{\epsilon}\in L^\infty(\mathbb R^n)$, we have from the classical regularity of the $u$-equation in \eqref{eq4.31}  that $|\Delta u_{\epsilon}(x)|\leq C_\epsilon(1+V(x))$.  We next prove
\begin{equation}\label{eq4.34}
|\lambda_\epsilon|\leq C<\infty, \text{ with $C>0$   independent of  $\epsilon>0$}.
\end{equation}
To show this, we apply the integration by parts to the $m$-equation and the $u$-equation in \eqref{eq4.31}, then get
\begin{align*}
\int_{\mathbb R^n}m\Delta u_{\epsilon}\,dx=\int_{\mathbb R^n}w_{\epsilon}\cdot\nabla u_{\epsilon}\,dx=-C_H\gamma\int_{\mathbb R^n}m_{\epsilon}|\nabla u_{\epsilon}|^{\gamma}\,dx,
\end{align*}
and
\begin{equation}\label{eq4.301}
\begin{split}
\lambda_\epsilon M&=-(1-\gamma)C_H\int_{\mathbb R^n}m_\epsilon|\nabla u_\epsilon|^{\gamma}\,dx+\int_{\mathbb R^n}V m_\epsilon\,dx-\int_{\mathbb R^n}m_{\epsilon}\big(K_{(n-\gamma')}\ast m_{\epsilon}\big ) \ast\eta_{\epsilon}\,dx\\
&=C_L\int_{\mathbb R^n}m_\epsilon\bigg|\frac{w_\epsilon}{m_\epsilon}\bigg|^{\gamma'}\, dx+\int_{\mathbb R^n}V m_\epsilon\,dx-\int_{\mathbb R^n}m_{\epsilon}\big(K_{(n-\gamma')}\ast m_{\epsilon}\big) \ast\eta_{\epsilon}\,dx\end{split}
\end{equation}
where we have used the fact that $C_L=\frac{1}{\gamma'}(\gamma C_H)^{\frac{1}{1-\gamma}}$.
Collecting \eqref{eq4.32}, \eqref{eq4.33} and \eqref{eq4.301}, one finds \eqref{eq4.34} holds.

{
Next, we let $\epsilon\rightarrow 0$ and show the existence of the minimizer $(m_{M},w_{M})$ to problem \eqref{ealphacritical-117}.  Noting $(m_{\epsilon},u_{\epsilon},\lambda_{\epsilon})$ satisfies (\ref{assumptionsincrucialemmanewest}) with $k$ replaced by $\epsilon.$
We utilize Young's inequality for convolution and Hardy-Littlewood-Sobolev inequality \eqref{eqHLS_1} to get
 \begin{align*}
\sup_{k}\Vert \big(K_{(n-\gamma')}\ast m_{k}\big ) \ast\eta_{k}\Vert_{L^{1+\frac{n+\alpha^*}{n-\alpha^*}}(\mathbb R^n)}\leq C(n,\gamma')\sup_{k}\Vert  m_{k}^{\frac{n-\alpha^*}{n+\alpha^*}}\Vert_{L^{1+\frac{n+\alpha^*}{n-\alpha^*}}(\mathbb R^n)}=C(n,\gamma')\sup_{k}\Vert  m_{k}\Vert_{L^{\frac{2n}{2n-\gamma'}}(\mathbb R^n)}<\infty,
\end{align*}
and
\begin{align*}
\sup_{k}\Vert \big(K_{(n-\gamma')}\ast m_{k}\big ) \ast\eta_{k}\Vert_{L^{1+\frac{n+\alpha^*}{n-\alpha^*}}\big(B_{2R}(x_0)\big)}\leq C(n,\gamma') \sup_{k}\Vert  m_k\Vert_{L^{\frac{2n}{2n-\gamma'}}\big (B_{2R}(x_0)\big)}<\infty,
\end{align*}}
Then, collecting \eqref{eq4.32} and \eqref{eq4.34}, we invoke Lemma \ref{blowupanalysismlinfboundcritical} to conclude that
 \begin{equation}\label{eq4.35}
 \limsup_{\epsilon\to 0^+}\|m_\epsilon\|_{L^\infty(\mathbb R^n)}<\infty.
 \end{equation}
Then, by using Lemma \ref{sect2-lemma21-gradientu}, we obtain
 \begin{equation}\label{eq4.36}
 |\nabla u_{\epsilon}(x)|\leq C(1+V(x))^{\frac{1}{\gamma}}, \text{ where  $C>0$ is independent of $\epsilon$.}
 \end{equation}
 Since $u_{\epsilon}$ is bounded from below,  without loss of generality, we assume that $u_{\epsilon}(0)=0$.
In light of \eqref{29uklemma22}, one finds that $u_{\epsilon}(x)\geq C_\epsilon V^{\frac{1}{\gamma}}(x)-C_\epsilon\to +\infty$ as $|x|\to+\infty$, which indicates each $u_\epsilon(x)\in C^2( \mathbb R^n)$ admits its minimum at some finite point $x_{\epsilon}$. By using \eqref{eq4.34}, \eqref{eq4.35} and the coercivity of $V$, we obtain from the $u$-equation of \eqref{eq4.31}  that $x_{\epsilon}$ is uniformly bounded with respect to $\epsilon$.  The fact $u_{\epsilon}(0)=0$ together with \eqref{eq4.36} implies that
there exists $C>0$ independent of $\epsilon$ such that
\begin{equation*}
-C\leq u_\epsilon(x)\leq C|x|(1+V(x))^{\frac{1}{\gamma}} \text{ for all } x\in\mathbb R^n,
\end{equation*}
where we have used \eqref{V2mainassumption_3} in the second inequality.
Since $u_{\epsilon}$ are bounded from below uniformly, one can employ Lemma \ref{lowerboundVkgenerallemma22} to get that
 $u_{\epsilon}(x)\geq CV^{\frac{1}{\gamma}}(x)-C \text{ with $C>0$ independent of $\epsilon$.}$
Thus, with the assumptions \eqref{V2mainasumotiononv} imposed on $V$, we get 
 \begin{equation}\label{eq4.37}
  C_1V^{\frac{1}{\gamma}}(x)-C_1\leq u_{\epsilon}\leq C_2|x|(1+V(x))^{\frac{1}{\gamma}},\text{for all } x\in\mathbb R^n.
 \end{equation}
 where $C_1,C_2>0$ are independent of $\epsilon$.

 In light of \eqref{eq4.35} and \eqref{eq4.36}, one finds for any $R>1$ and $p>1$,
 \begin{equation}\label{eq4.306}
 \|w_{\epsilon}\|_{L^p(B_{2R}(0))}=C_H\gamma\|m_{\epsilon}|\nabla u_{\epsilon}|^{\gamma-1}\|_{L^p(B_{2R}(0))}\leq C_{p,R}<\infty,\end{equation}
 where the constant  $C_{p, R}>0$ depends only on $p$, $R$ and is independent of $\epsilon$.  Then, with the help of Lemma \ref{lemma21-crucial}, we obtain from \eqref{eq4.306} that $\|m_\epsilon\|_{W^{1,p}(B_{2R}(0))}\leq C_{p,R}<\infty$. Taking $p>n$ large enough, we utilize Sobolev embedding theorem to get
\begin{equation}\label{eq4.38}
\|m_\epsilon\|_{C^{0,\theta_1}(B_{2R}(0))}\leq C_{\theta_1,R}<\infty \text{ for some $\theta_1\in(0,1)$.}
\end{equation}

To estimate $u_{\epsilon},$ we rewrite the $u$-equation of \eqref{eq4.31} as
\begin{equation}\label{eq4.39}
-\Delta u_{\epsilon}=-C_H|\nabla u_{\epsilon}|^{\gamma}+f_\epsilon(x)\ \text{ with }f_\epsilon(x):=-\lambda_{\epsilon}+V(x)-\big(K_{(n-\gamma')}\ast m_{\epsilon}\big)\ast \eta_{\epsilon},
 \end{equation}
Since $m_\epsilon\in C^{0,\theta_1}(B_{2R}(0))$, then $m_\epsilon\ast \eta_\epsilon\in L^{1}(B_{2R}(0))\cap  L^{\tilde{q}}(B_{2R}(0))$ with $\tilde{q}>\frac{n}{\alpha}$. Thus, we deduce from Lemma \ref{HolderforRiesz} that $\big(K_{(n-\gamma')}\ast m_{\epsilon}\big)\ast \eta_{\epsilon}\in C^{0,\theta_2}(B_{2R}(0))$ for some $\theta_2\in (0,1)$. Now, by using \eqref{eq4.35}, \eqref{eq4.36} and the fact that $V$ is locally H\"older continuous, we obtain that for any $p>1$,
 $$\|f_\epsilon\|_{L^{p}(B_{2R}(0))}+\||\nabla u_\epsilon|^{\gamma}\|_{L^{p}(B_{2R}(0))}\leq C_{p,R}<\infty.$$
Then we utilize the standard elliptic regularity in \eqref{eq4.39} to get
 \begin{equation}\label{eq4.401}
     \|u_{\epsilon}\|_{C^{2,\theta_3}(B_{R}(0))} \leq C_{\theta_3,R}<\infty, \text{ for some $\theta_3\in(0,1)$,}
 \end{equation}
 where $R>0.$
Letting $R\to\infty$ and proceeding the standard diagonalization procedure, we invoke Arzel\`{a}-Ascoli theorem to find there exists $u_M\in C^2(\mathbb R^n)$ such that
 \begin{equation}\label{eq4.42}
 u_{\epsilon}\overset{\epsilon\to0^+}\longrightarrow u_{M} \text{ in }C^{2,\theta_4}_{\rm loc}(\mathbb R^n) \text{ for some $\theta_4\in(0,1)$.}
 \end{equation}
In addition, by using Lemma \ref{lemma21-crucial} and \eqref{eq4.32}, we find there exists $(m_{M},w_{M})\in W^{1,\hat q}(\mathbb R^n)\times \big(L^{1}(\mathbb R^n)\cap L^{\hat q}(\mathbb R^n) \big)$ such that
\begin{align}\label{eq4.43}
m_{\epsilon}\overset{\epsilon\to0^+}\to m_{M} \text{ a.e. in $\mathbb R^n$,  \ and \ \ }(m_{\epsilon},w_{\epsilon})\overset{\epsilon\to0^+}\rightharpoonup (m_{M},w_{M}) \text{~in~}  W^{1,\hat q}(\mathbb R^n)\times  L^{\hat q}(\mathbb R^n).
\end{align}
Moreover, invoking Lemma \ref{lem4.1}, one finds
\begin{align}\label{eq4.44}
m_{\epsilon}\overset{\epsilon\to0^+}\rightarrow m_{M} \text{~in~} L^1(\mathbb R^n)\cap L^{\frac{2n}{2n-\gamma'}}(\mathbb R^n) .
\end{align}

Passing to the limit as $\epsilon\to0^+$ in \eqref{eq4.31}, we then obtain from \eqref{eq4.34} and \eqref{eq4.42}-\eqref{eq4.44} that there exists $\lambda_M\in \mathbb R$ such that
$(m_M, u_M, w_M)$ satisfies  \eqref{125potentialfreesystem}.
In addition, we infer from \eqref{eq4.36} and \eqref{eq4.37} that 
 \begin{equation}\label{eq4.46}
 |\nabla u_{M}(x)|\leq C(1+V(x))^{\frac{1}{\gamma}}\ \text{and} \ C_1|x|^{1+\frac{b}{\gamma}}-C_1\leq u_{M}\leq C_2|x|^{1+\frac{b}{\gamma}}+C_2,\  \forall x\in\mathbb R^n.
 \end{equation}
  Recall that $m_{\epsilon}\to m_{M} $ a.e. as $\epsilon\to0^+$ in $\mathbb R^n$, then we use \eqref{eq4.35} to get that $m_M\in L^\infty(\mathbb R^n)$.  Then, proceeding the same argument as shown in the proof of Proposition 5.2 in \cite{cirant2024critical}, one can further find from \eqref{125potentialfreesystem} and \eqref{eq4.46} that
  \begin{equation}\text{$w_{M}=-C_H\gamma m_{M}|\nabla u_{M}|^{\gamma-2}\nabla u_{M}\in L^p(\mathbb R^n)$ and $m_M\in W^{1,p}(\mathbb R^n)$,
$ \forall p>1$.}\end{equation}

Finally, we prove that $(m_{M},w_{M})\in\mathcal K_{M}$ is a minimizer of $e_{\alpha^*,M}$.  To this end, we claim that for $M<M^*$,
 \begin{equation}\label{eq4.406}
 \lim_{\epsilon\to0^+}e_{\epsilon, M}=e_{\alpha^*,M},
 \end{equation}
 where $e_{\alpha^*,M}$ is given in \eqref{ealphacritical-117}.  On one hand, in view of \eqref{inlightofsect4}, it is straightforward to get $\lim\limits_{\epsilon\to0^+}e_{\epsilon, M}\geq e_{\alpha^*,M}$.  On the other hand,  we aim to show $ \lim\limits_{\epsilon\to0^+}e_{\epsilon, M}\leq e_{\alpha^*,M}$.  Due to the definition of $e_{\alpha^*,M}$,  for any $\delta>0$, we choose $(m,w)\in\mathcal K_{M}$ such that $\mathcal{E}(m,w)\leq e_{\alpha^*,M}+\frac{\delta}{2}$.  In light of \eqref{inlightofsect4}, we conclude that for $\epsilon>0$  small enough,  $\mathcal{E}_\epsilon(m,w)\leq \mathcal{E}(m,w)+\frac{\delta}{2} $.  Thus,
 $$e_{\epsilon, M}\leq \mathcal{E}_\epsilon(m,w)\leq \mathcal{E}(m,w)+\frac{\delta}{2} \leq e_{\alpha^*,M}+{\delta}.$$
 Letting $\epsilon\to0^+$ at first and then $\delta\to 0^+$, one has $ \lim\limits_{\epsilon\to0^+}e_{\epsilon, M}\leq e_{\alpha^*,M}$.  Combining the two facts, we finish the proof of \eqref{eq4.406}.

  We collect  \eqref{eq4.43}, \eqref{eq4.44}, \eqref{eq4.406} and the convexity of $\int_{\mathbb R^n}\big|\frac{w}{m}\big|^{\gamma'}m\,dx $ to get
 $$e_{\alpha^*,M}=\lim_{\epsilon\to0^+}e_{\epsilon, M}=\lim_{\epsilon\to0^+}\mathcal{E}_\epsilon(m_\epsilon,w_\epsilon)\geq \mathcal{E}(m_M,w_M)\geq  e_{\alpha^*,M},$$
which implies $(m_{M},w_{M})\in\mathcal K_{M}$ is a minimizer of $e_{\alpha^*,M}$. This completes the proof of Conclusion (i).

 \vskip.1truein

Now, we focus on Conclusion (ii) of Theorem \ref{thm11}.  We have the fact that $\big(m_{\alpha^*},w_{\alpha^*}, u_{\alpha^*}\big)$ given in Theorem \ref{thm11-optimal} is a minimizer of problem \eqref{sect2-equivalence-scaling} with $\alpha=\alpha^*=(n-\gamma')$. To simplify notation, we rewrite $(m_{\alpha^*},w_{\alpha^*},u_{\alpha^*})$ as $(m_*,w_*,u_*)$, then define
\begin{align}\label{scalingwitht}
( m_{*}^t, w_*^t)=\bigg(\frac{M}{M^*}t^n m_*(t(x-x_0)),\frac{M}{M^*}t^{n+1} w_*(t(x-x_0))\bigg)\in \mathcal K_{M}, \ \ \forall t>0,~ x_0\in\mathbb R^n.
\end{align}
where  the constraint set $\mathcal K_{M}$ and $M^*>0$ are  defined by \eqref{constraint-set-K}  and (\ref{Mstar-critical-mass}), respectively.
Since $u_* \in C^2(\mathbb R^n)$ and $m_*$ decays exponentially  as stated in Theorem \ref{thm11-optimal}, we utilize Lemma \ref{poholemma} to find
\begin{align}\label{invoke45}
C_L\int_{\mathbb R^n}\bigg|\frac{w_*}{m_*}\bigg|^{\gamma'} m_*\,dx=\frac{1}{2}\int_{\mathbb R^n}m_*(x)\big (K_{(n-\gamma')}\ast m_*\big)(x) \,dx.
\end{align}
Thanks to \eqref{invoke45}, we substitute \eqref{scalingwitht} into \eqref{41engliang}, then obtain that if $M>M^*,$
\begin{align}\label{supercriticalmasscase}
e_{\alpha^*,M}\leq \mathcal E(m_*^t,w_*^t)=&\frac{M}{M^*}\bigg(C_Lt^{\gamma'}\int_{\mathbb R^n}\Big|\frac{ w_*}{ m_*}\Big|^{\gamma'} m_*\,dx+\int_{\mathbb R^n}V(x)m_*\, dx\bigg)-\frac{t^{\gamma'}}{2}\bigg(\frac{M}{M^*}\bigg)^{2}\int_{\mathbb R^n}m_*(x)(K_{(n-\gamma')}\ast m_*)(x)\,dx\nonumber\\
=&\frac{M}{M^*}\Big[1-\bigg(\frac{M}{M^*}\bigg)\Big]\frac{t^{\gamma'}}{2}\int_{\mathbb R^n}m_*(x)\big(K_{(n-\gamma')}\ast m_*\big)(x)\,dx +MV(x_0)+o_t(1)\nonumber\\
&\rightarrow -\infty \ \ \text{as~}t\rightarrow +\infty.
\end{align}
Therefore, we have $e_{\alpha^*,M}=-\infty$ for $M>M^*,$ which indicates that problem (\ref{ealphaM-117}) does not admit any minimizer.

Now, we are concentrated at the critical case $M=M^*$ and plan to show Conclusion (iii).  To begin with, we prove that up to a subsequence,
\begin{align}\label{criticalfirstwantshow}
\lim_{ M\nearrow M^*}e_{\alpha^*,M}=e_{\alpha^*,M^*}=0.
\end{align}
Indeed, since $\inf_{x\in\mathbb R^n} V(x)=0$ as shown in $(V1)$ and $e_{\alpha^*,M^*}$ is defined by \eqref{ealphaM-117}, we have for any $\delta>0$, $\exists (m,w)\in\mathcal A_{M^*}$ such that
\begin{align}\label{combine419-thm12}
e_{\alpha^*,M^*}\leq \mathcal E(m,w)\leq e_{\alpha^*,M^*}+\delta.
\end{align}
Noting that $\frac{M}{M^*}(m,w)\in\mathcal A_{M}$, we further obtain
\begin{align}\label{combine420-thm12}
e_{\alpha^*,M}&\leq \mathcal E\bigg(\frac{M}{M^*}m,\frac{M}{M^*}w\bigg)\\
=&\mathcal E(m,w)\nonumber
+\Big(\frac{M}{M^*}-1\Big)\bigg[C_L\int_{\mathbb R^n}\Big|\frac{w}{m}\Big|^{\gamma'}m\,dx+\int_{\mathbb R^n}V(x)m\,dx\bigg]\nonumber
+\frac{1}{2}\bigg[1-\bigg(\frac{M}{M^*}\bigg)^{2}\bigg]\int_{\mathbb R^n}m(x)\big(K_{(n-\gamma')}\ast m\big)(x)\,dx.
\end{align}
By a straightforward computation, one has as $M\nearrow M^*,$
\begin{align}\label{combine421-thm12}
\Big(\frac{M}{M^*}-1\Big)\Big[C_L\int_{\mathbb R^n}\bigg|\frac{w}{m}\bigg|^{\gamma'}m\,dx+\int_{\mathbb R^n}V(x)m\,dx\Big]+\frac{1}{2}\bigg[1-\bigg(\frac{M}{M^*}\bigg)^{2}\bigg]\int_{\mathbb R^n}m(x)\big(K_{(n-\gamma')}\ast m\big)(x)\,dx \rightarrow 0.
\end{align}
We collect (\ref{combine419-thm12}), (\ref{combine420-thm12}) and (\ref{combine421-thm12}) to get
\begin{align}\label{424limitbefore}
\limsup_{M\nearrow M^*}e_{\alpha^*,M}\leq e_{\alpha^*,M^*}+\delta, \ \ \forall \delta>0.
\end{align}
Letting $\delta\rightarrow 0$ in (\ref{424limitbefore}), one has from (\ref{424limitbefore}) that
\begin{align}\label{combine4331critical}
\limsup_{M\nearrow M^*}e_{\alpha^*,M}\leq e_{\alpha^*,M^*}.
\end{align}
In addition, define $(\bar m_{\alpha^*,M},\bar w_{\alpha^*,M})\in\mathcal A_M$ as a minimizer of $e_{\alpha^*,M}=\inf_{(m,w)\in\mathcal A_M}\mathcal E(m,w)$ for any fixed $M\in(0,M^*)$, then we find $\frac{M^*}{M}({\bar m}_{\alpha^*,M},{\bar w}_{\alpha^*,M})\in\mathcal A_{M^*}$ and
\begin{align*}
e_{\alpha^*,M^*}\leq &\mathcal E\Big(\frac{M^*}{M}(\bar m_{\alpha^*,M},\bar w_{\alpha^*,M})\Big)\\
=&\frac{M^*}{M}\bigg[C_L\int_{\mathbb R^n}\Big|\frac{\bar w_{\alpha^*,M}}{\bar m_{\alpha^*,M}}\Big|^{\gamma'}{\bar m}_{\alpha^*,M}\,dx+\int_{\mathbb R^n}V(x)\bar m_{\alpha^*,M}\, dx-\frac{1}{2}\left(\frac{M^*}{M}\right)\int_{\mathbb R^n}{\bar m}_{\alpha^*,M}(x)\big(K_{(n-\gamma')}\ast \bar m_{\alpha^*,M} \big)(x)\,dx\bigg]\\
\leq &\frac{M^*}{M}\mathcal E(\bar m_{\alpha^*,M},\bar w_{\alpha^*,M})=\frac{M^*}{M}e_{\alpha^*,M},\ \ \forall M<M^*.
\end{align*}
  It follows that
\begin{align}\label{combine4332critical}
e_{\alpha^*,M^*}\leq \liminf_{M\nearrow M^*}\frac{M^*}{M}e_{\alpha^*,M}=\lim_{M\nearrow M^*} e_{\alpha^*,M}.
\end{align}
Combining (\ref{combine4331critical}) with (\ref{combine4332critical}), one has
\begin{align}\label{togetherwithenergycriticalvalue}
\lim_{M\nearrow M^*} e_{\alpha^*,M}=e_{\alpha^*,M^*}\geq 0.
\end{align}
In light of assumptions (V1) and (V2) stated in Subsection \ref{mainresults11} for potential $V$, we set $M=M^*$ in (\ref{supercriticalmasscase}) to get
\begin{align*}
e_{\alpha^*,M^*}\leq \mathcal E(m_*^t,w_*^t)=M^*V(x_0)+o_t(1)\rightarrow 0,\text{~if~}V(x_0)=0\text{~and~}t\rightarrow +\infty.
\end{align*}
Hence $e_{\alpha^*,M^*}\leq 0$, which together with (\ref{togetherwithenergycriticalvalue}) implies (\ref{criticalfirstwantshow}).

Now, we focus on the proof Conclusion (iii).  If conclusion (iii) is not true, then we assume that $e_{\alpha^*,M^*}$ has a minimizer $(\hat m,\hat w)\in\mathcal A_{M^*}$. By using (\ref{criticalfirstwantshow}), we further obtain
\begin{align*}
0=\mathcal E(\hat m,\hat w)=\int_{\mathbb R^n}C_L\Big|\frac{\hat w}{\hat m}\Big|^{\gamma'}\hat m\,dx +\int_{\mathbb R^n}V(x)\hat m\,dx-\frac{1}{2}\int_{\mathbb R^n}\hat m(x)\big(K_{(n-\gamma')}\ast \hat m \big)(x)\, dx\geq 0.
\end{align*}
Combining this with \eqref{optimalinequality415}, one gets 
\begin{align}\label{eq5.85}
C_L\int_{\mathbb R^n}\Big|\frac{\hat w}{\hat m}\Big|^{\gamma'}\hat m\,dx=\frac{1}{2}\int_{\mathbb R^n}\hat m(x)\big(K_{(n-\gamma')}\ast \hat m \big)(x)\, dx\text{~and~}\int_{\mathbb R^n}V(x)\hat m\,dx =0,
\end{align}
which implies $\text{supp}V(x)\cap \text{supp~}\hat m=\emptyset.$  Whereas, with the assumption (\ref{V2mainassumption_3}) and the fact $\gamma'< n,$ we have $\text{supp}V=\mathbb R^n.$  It follows that $\hat m=0$ a.e., which is a contradiction.  Consequently, we complete the proof of Conclusion (iii).
\end{proof}

Theorem \ref{thm11} implies that when the potential $V$ satisfies some mild assumptions given by (V1), (V2) and (V3) stated in Section \ref{intro1}, system (\ref{goalmodel}) admits the ground states only when $M<M^*$, where $M^*$ is explicitly shown in Theorem \ref{thm11-optimal} and has a strong connection with the existence of ground states to the potential-free nonlocal Mean-field Games system.  In the next section, we shall discuss the asymptotic behaviors of ground states to problem (\ref{goalmodel}) as $M\nearrow M^*.$    



\section{Asymptotics of Ground States as $M\nearrow M^*$ }\label{sect5preciseblowup}
 This section is devoted to the proof of Theorem \ref{thm13basicbehavior} and Theorem \ref{thm14preciseblowup}. More precisely,  we shall describe the asymptotic profile of 
 least energy solutions to \eqref{MFG-SS} as $M\nearrow M^*.$ 
\subsection{Basic Blow-up Behaviors}

In this subsection, we analyze the basic asymptotic behaviors of ground states to \eqref{MFG-SS} as $M\nearrow M^*$ and prove Theorem \ref{thm13basicbehavior}.

\medskip

\textbf{Proof of Theorem \ref{thm13basicbehavior}:}
\begin{proof}
    To prove Conclusion (i), we perform the blow-up argument and assume
    $$\limsup_{M\nearrow M^*}\int_{\mathbb R^n}\bigg|\frac{w_M}{m_M}\bigg|^{\gamma'}m_M\,dx<+\infty.$$
Then we utilize Lemma \ref{lemma21-crucial} to get 
\begin{align}\label{inlightofsect5blowupbasicpro}
\limsup_{M\nearrow M^*}\Vert m_M\Vert_{W^{1,\hat q}(\mathbb R^n)}, ~\limsup_{M\nearrow M^*}\Vert w_M\Vert_{L^{\hat q}(\mathbb R^n)},~\limsup_{M\nearrow M^*}\Vert w_M\Vert_{L^1(\mathbb R^n)}<+\infty.
\end{align}
 Consequently, we have there exists $(m,w)\in  W^{1,\hat q}(\mathbb R^n)\times  L^{\hat q}(\mathbb R^n) $ such that
\begin{align}\label{combine1sect5}
m_M\rightharpoonup m \text{~in~} W^{1,\hat q}(\mathbb R^n)~\text{ and }~w_M\rightharpoonup w\text{~in~}  L^{\hat q}(\mathbb R^n)~\text{ as }M\nearrow M^*.
\end{align}
Now, we prove $(m,w)\in \mathcal K_{M^*}$ given by \eqref{constraint-set-K}.  Indeed, noting \eqref{inlightofsect5blowupbasicpro}, we have
\begin{align}\label{combinesec3}
\limsup_{M\nearrow M^*}\int_{\mathbb R^n}V(x)m_M\,dx<+\infty.
\end{align}
By using the assumptions (V1), (V2) and (V3) satisfied by $V$, we conclude from \eqref{combine1sect5}, \eqref{combinesec3} and Lemma \ref{lem4.1} that
\begin{align}\label{eq5.4}
m_M\rightarrow  m \text{~in~} L^1(\mathbb R^n)\cap L^{\frac{2n}{2n-\gamma'}}(\mathbb R^n), ~\text{as $M\nearrow M^*$},
\end{align}
which implies $\int_{\mathbb R^n}m\,dx=M^*.$  Moreover, thanks to \eqref{combine1sect5}, one gets $\Delta m= \nabla\cdot w$ {weakly.}  It follows that
\begin{equation*}
\int_{\mathbb R^n}|w| dx=\int_{\mathbb R^n}|w| |m|^{-\frac{(\gamma'-1)}{\gamma'}} |m|^{\frac{(\gamma'-1)}{\gamma'}}dx\leq \left(\int_{\mathbb R^n}|m|\Big|\frac{w}{m}\Big|^{\gamma'}\, dx \right)^\frac{1}{\gamma'}\left(\int_{\mathbb R^n}m\,dx\right)^\frac{\gamma'-1}{\gamma'}<+\infty,
\end{equation*}
which implies $w\in L^1(\mathbb R^n)$. 
Hence, we obtain $(m,w)\in \mathcal K_{M^*}$ and further $\liminf\limits_{M\nearrow M^*} \mathcal E(m_M,w_M)\geq\mathcal E(m,w)$ due to \eqref{combine1sect5} and \eqref{eq5.4}. Moreover, one has from \eqref{criticalfirstwantshow} that
$$e_{\alpha^*,M^*}\geq \mathcal E(m,w)\geq e_{\alpha^*,M^*}.$$
Therefore, $(m,w)$ is a minimizer of $e_{\alpha^*,M^*},$ which yields a contradiction to Conclusion (iii) in Theorem \ref{thm11}. This finishes the proof of Conclusion (i).

(ii).  Note that  
\begin{align*}
    {\varepsilon}_M={\varepsilon}:=\Big(C_L\int_{\mathbb R^n}\bigg|\frac{w_M}{m_M}\bigg|^{\gamma'}m_M\,dx\Big)^{-\frac{1}{\gamma'}}\rightarrow 0\text{~as~}M \nearrow M^*.
\end{align*}
As stated in Conclusion (i) of Theorem \ref{thm11}, we have each $u_M\in C^2(\mathbb R^n)$ is bounded from below and satisfies $\lim\limits_{|x|\rightarrow+\infty} u_M(x)=+\infty$.  Hence, there exists $x_{\varepsilon}\in \mathbb R^n$ such that $u_M(x_{\varepsilon})=\inf\limits_{x\in\mathbb R^n}u_{M}(x)$, which indicates $0=u_{\varepsilon}(0)=\inf\limits_{x\in\mathbb R^n}u_{\varepsilon}(x)$ thanks to the definition given in (\ref{thm51property2}).  

In light of \eqref{125potentialfreesystem} and \eqref{thm51property2}, we find that $(u_\varepsilon, m_\varepsilon, w_\varepsilon)$ satisfies the following system
 \begin{align}\label{eqscalingaftersect5mfgnew}
 \left\{\begin{array}{ll}
 -\Delta u_{\varepsilon}+C_H|\nabla u_{\varepsilon}|^{\gamma}+\lambda_M \varepsilon^{\gamma'}=-\big(K_{(n-\gamma')}\ast m_{\varepsilon}\big)(x)+\varepsilon^{\gamma'}V(\varepsilon x+x_{\varepsilon}),\\
 -\Delta m_{\varepsilon}-C_H\gamma\nabla\cdot(m_{\varepsilon}|\nabla u_{\varepsilon}|^{\gamma-2}\nabla u_{\varepsilon})=-\Delta m_{\varepsilon}+\nabla \cdot w_{\varepsilon}=0,\\
 \int_{\mathbb R^n}m_{\varepsilon}\,dx=M.
 \end{array}
 \right.
 \end{align}
Collecting \eqref{thm51property1}, \eqref{optimalinequality415} and \eqref{criticalfirstwantshow}, one gets
\begin{align}\label{54equalitynormalize}
C_L\int_{\mathbb R^n}\bigg|\frac{w_{\varepsilon}}{m_{\varepsilon}}\bigg|^{\gamma'}m_{\varepsilon}\,dx=\varepsilon^{\gamma'}C_L\int_{\mathbb R^n}\bigg|\frac{w_{M}}{m_{M}}\bigg|^{\gamma'}m_{M}\,dx\equiv 1,
\end{align}
\begin{align}\label{55equalitynormalize}
\int_{\mathbb R^n}m_{\varepsilon}(x)\big(K_{(n-\gamma')}\ast m_{\varepsilon}\big)(x)\,dx=\varepsilon^{\gamma'}\int_{\mathbb R^n}m_{M}(x)\big(K_{(n-\gamma')}\ast m_{M}\big)(x)\,dx\rightarrow 2,
 \end{align}
and
\begin{align}\label{59limitafterV}
\int_{\mathbb R^n} V(\varepsilon x+x_{\varepsilon})m_{\varepsilon}\,dx=\int_{\mathbb R^n}V(x)m_M\,dx\rightarrow 0\text{~as~}M\nearrow M^*.
\end{align}
Following the similar argument employed in the derivation of \eqref{eq4.301}, we utilize \eqref{eqscalingaftersect5mfgnew} and \eqref{54equalitynormalize} to obtain
\begin{align*}
M\lambda_M=\mathcal E(m_M,w_M)-\frac{1}{2}\int_{\mathbb R^n}m_M(x)\big(K_{(n-\gamma')}\ast m_M \big)(x)\,dx=o(1)-\frac{1}{2}\varepsilon^{-\gamma'}\int_{\mathbb R^n}m_{\varepsilon}(x)\big(K_{(n-\gamma')}\ast m_{\varepsilon} \big)(x)\,dx,
\end{align*}
which implies
\begin{align}\label{eq5.8}
\lambda_M\varepsilon^{\gamma'}\rightarrow -\frac{1}{M^*}\text{~as~}M\nearrow M^*.
\end{align}

We apply the maximum principle to the $u$-equation in \eqref{eqscalingaftersect5mfgnew}, then deduce that
\begin{align}\label{maximumprincipleconclusionsect5}
\lambda_M\varepsilon^{\gamma'}\geq -\big(K_{(n-\gamma')}\ast m_{\varepsilon}\big)(0)+\varepsilon^{\gamma'}V(x_{\varepsilon})\geq -\big(K_{(n-\gamma')}\ast m_{\varepsilon}\big)(0),
\end{align}
which indicates
\begin{align}\label{eq510mlowerbounduniform}
\big(K_{(n-\gamma')}\ast m_{\varepsilon}\big)(0)\geq -\lambda_M\varepsilon^{\gamma'}\geq C>0.
\end{align}
{
Now, we claim that there exists some constant $C>0$ such that
\begin{align}\label{eq6.110}
\varepsilon^{\gamma'}V(x_{\varepsilon} )\leq C.
\end{align}
If this is not the case, one can find some subsequence $\varepsilon_l\rightarrow 0$ such that $\varepsilon_l^{\gamma'}V(x_{\varepsilon_l})\rightarrow +\infty$. Then, with the aid of \eqref{maximumprincipleconclusionsect5}, one has 
\begin{align}\label{eq6.111}
\frac{\big(K_{\alpha^*}\ast m_{\varepsilon_l}\big)(0)}{\varepsilon_{l}^{ \gamma'}V(x_{\varepsilon_l})}\geq C,
\end{align}
where $C>0$ is some constant independent of $\varepsilon_l.$  
Define
\begin{align}\label{eqscaling}
v_{l}(x):=a_{l}^{\gamma'-2}u_{l}(x_0+a_{l}x),~~\mu_{l}(x):=a_{l}^{n}m_{l}(x_0+a_{l}x),~~a_{l}:=\frac{1}{\varepsilon_lV(x_{\varepsilon_l})^{\frac{1}{\gamma'}}},
\end{align}
then one has
\begin{align*}
a_{l}^{\gamma'}=\frac{1}{\varepsilon_l^{\gamma'}V(x_{\varepsilon_l})}\rightarrow 0,~~a_{l}^{\gamma'}\varepsilon_l^{\gamma'}V(x_{\varepsilon_l})=1.
\end{align*}
By substituting \eqref{eqscaling} into \eqref{eqscalingaftersect5mfgnew}, we find
\begin{align}\label{413rescale2024206}
\left\{\begin{array}{ll}
-\Delta v_l+C_H|\nabla v_l|^{\gamma}+a_l^{\gamma'}\lambda_{M}=a_l^{\gamma'}V(x_l+a_lx)-a^{\gamma'-n}_l\Big(K_{\alpha^*}\ast\mu_l \Big), &x\in\mathbb R^n,\\
\Delta \mu_l+C_H\gamma\nabla\cdot(|\nabla v_l|^{\gamma-2}\nabla v_l\mu_l)=0,&x\in\mathbb R^n.
\end{array}
\right.
\end{align}
By using the assumption \eqref{V2mainassumption_2}, one gets
\begin{align*}
a_{l}^{\gamma'}\varepsilon^{\gamma'}_l V(a_{l}\varepsilon_l x+x_{\varepsilon_l})=\frac{V(a_{\varepsilon_l}\varepsilon_l x+x_{\varepsilon_l})}{V(x_{\varepsilon_l})}\leq C,
\end{align*}
where $C>0$ is some constant independent of $l.$  Noting that
\begin{align*}
\Vert \mu_l^{\frac{n-\alpha^*}{n+\alpha^*}}\Vert^{1+\frac{n+\alpha^*}{n-\alpha^*}}_{L^{1+\frac{n+\alpha^*}{n-\alpha^*}}(B_R(0))}=a_l^{\frac{\gamma'n}{2n-\gamma'}}\Vert m_{\varepsilon_l}\Vert_{L^{\frac{2n}{2n-\gamma'}}(B_{Ra_l}(x_l))}\rightarrow 0\text{~as~}l\rightarrow +\infty,
\end{align*}
we utilize the maximal regularity shown in Lemma \ref{thmmaximalregularity} to obtain
\begin{align*}
\Vert |\nabla v_{l}|^{\gamma}\Vert_{L^{1+\frac{n+\alpha^*}{n-\alpha^*}}(B_{R/2})}\leq C,
\end{align*}
where $R>0$ and $C>0$ are some constants.  Focusing on the $m$-equation of \eqref{413rescale2024206}, we similarly apply the standard elliptic regularity estimates (See Theorem 1.6.5 in \cite{bogachev2022fokker}) to obtain $\mu_l\in C^{0,\theta}(B_{R/4}(0))$ with $\theta\in(0,1)$ independent of $l$. By a direct calculation, we conclude from  \eqref{eq6.111} that
\begin{align}\label{eq6.112}
\Big(K_{\alpha^*}\ast \mu_{l}\Big)(0)=a_{l}^{r}\Big(K_{\alpha^*}\ast m_{l}\Big)(0)= \frac{\Big(K_{\alpha^*}\ast \mu_{l}\Big)(0)}{\varepsilon_l^{r}V(x_{\varepsilon_l})}\geq C>0.
\end{align}
This together with the H\"{o}lder's continuity of $\mu_l$ implies that
\begin{align}\label{lbd11172024}
\int_{B_{R/4}(0)}\mu_{l}(x)\,dx\geq C>0,
\end{align}
where $R>0$ sufficiently large and independent of $l$. 
In light of $\varepsilon^{\gamma'}_lV(x_{\varepsilon_l})\rightarrow+\infty$, we have the fact that $|x_{\varepsilon_l}|\rightarrow +\infty.$  As a consequence, there exists $\delta>0$ such that $V(x_{\varepsilon_l})\geq 2\delta.$  Then
It follows from \eqref{lbd11172024} that 
\begin{align*}
&\int_{\mathbb R^n} V(\varepsilon_l x+x_{\varepsilon_l})m_{\varepsilon_l} (x)\,dx\nonumber\\
=&\int_{\mathbb R^n}V(\varepsilon_l a_{\varepsilon_l }x+x_{\varepsilon_l})\mu_{l}\,dx\geq \delta\int_{B_{{R}}(0)}\mu_l\,dx\geq C{\delta}>0,
\end{align*}
as $\varepsilon_l\rightarrow 0.$
Whereas, 
\begin{align*}
\int_{\mathbb R^n} V(\varepsilon x+x_{\varepsilon})m_{\varepsilon} (x)\,dx\rightarrow 0 ~\text{as}~\varepsilon\rightarrow 0,
\end{align*}
which reaches a contradiction.  This completes the proof of our claim \eqref{eq6.110}.
}



Moreover, since $V$ satisfies (\ref{V2mainassumption_2}), one further obtains for $R>0$ large enough,
\begin{align}\label{eq512insect5new}
				\varepsilon^{\gamma'}V(\varepsilon x+x_{\varepsilon})\leq C_R<+\infty,\text{ for all }|x|\leq 4R,
			\end{align}
			where constant $C_R>0$ depends on $R$ and is independent of $\varepsilon.$
			
		Similarly as discussed in the proof of Theorem \ref{thm11}, we  estimate  $\nabla u_{\varepsilon}$ and rewrite the $u$-equation of \eqref{eqscalingaftersect5mfgnew} as
			\begin{equation}\label{negativedeltaueq6521}
				-\Delta u_{\varepsilon}=-C_H|\nabla u_{\varepsilon}|^{\gamma}+g_\varepsilon(x)\ \text{ with }g_\varepsilon(x):=-\lambda_{M}\varepsilon^{\gamma'} +\varepsilon^{\gamma'} V(x_{\varepsilon}+\varepsilon x)-\big(K_{\alpha^*}\ast m_{\varepsilon}\big)(x).
			\end{equation}
		Noting that $\big(K_{\alpha^*}\ast m_{\varepsilon}\big)\in L^{1+\frac{n+\alpha^*}{n-\alpha^*}}(\mathbb R^n)$, we utilize Lemma \ref{thmmaximalregularity} to get $|\nabla u_{\varepsilon}|^{\gamma}\in L_{\rm loc}^{1+\frac{n+\alpha^*}{n-\alpha^*}}(\mathbb R^n)$, i.e. $|\nabla u_{\varepsilon}|^{\gamma-1}\in L_{\rm loc}^{\big(1+\frac{n+\alpha^*}{n-\alpha^*}\big)\gamma'}(\mathbb R^n)$.  By using Lemma \ref{lemma21-crucial-cor}, we further obtain that $m\in C^{0,\theta}_{\text{loc}}(\mathbb R^n)$ for some $\theta\in(0,1)$ since $m$ satisfies the second equation in (\ref{eqscalingaftersect5mfgnew}).
 \begin{align}\label{eq515C2thetaquestion}
\Vert u_{\varepsilon}\Vert_{C^{2,\theta}(B_R(0))}\leq C<\infty.
 \end{align}
  In light of (\ref{eq510mlowerbounduniform}), we have from (\ref{eq515C2thetaquestion}) that there exists a constant $R_0\in(0,1)$ such that
 \begin{align}\label{eq516lowerboundmindependent}
 m_{\varepsilon}(x)\geq C>0,~~\forall |x|<R_0.
 \end{align}

Now, we claim that up to a subsequence,
\begin{equation}\label{eq5.18}
    \lim_{\varepsilon\to0}x_{\varepsilon}= x_0 \ \text{ with }~V(x_0)=0.
\end{equation}
If not, one has either $|x_{\varepsilon}|\rightarrow +\infty$ or $x_{\varepsilon}\rightarrow x_0$ with $V(x_0)>0.$  In the two cases, we both have
$\lim\limits_{x_\varepsilon\to0}V(\varepsilon x+x_\varepsilon)\geq A$ a.e. in $\mathbb R^n$ for some $A>0$.  It then follows from \eqref{eq516lowerboundmindependent} that
\begin{align*}
\lim_{\varepsilon\to 0}\int_{\mathbb R^n}V(\varepsilon x+x_{\varepsilon})m_{\varepsilon}\,dx\geq \frac{A}{2}\int_{B_{R_0}(0)}m_{\varepsilon} (x)\, dx\geq \frac{AC}{2}|B_{R_0}(0)|,
\end{align*}
which contradicts \eqref{59limitafterV}.  Therefore, we find \eqref{eq5.18} holds.

By using (\ref{54equalitynormalize}), we find there exists $(m_0,w_0) \in W^{1,\hat q}(\mathbb R^n)\times \big(L^1(\mathbb R^n)\cap L^{\hat q}(\mathbb R^n)\big)$ such that
\begin{align}\label{eq5.20519}
(m_{\varepsilon},w_{\varepsilon})\rightharpoonup (m_0,w_0) \ \ \text{in~} W^{1,\hat q}(\mathbb R^n)\times  L^{\hat q}(\mathbb R^n) \ \text{ as }\varepsilon\to 0,
\end{align}
where  $m_0\not\equiv 0$ thanks to \eqref{eq516lowerboundmindependent}  and  $\hat q$ is given by (\ref{hatqconstraint}).  Furthermore, invoking (\ref{eq515C2thetaquestion}), one has $u_{\varepsilon}\rightarrow u_0$ in $C^2_{\rm loc}(\mathbb R^n)$.  Moreover, combining \eqref{eqscalingaftersect5mfgnew} with \eqref{eq5.8}, we obtain  $(m_0,u_0)$ satisfies
 \begin{align} \label{eq5.20}
 \left\{\begin{array}{ll}
 -\Delta u_0+C_H|\nabla u_0|^{\gamma}-\frac{1}{M^*}=-K_{(n-\gamma')}\ast m_{0},\\
 -\Delta m_0=-C_H\gamma\nabla\cdot(m_0|\nabla u_0|^{\gamma-2}\nabla u_0)=-\nabla \cdot w_0,\\
 0<\int_{\mathbb R^n}m _0\,dx\leq M^*,~~w_0=-C_Hm_0|\nabla u_0|^{\gamma-2}\nabla u_0,
 \end{array}
 \right.
 \end{align}
 where we have followed the procedure performed in the proof of \eqref{combining1strongconverge2} shown in Section \ref{sect3-optimal}.  In particular, we have used Lemma \ref{poholemma} to obtain that $(m_0, w_0)$ is a minimizer of \eqref{sect2-equivalence-scaling}  and $\int_{\mathbb R^n}m _0\,dx=M^*$.  Thus, we have from (\ref{eq5.20}) that $(u_0,m_0,w_0)$ satisfies \eqref{limitingproblemminimizercritical0}.
On the other hand, we obtain $m_\varepsilon\to m_0$ in $L^1(\mathbb R^n)$, and then with the aid of  \eqref{eq5.20519}, one finds
 \begin{align*}
m_{\varepsilon}\to m_0  \ \ \text{in~} L^{p}(\mathbb R^n),~ \forall ~ p\in[1,{\hat q}^*)  \ \text{ as }\varepsilon\to 0,
\end{align*}
which indicates that \eqref{eq1.40} holds.

Finally, we prove that \eqref{eq141realtionuminmmax}  holds when \eqref{cirant-V} is imposed on $V$.  To this end, we argue by contradiction and assume that, 
then, up to a subsequence, 
 \begin{align}\label{eq5.23}
 \liminf_{\varepsilon\rightarrow 0}\frac{|\bar x_{\varepsilon}-x_{\varepsilon}|}{\varepsilon}=+\infty.
 \end{align}
Define
\begin{align}\label{eq5.24}
\left\{\begin{array}{ll}
\bar m_{\varepsilon}(x):=\varepsilon^{n}m_M(\varepsilon x+\bar x_{\varepsilon})=m_{\varepsilon}\Big(x+\frac{\bar x_{\varepsilon}-x_{\varepsilon}}{\varepsilon}\Big),\\
\bar w_{\varepsilon}(x):=\varepsilon^{n+1}w_M(\varepsilon x+\bar x_{\varepsilon})=w_{\varepsilon}\Big(x+\frac{\bar x_{\varepsilon}-x_{\varepsilon}}{\varepsilon}\Big),\\
\bar u_{\varepsilon}(x):=\varepsilon^{\frac{2-\gamma}{\gamma-1}} u_M(\varepsilon x+\bar x_{\varepsilon})=u_{\varepsilon}\Big(x+\frac{\bar x_{\varepsilon}-x_{\varepsilon}}{\varepsilon}\Big).
\end{array}
\right.
\end{align}
Now, we claim that $\exists R_0>0$ and $C>0$ independent of $\varepsilon$ such that
\begin{align}\label{claimuniformlowerboundmcriticalsect5}
\bar m_{\varepsilon}(x)\geq C>0,\ \ \forall~ |x|<R_0.
\end{align}
Invoking \eqref{eq5.24}, we have (\ref{claimuniformlowerboundmcriticalsect5}) is equivalent to
\begin{align}\label{eq5.25lowerboundsect5}
m_{\varepsilon}(x)\geq C>0, \ \ \forall~\Big|x-\frac{\bar x_{\varepsilon}-x_{\varepsilon}}{\varepsilon}\Big|<R_0.
\end{align}
In light of \eqref{eq510mlowerbounduniform}, we find
\begin{equation}\label{eq5.27}
    \bar m_\varepsilon(0)=\|\bar m_\varepsilon\|_{L^\infty(\mathbb R^n)}=\| m_\varepsilon\|_{L^\infty(\mathbb R^n)}>C>0.
\end{equation}
To show (\ref{eq5.25lowerboundsect5}), we have from the first equation in \eqref{eqscalingaftersect5mfgnew} that $\bar u_{\varepsilon}$ satisfies
\begin{align}
-\Delta \bar u_{\varepsilon}+C_H|\nabla \bar u_{\varepsilon}|^{\gamma}=\bar g_{\varepsilon}(x):=-\lambda_M\varepsilon^{\gamma'}-\big(K_{(n-\gamma')}\ast \bar m_{\varepsilon}\big)(x)+\varepsilon^{\gamma'}V(\varepsilon x+\bar x_{\varepsilon}).
\end{align}  
Following the argument shown in \cite[Theorem 4.1]{cesaroni2018concentration}, we consider the following two cases:

\textbf{Case 1:} Assume that there exists some constant $C>0$ independent of $\varepsilon$ such that $\bar x_{\varepsilon}$ satisfies
 \begin{align*}
 \limsup_{\varepsilon\rightarrow 0}\varepsilon^{\gamma'}V(\bar x_{\varepsilon})\leq C<+\infty.
 \end{align*}
 Then thanks to \eqref{eq5.27}, we follow the same argument performed in the derivation of \eqref{negativedeltaueq6521}, \eqref{eq515C2thetaquestion} and \eqref{eq516lowerboundmindependent} to obtain the claim (\ref{claimuniformlowerboundmcriticalsect5}).

\textbf{Case 2:}  Suppose that $\bar x_{\varepsilon}$ satisfies
\begin{align}\label{eq5.29}
\liminf_{\varepsilon\rightarrow 0} \varepsilon^{\gamma'}V(\bar x_{\varepsilon})=+\infty.
\end{align}
Define
\begin{align}\label{eq5.30}
\tilde m(x)=\varepsilon^n m_M(\varepsilon x)=m_{\varepsilon}\bigg(x-\frac{x_{\varepsilon}}{\varepsilon}\bigg),\
\tilde u(x)=\varepsilon^{\frac{2-\gamma}{\gamma-1}}u_M(\varepsilon x),
\tilde w(x)=\varepsilon^{n+1}w_M(\varepsilon x),
\end{align}
then obtain from \eqref{eqscalingaftersect5mfgnew} that $(\tilde m,\tilde u,\tilde w)$ satisfies
\begin{align}
\left\{\begin{array}{ll}
-\Delta \tilde u+C_H|\nabla \tilde u|^{\gamma}+\lambda_M\varepsilon^{\gamma'}=\varepsilon^{\gamma'}V(\varepsilon x)-K_{(n-\gamma')}\ast \tilde m_{\varepsilon},&x\in\mathbb R^n,\\
-\Delta \tilde m-C_H\gamma\nabla\cdot(\tilde m|\nabla \tilde u|^{\gamma-2}\nabla \tilde u)=0,&x\in\mathbb R^n,\\
\int_{\mathbb R^n}\tilde m_{\varepsilon}\,dx =M\nearrow M^*, \ \ \tilde w_{\varepsilon}=-C_H\gamma\tilde m|\nabla \tilde u|^{\gamma-2}\nabla \tilde u.
\end{array}
\right.
\end{align}
Since $V$ satisfies (\ref{cirant-V}), we utilize Lemma \ref{sect2-lemma21-gradientu} to get
\begin{align}\label{eq5.31sect5gradtildeu}
|\nabla \tilde u_\varepsilon|\leq C\big(1+\sigma_\varepsilon^{\frac{1}{\gamma}}|x|^{\frac{b}{\gamma}}\big), \ \ \sigma_\varepsilon:=\varepsilon^{\gamma'+b}.
\end{align}
 Denote $y_{\varepsilon}:=\frac{x_{\varepsilon}}{\varepsilon}$ and $\bar y_{\varepsilon}:=\frac{\bar x_{\varepsilon}}{\varepsilon}$, which are  the minimum and maximum points of $\tilde u_\varepsilon(x)$ and $\tilde m_\varepsilon(x),$ respectively. With the aid of \eqref{eq5.18}, we obtain
$|y_{\varepsilon}|\leq C\varepsilon^{-1}.$  Then, we obtain from  \eqref{eq5.31sect5gradtildeu} that
\begin{align}\label{eq5.33}
|\tilde u_{\varepsilon}(0)|\leq |\tilde u_{\varepsilon}(y_{\varepsilon})|+| y_{\varepsilon}|\sup_{|y|\leq |y_{\varepsilon}|}|\nabla \tilde u_{\varepsilon}(y)|
\leq 1+C\varepsilon^{-1}+C\varepsilon^{-1}\sigma_\varepsilon^{\frac{1}{\gamma}}|y_{\varepsilon}|^{\frac{b}{\gamma}}\leq 1+C\varepsilon^{-1}.
\end{align}
As a consequence,
\begin{align}\label{eq534uniformupperboundsect5}
\tilde u_{\varepsilon}(x)\leq \tilde u_{\varepsilon}(0)+|x|\sup|\nabla u_{\varepsilon}(x)|
\leq 1+C\varepsilon^{-1}+\sigma_\varepsilon^{\frac{1}{\gamma}}|x|^{\frac{b}{\gamma}+1}.
\end{align}
Collecting \eqref{eq5.29}, \eqref{eq5.33} and (\ref{eq534uniformupperboundsect5}), we proceed the same argument shown in \cite[Theorem 4.1]{cesaroni2018concentration} to get $\tilde m_{\varepsilon}\in C^{0,\theta}(B_R(\bar y_\varepsilon))$ with $\theta\in(0,1)$ and $R>0$ independent of $\varepsilon.$  Since $\bar y_\varepsilon$ is maximum point of $\tilde m_\varepsilon(x)$, we combine \eqref{eq5.27} with \eqref{eq5.30} to get $\tilde m_{\varepsilon}(\bar y_{\varepsilon})\geq C>0.$
Hence, we have there exists some $R_0>0$ independent of $\varepsilon$ such that
\begin{align*}
\tilde m_{\varepsilon}(x)>\frac{C}{2}>0,\forall~|x-\bar y_{\varepsilon}|<R_0.
\end{align*}
Noting $\bar y_\varepsilon=\frac{\bar x_\varepsilon}{\varepsilon}$, we find from the above estimate and \eqref{eq5.30} that (\ref{eq5.25lowerboundsect5}) holds. 

Thus, if the potential $V$ satisfies (\ref{cirant-V}), then (\ref{claimuniformlowerboundmcriticalsect5}) and \eqref{eq5.25lowerboundsect5} hold.  Whereas, \eqref{eq5.25lowerboundsect5} together with \eqref{eq5.23} contradicts the fact that  $m_{\varepsilon}$ converges strongly  to $m_0$ in $L^1(\mathbb R^n)$.  As a consequence, \eqref{eq141realtionuminmmax} holds and this  completes the proof of Theorem \ref{thm13basicbehavior}. \end{proof}


In Theorem \ref{thm13basicbehavior}, we see that  as $M\nearrow M^*,$ the ground states $(m_M,w_M)$ to problem (\ref{ealphaM-117}) concentrate and become localized patterns, in which the profiles are determined by $(m_0,w_0),$ the minimizer to problem (\ref{optimal-inequality-sub}).  We mention that with some typical expansions imposed on potential $V$ locally, the detailed asymptotics of ground states can be captured and we shall discuss them in Subsection \ref{refinedblow}.


\subsection{Refined Blow-up Behaviors}\label{refinedblow}
In this subsection, we shall analyze the refined asymptotic profiles of the rescaled minimizer $(m_{\varepsilon},w_{\varepsilon})$ and prove Theorem \ref{thm14preciseblowup}. 
As shown in Theorem \ref{thm14preciseblowup}, we assume $V(x)$ has $l\in \mathbb N$ distinct zeros defined by $\{P_1,\cdots, P_l\}\subset \mathbb{R}^n$; moreover, $\exists a_i>0$, $q_i>0$, $d>0$ such that
\begin{align}\label{taylorexpansionV}
V(x)=a_i|x-P_i|^{q_i}+O(|x-P_i|^{q_i+1}), \ \ \text{~if~}|x-P_i|\leq d.
\end{align}
Define $q=\max\{q_1,\cdots,q_l\}$, $Z=\{P_i~|~q_i=q, i=1,\cdots,l\}$ and denote
\begin{equation}\label{eq6.3}
\mu=\min\{\mu_i~|~P_i\in Z, i=1,\cdots,l\} ~\text{ with }~\mu_i=\min\limits_{y\in\mathbb R^n}H_i(y),\ H_i(y)=\int_{\mathbb R^n} a_i|x+y|^{q_i}m_0(x)\,dx.
\end{equation}
Set $Z_0=\{P_i~|~P_i\in Z\text{~and~} \mu_i=\mu, i=1,\cdots,l\}$ consisted of all weighted flattest zeros of $V(x).$ Collecting the above notations, we first establish the precise upper bound of $e_{\alpha^*,M}$  as $M\nearrow M^*$ stated as follows:
\begin{lemma}\label{lemma61}
The $e_{\alpha^*,M}$, defined by \eqref{ealphacritical-117}, satisfies
\begin{align}\label{eq6.5}
e_{\alpha^*,M}\leq [1+o(1)]\frac{q+\gamma'}{q}\bigg(\frac{q\mu}{\gamma'}\bigg)^{\frac{\gamma'}{\gamma'+q}}\bigg[1-\frac{M}{M^*}\bigg]^{\frac{q}{\gamma'+q}},\ \text{as} \ M\nearrow M^*. 
\end{align}
\end{lemma}
\begin{proof}
The proof is similar as the argument shown in \cite[Lemma 6.1]{cirant2024critical} with slight modifications.  We omit the details.
\end{proof}

In Section \ref{sect5preciseblowup}, we find 
$(m_{\varepsilon},w_{\varepsilon},u_{\varepsilon})$
converges to $(m_0,w_0,u_0)$ in the following sense: $$\text{$m_{\varepsilon}\rightarrow m_0$ in $L^p(\mathbb R^n)\ \forall ~p\in[1,{\hat q}^*) $, $w_{\varepsilon}\rightharpoonup w_0$ $L^{\hat q}(\mathbb R^n)$ and $u_{\varepsilon}\rightarrow u_0$ in $C^2_{\text{loc}}(\mathbb R^n)$},$$
where $(m_0,w_0)$ is the minimizer of $\Gamma_{\alpha^*}$ and correspondingly, $(u_0, m_0,w_0)$ satisfies \eqref{limitingproblemminimizercritical0}. Moreover, Lemma \ref{mdecaylemma} and Lemma \ref{poholemma} imply $\exists\delta_1>0$ and $C_{\delta_1}>0$ such that
\begin{equation}\label{eq6.1}
 m_0(x)\leq C_{\delta_1}C^{-\delta_1|x|},\end{equation}
and
\begin{align}\label{6dian2}
C_L\int_{\mathbb R^n}\Big|\frac{w_0}{m_0}\Big|^{\gamma'}m_0\,dx=\frac{1}{2}\int_{\mathbb R^n}m_0(x)\big(K_{(n-\gamma')}\ast m_0 \big)(x)\,dx=1.
\end{align}

Next, invoking Lemma \ref{lemma61}, \eqref{eq6.1} and \eqref{6dian2}, we are going to prove Theorem \ref{thm14preciseblowup}, which is


\medskip
\textbf{Proof of Theorem \ref{thm14preciseblowup}:}
\begin{proof}
Thanks to Theorem \ref{thm13basicbehavior}, we have $x_{\varepsilon}\rightarrow P_i$ for some $1\leq i\leq l.$  In addition, noting that $(m_M,w_M)$ is the minimizer of problem \eqref{ealphacritical-117}, one gets
\begin{align}
e_{\alpha^*,M}=\mathcal E(m_M,w_M)=&\varepsilon^{-\gamma'}C_L\int_{\mathbb R^n}\bigg|\frac{w_{\varepsilon}}{m_{\varepsilon}}\bigg|^{\gamma'}m_{\varepsilon}\,dx-\frac{\varepsilon^{-r}}{2}\int_{\mathbb R^n}m_{\varepsilon}\big(K_{(n-\gamma')}\ast m_{\varepsilon}\big)(x)\,dx+\int_{\mathbb R^n}V(\varepsilon x+x_{\varepsilon})m_{\varepsilon}(x)\,dx\nonumber\\
\geq &\frac{1}{2}\varepsilon^{-\gamma'}\Big[\frac{M^*}{M}-1\Big]\int_{\mathbb R^n}m_{\varepsilon}\big(K_{(n-\gamma')}\ast m_{\varepsilon}\big)(x)\,dx+\int_{\mathbb R^n}V(\varepsilon x+x_{\varepsilon})m_{\varepsilon}(x)\,dx.\label{eq6.10}
\end{align}
By the direct calculation, we obtain
\begin{align}\label{eq6.11}
\int_{\mathbb R^n}V(\varepsilon x+x_{\varepsilon})m_{\varepsilon}(x)\,dx=\varepsilon^{q_i}\int_{\mathbb R^n}\frac{V(\varepsilon x+x_{\varepsilon})}{|\varepsilon x+x_{\varepsilon}-P_i|^{q_i}}\bigg|x+\frac{x_{\varepsilon}-P_i}{\varepsilon}\bigg|^{q_i}m_{\varepsilon}(x)\,dx.
\end{align}
In light of $x_{\varepsilon}\rightarrow P_i$, then one has
\begin{align}\label{eq6.12}
\lim_{\varepsilon\rightarrow 0}\frac{V(\varepsilon x+x_{\varepsilon})}{|\varepsilon x+x_{\varepsilon}-P_i|^{q_i}}=\lim_{\varepsilon\rightarrow 0}\frac{a_i|\varepsilon x+x_{\varepsilon}-P_i|^{q_i}+O(|\varepsilon x+x_{\varepsilon}-P_i|^{q_i+1})}{|\varepsilon x+x_{\varepsilon}-P_i|^{q_i}}=a_i, \ \text{~a.e.~in~}\mathbb R^n.
\end{align}
 Now, we claim that
\begin{align}\label{629uniformlyboundxp}
\text{$q_i=q=\max\{q_1,\cdots,q_l\}$ and }~\limsup_{\varepsilon\rightarrow 0}\Big|\frac{x_{\varepsilon}-P_i}{\varepsilon}\Big| \text{~is uniformly bounded.}
\end{align}
Indeed, if $\eqref{629uniformlyboundxp}$ is not true, then we have either $q_i<q$ or up to  a subsequence, $\lim_{\varepsilon\rightarrow 0}\big|\frac{x_{\varepsilon}-P_i}{\varepsilon}\big|=+\infty$. Then by using Fatou's lemma, we conclude from \eqref{eq1.40}, \eqref{eq6.11} and \eqref{eq6.12} that
\begin{align*}
\lim_{\varepsilon\to 0}\varepsilon^{-q}\int_{\mathbb R^n}V(\varepsilon x+x_{\varepsilon})m_{\varepsilon}\,dx
=\lim_{\varepsilon\to 0}\varepsilon^{q_i-q}\int_{\mathbb R^n}\frac{V(\varepsilon x+x_{\varepsilon})}{|\varepsilon x+x_{\varepsilon}-P_i|^{q_i}}\Big|x+\frac{x_{\varepsilon}-P_i}{\varepsilon}\Big|^{q_i}m_{\varepsilon}\,dx\geq \beta\gg1
\end{align*}
for any constant $\beta\gg1$ large enough.  Combining \eqref{54equalitynormalize} with \eqref{eq6.10}, one gets
\begin{align*}
e_{\alpha^*,M}\geq& \frac{1}{2} \varepsilon^{-\gamma'}\Big[\frac{M^*}{M}-1\Big]\int_{\mathbb R^n}m_{\varepsilon}\big(K_{(n-\gamma')}\ast m_{\varepsilon}\big)(x)\,dx+\beta\varepsilon^{q}
=[1+o_{\varepsilon}(1)]\Big[\frac{M^*}{M}-1\Big]\varepsilon^{-\gamma'}+\beta\varepsilon^{q}\\
\geq &(1+o_{\varepsilon}(1))\frac{q+\gamma'}{q}\bigg(\frac{q\beta}{\gamma'}\bigg)^{\frac{\gamma'}{\gamma'+q}}\Bigg[\frac{M^*}{M}-1\Bigg]^{\frac{q}{\gamma'+q}}, ~\text{ where }~ \beta\gg1,
\end{align*}
which contradicts  Lemma \ref{lemma61}.  This completes the proof of claim \eqref{629uniformlyboundxp}.

With the help of \eqref{629uniformlyboundxp}, we obtain that $\exists y_0\in\mathbb R^n$ such that, up to a subsequence,
\begin{align*}
\lim_{\varepsilon\rightarrow 0}\frac{x_{\varepsilon}-P_i}{\varepsilon}=y_0.
\end{align*}
Then we aim to prove that $y_0$ satisfies \eqref{132thm14}, i.e. $H_i(y_0)=\inf\limits_{y\in \mathbb R^n}H_i(y)=\mu$ with $P_i\in Z_0$.  To begin with, noting $q_i=q$, we apply Fatou's lemma then conclude from \eqref{taylorexpansionV}, \eqref{eq6.3} and \eqref{eq1.40} that
\begin{align}
\lim_{\varepsilon\rightarrow 0}\varepsilon^{-q}\int_{\mathbb R^n} V(\varepsilon x+x_{\varepsilon})m_{\varepsilon}\,dx&=\lim_{\varepsilon\rightarrow 0}\int_{\mathbb R^n} \frac{V\Big(\varepsilon \big(x+\frac{x_{\varepsilon}-P_i}{\varepsilon}\big)+P_i\Big)}{|\varepsilon \big(x+\frac{x_{\varepsilon}-P_i}{\varepsilon}\big)|^q}\bigg| x+\frac{x_{\varepsilon}-P_i}{\varepsilon}\bigg|^q m_{\varepsilon}\,dx\nonumber\\
&\geq \int_{\mathbb R^n}a_i|x+y_0|^{q}m_0\,dx  \label{byusingsect61}
\geq \mu,
\end{align}
where the last two equalities hold if and only if \eqref{132thm14} holds. Thus, we have
\begin{equation}\label{eq6.16}
\begin{split}
e_{\alpha^*,M}\geq &\varepsilon^{-\gamma'}\Big[\frac{M^*}{M}-1\Big](1+o(1))+\varepsilon^{q}\mu(1+o(1))\\
\geq &(1+o(1))\frac{q+\gamma'}{q}\Big(\frac{q\mu}{\gamma'}\Big)^{\frac{\gamma'}{\gamma'+q}}\Big[\frac{M^*}{M}-1\Big]^{\frac{q}{\gamma'+q}}\\
=&(1+o(1))\frac{q+\gamma'}{q}\Big(\frac{q\mu}{\gamma'}\Big)^{\frac{\gamma'}{\gamma'+q}}\Big[1-\frac{M}{M^*}\Big]^{\frac{q}{\gamma'-q}}\Big(\frac{M^*}{M}\Big)^{\frac{q}{\gamma'+q}}\\
\geq&(1+o(1))\frac{q+\gamma'}{q}\Big(\frac{q\mu}{\gamma'}\Big)^{\frac{\gamma'}{\gamma'+q}}\Big[1-\frac{M}{M^*}\Big]^{\frac{q}{\gamma'+q}},
\end{split}
\end{equation}
where the equality holds in the second step if and only if
\begin{align}\label{eq6.17}
\varepsilon=\left[\frac{\gamma'}{q\mu}\left[1-\frac{M}{M^*}\right]\right]^{\frac{1}{\gamma'+q}}(1+o(1)).
\end{align}
Thus, combining \eqref{eq6.16} with \eqref{eq6.5}, one has all equalities in \eqref{eq6.16} hold.  It immediately follows that all "=" in \eqref{byusingsect61} also hold.  Now, we obtain 
 \eqref{131thm14} and \eqref{132thm14}, which completes the proof of Theorem. \ref{thm14preciseblowup}.
\end{proof}
Theorem \ref{thm14preciseblowup} implies that if the local expansion \eqref{taylorexpansionV} is imposed on potential $V$, then the minimizers to problem \eqref{ealphacritical-117} will concentrates at the location where $V$ is weighted flattest as $M\nearrow M^*.$  In particular, the asymptotic behavior of scaling factor $\varepsilon$ is accurately characterized.

\section{Discussion}

In this paper, we mainly investigated the existence of ground states to (\ref{goalmodel}) with critical mass exponent in the nonlocal coupling.  First of all, we analyzed the attainability of the best constant in the Gagliardo-Nirenberg type's ratio defined by (\ref{sect2-equivalence-scaling}), which corresponds the existence of ground states to the potential-free Mean-field Games system.  Next, with the aid of Gagliardo-Nirenberg type's inequality, we employ the variational approach to classify the existence of minimizers to the constrained minimization problem (\ref{ealphaM-117}).  In particular, while discussing the existence of classical solutions to (\ref{goalmodel}) under the subcritical mass, we introduced the mollifier and showed the $L^{\infty}$ of $m$ to the mollified minimization problems, in which the Hardy-Littlewood-Sobolev inequality is crucial.  Then taking the limit and applying standard elliptic regularities, we obtained the existence of classical solutions to (\ref{goalmodel}) under the subcritical mass.  Finally, with some assumptions imposed in the potential $V,$ we performed the scaling argument and blow-up analysis to derive the asymptotic behaviors of ground states to (\ref{goalmodel}) in the singular limit of $M$, where the Pohozaev identities have been intensively used for the $L^1$ convergence of $m.$

There are some interesting problems that deserve the explorations in the future.  In Section \ref{sect3-optimal}, some technical restriction on $m$ was imposed, which is the boundedness of $\int_{\mathbb R^n}m|x|^b\,dx$ for sufficiently small $b>0.$  It is an open problem to remove this condition while establishing the Gagliardo-Nirenberg type's inequality.  It is also intriguing to investigate the properties of ground states including uniqueness, symmetries, etc. to potential-free Mean-field Games systems (\ref{MFG-SS}) with the Hartree coupling and polynomial Hamiltonian.  The extension of our results into a general class of potential $V$ is a challenging problem due to the lower bounds of the value function $u.$ 

\begin{appendices}
\setcounter{equation}{0}
\renewcommand\theequation{A.\arabic{equation}}

\section{Basic proerties of Riesz potential}\label{appendixA}
This Appendix is devoted to some well-known results for the estimates involving Riesz potential, which can be found in \cite[Theorem 4.3]{LiebLoss}, \cite[Theorem 14.37]{WheedenZygmund} and \cite[Theorem 2.8]{bernardini2023ergodic}.

\begin{lemma}[Hardy Littlewood-Sobolev inequality] \label{H-L-S}
Assume that $0<\alpha<n$ and $1<r<\frac{n}{\alpha}$. Then for any $f\in L^{r}(\mathbb R^n)$, it holds
\begin{align}\label{eqHLS_1}
	\Vert K_{\alpha}* f\Vert_{L^{\frac{nr}{n-\alpha r}}(\mathbb R^n)}\leq C(n,\alpha,r) 	\Vert  f\Vert_{L^{r}(\mathbb R^n)},
\end{align}
where constant $C>0$ depending on $n$, $\alpha$ and $r.$ 

Moreover, suppose that $r,s>1$ with $\frac{1}{r}-\frac{\alpha}{n}+\frac{1}{s}=1$, $f\in L^{r}(\mathbb R^n)$ and $g\in L^{s}(\mathbb R^n)$. Then, we have there exists a sharp constant $C(n,\alpha,r)$ independent
of $f$ and $g$ such that
\begin{align}\label{eqHLS_2}
\bigg|\int_{\mathbb R^n}\int_{\mathbb R^n}\frac{ f(x) g(y)}{|x-y|^{n-\alpha}}\,dx\,dy\bigg|\leq C(n,\alpha,r)\Vert f\Vert_{L^{r}(\mathbb R^n)}\Vert g\Vert_{L^{s}(\mathbb R^n)}.
\end{align}
\end{lemma}
In particular, we find from Lemma \ref{H-L-S} that if $r=s$ in \eqref{eqHLS_2} and $f\in L^{\frac{2n}{n+\alpha}}(\mathbb R^n)$,  then there exists a sharp constant $C(n, \alpha)$ independent
of $f$ and $g$ such that
\begin{align}\label{eqHLS_3}
	\bigg|\int_{\mathbb R^n}\int_{\mathbb R^n}\frac{f(x) f(y)}{|x-y|^{n-\alpha}}\,dx\,dy\bigg|\leq C(n,\alpha)\Vert f\Vert^{2}_{L^{\frac{2n}{n+\alpha}}(\mathbb R^n)}.
\end{align}

\begin{lemma}[C.f. Theorem 2.8 in \cite{bernardini2023ergodic}] \label{HolderforRiesz}
Let $0<\alpha<n$ and $1<r\leq +\infty$ be positive constants such that $r>\frac{n}{\alpha}$ and $s\in\big[1,\frac{n}{\alpha}\big)$ . Then for any $f\in L^{r}(\mathbb R^n)\cap L^{s}(\mathbb R^n)$, we have
\begin{align}\label{InfityRiesz}
	\Vert K_{\alpha}* f\Vert_{L^{\infty}(\mathbb R^n)}\leq C_1 \Vert  f\Vert_{L^{r}(\mathbb R^n)}+C_2 \Vert  f\Vert_{L^{s}(\mathbb R^n)}.
\end{align}\label{holderRiesz}
where $C_1=C(n,\alpha,r)$ and $C_2=C(n,\alpha,s)$. Moreover, if  $0<\alpha-\frac{n}{r}<1$, we have 
\begin{align}
K_{\alpha}* f\in C^{0,\alpha-\frac{n}{r}}(\mathbb R^n).	
\end{align}
In particular, there exists constant $C:=C(n,\alpha, r)>0$ such that
\begin{align*}
\frac{\big|K_{\alpha}* f(x)-K_{\alpha}\ast f(y)\big|}{|x-y|^{\alpha-\frac{n}{r}}}\leq C \Vert  f\Vert_{L^{r}(\mathbb R^n)}, \ \  \ \ \ \forall  x \neq y.
\end{align*} 
\end{lemma}

Lemma \ref{HolderforRiesz} exhibits the $L^{\infty}$ and H\"{o}lder estimates of $K_{\alpha}* f$ under certain conditions of $f$ and $\alpha$.

\end{appendices}

\section*{Acknowledgments}

Xiaoyu Zeng is supported by NSFC (Grant Nos. 12322106, 12171379, 12271417) . Huan-Song Zhou  is supported by NSFC (Grant Nos. 11931012, 12371118) . 


\bibliographystyle{abbrv}
\bibliography{ref}

\end{document}